%% file: deltas_spheres.tex
\documentclass[reqno, 11pt, american]{amsart}
\title[High-energy eigenfunctions of point perturbations of the Laplacian on $\mathbb{S}^{2}$ and $\mathbb{S}^{3}$]{High-energy eigenfunctions of point perturbations of the Laplacian on the spheres $\mathbb{S}^2$ and $\mathbb{S}^3$}

\author[S. Verdasco]{Santiago Verdasco}
\address{M$^2$ASAI. Universidad Politécnica de Madrid, ETSI Navales, Avda. de la Memoria, 4, 28040, Madrid, Spain.}
\email{santiago.verdasco@upm.es}

\input{Repository/packages}

\input{Repository/commands}

\def \LapV{\Lap {}+ V}
\newcommand{\SDM}{semiclassical defect measure}
\newcommand{\SDMs}{semiclassical defect measures}
\newcommand{\SSDMs}{Semiclassical defect measures}

\NewDocumentCommand{\sympf}{O{} m m O{}}{{#1\{ #2 \, , #3 #1\}}_{#4}}

\begin{document}

\begin{abstract}

We study the set of Quantum Limits, and more generally, of semiclassical measures of sequences of eigenfunctions of perturbations of the Laplacian on the spheres $\S[2]$ and $\S[3]$ by point-scatterers. In the unperturbed case, it is known that the set of semiclassical measures coincides with the set of measures that are invariant under the geodesic flow; on the other hand, when the Laplacian is perturbed by a generic smooth potential, the set of semiclassical measures turns out to be strictly contained within that of invariant measures. In this article, we prove that the addition of a perturbation by a finite set of point-scatterers has a different effect: (i) all invariant measures are semiclassical measures for some sequence of eigenstates of the perturbed operator, and (ii) as soon as the set of scatterers contains a pair of antipodal points, it is possible to construct a sequence of eigenfunctions whose semiclassical measure is not invariant under the geodesic flow. We also show that this geometric condition is sharp: if the set of scatterers does not contain a pair of antipodal points, then the sets of invariant and semiclassical measures coincide.
\end{abstract}

\maketitle

\section{Introduction}

The study of high-frequency eigenfunctions of self-adjoint Schrödinger operators on a smooth compact Riemannian manifold $(M, g)$ has been the subject of extensive research over the last 60 years, with a special interest in eigenfunctions of the positive Laplace-Beltrami operator $\Lap$, or simply Laplacian (see \cite{Zelditch2017} for instance).

If $V \in L^{\infty}(M; \R)$, the Schrödinger operator $\LapV$ is self-adjoint on $L^2(M)$ with compact resolvent; thus, its spectrum $\Spec(\LapV)$ is an increasing sequence of real numbers tending to $\infty$. For each $\lambda^2 \in \Spec(\LapV)$, the corresponding normalized eigenfunctions $u_{\lambda}$ satisfy
\begin{equation} \label{Eq: eigenfunction equation}
    (\LapV) u_{\lambda} = \lambda^2 u_{\lambda} \ , \qquad \norm{u_{\lambda}}[L^2(M)] = 1 \ ,
\end{equation}
and one is interested in understanding what the structure of an eigenfunction $u_\lambda$ is when $\lambda$ is large. One way to tackle this problem is to study what kind of concentration effects in phase-space such eigenfunctions might develop. Usually, this is done by studying their Wigner distributions, which are defined for any $u \in L^2(M)$ as
\begin{equation} \label{Eq: Wigner distribution}
    W_{h}[u](a) = \ip{u}{\Op[h]{a} u}[L^2(M)] \ , \qquad a \in \CinfK(T^*M) \ ,
\end{equation}
where $\Opp_{h}$ stands for the semiclassical Weyl quantization.

Given a sequence of eigenfunctions $(u_{n})_{n \in \N}$ satisfying \eqref{Eq: eigenfunction equation} for some increasing sequence $(\lambda_n)_{n \in \N}$ tending to $\infty$, one considers the corresponding sequence $(W_{h_n}[u_n])_{n \in \N}$  of Wigner distributions scaled with the semiclassical parameters $h_n^{-1} = \lambda_n$. This sequence of distributions is bounded, and its accumulation points are called \SDMs{} of $\LapV$. It can be proved that \SDMs{} are always probability measures on $T^*M$ that are supported on $S^*M$, the unit cosphere bundle over $M$.

Characterizing the set $\mathcal{M}_{\mathrm{sc}}(\LapV)$ of all \SDM{} of $\LapV$ is a very difficult task; however, it is well-known that as soon as $V \in \Cinf(M; \R)$, any \SDM{} of $\LapV$ belongs to $\mathcal{P}_{\mathrm{inv}} (S^*M)$, the set of all probability measures $\mu$ on $T^*M$ satisfying
\begin{enumerate}
    \item $\supp \mu \subseteq S^*M$.
    \item $\mu$ is invariant under the geodesic flow $\{\phi_t\}_{t \in \R}$ on $T^*M$: $(\phi_t)_{*} \mu = \mu$ for all $t \in \R$.
\end{enumerate}
In general, the set $\mathcal{M}_{\mathrm{sc}}(\LapV)$ is strictly contained in $\mathcal{P}_{\mathrm{inv}} (S^*M)$, even when $V=0$: this is the case, for instance, when $(M,g)$ has negative sectional curvature \cite{Anantharaman2008, AnantharamanNonnenmacher2007Anosovman, Riviere2010, DyatlovJin2018, DyatlovJinNonnenmacher2022}, or $(M,g) = (\T[d], \mathrm{flat})$ \cite{Jakobson1997}. A consequence of the results in those references is that that uniform measures along closed geodesics are not semiclassical defect measures. 

On the other hand, when $(M,g) = (\S[d], \mathrm{can})$, it was shown in \cite{JakobsonZelditch1999} that $\mathcal{M}_{\mathrm{sc}}(\Lap) = \mathcal{P}_{\mathrm{inv}}(S^*\S[d])$. This also holds on Compact Rank-One Symmetric Spaces \cite{Macia2008} and their quotients \cite{AzagraMacia2010}. This is a consequence of the fact that the multiplicities of eigenspaces $\Lap$ are very large: they grow polynomially as the eigenvalue increases.

This characterization no longer holds when a smooth potential is added. When $(M,g)$ is a sphere, or more generally, a Zoll manifold, it was shown in   \cite{MaciaRiviere2016} that measures in $\mathcal{M}_{\mathrm{sc}}(\LapV)$ are invariant under the Hamiltonian flow associated with the X-Ray transform of the potential (or by the flow of a higher order ray-like transform when $V$ is in the kernel of the X-Ray transform, \cite{MaciaRiviere2019}). As a consequence, for a large class of smooth potentials the set of \SDMs{} is strictly contained in that of invariants measures, \cite[Corollary 1.10]{MaciaRiviere2016}. A by product of these results, that is somewhat close to the framework we are interested in, is that for a positive smooth potential that is supported on a neighborhood of a point $q$ in the sphere and is a function of the distance to that point, the family of uniform measures along closed geodesics through $q$ are not \SDMs{} anymore. The case of point-point perturbations we study here can be seen as a limiting case of superpositions of such potentials, as we discuss below.

In this paper we study the problem of characterization of \SDMs{} on $\S[d]$, $d = 2,3$, of a singular perturbation of $\Lap$, namely, point-perturbations of $\Lap$. These perturbations model Dirac delta potentials, $\delta_{q}$, on a given set of points $Q \subseteq \S[d]$. These are usually interpreted as the limiting model of a sequence of smooth perturbations $(\LapV_n)_{n \in \N}$ as the support of $V_n$ reduces to the set $Q$, $\cap_{n \in \N} \supp V_n = Q$. A point-perturbation of $\Lap$ is defined as a self-adjoint extension of the symmetric operator $\Lap|_{\CinfK(\S[d] \setminus Q)}$; if $d \geq 4$, $\Lap$ is the only self-adjoint extension of $\Lap|_{\CinfK(\S[d] \setminus Q)}$.

If $d = 2, 3$, and $N \coloneqq \card Q$, the family of point-perturbations of $\Lap$ on $Q$ is parametrized by the space of Lagrangian subspaces $L$ of the symplectic vector space $\C[N] \times \C[N]$ (see \cite[Section 2]{Verdasco2026}), hence we denote $\Lap_L$ for such a point perturbation of $\Lap$. The operators $\Lap_L$ have compact resolvent (see \cite[Theorem 2.15]{Verdasco2026}), thus its spectrum is a sequence of eigenvalues tending to infinity. Moreover, from a spectral point of view, the spectrum of $\Lap_L$ looks like the spectrum of a rank $\leq N$ perturbation of $\Lap$: every eigenfunction of $\Lap$ that vanishes on $Q$ is an eigenfunction of $\Lap_L$. The perturbation creates a sequence of \emph{new eigenvalues} whose corresponding eigenfunctions have a Green's function type singularity at the points of $Q$. Sequences of \emph{new eigenfunctions} are the main interest in this setting, however, on the spheres $\S[2]$ and $\S[3]$, sequences of \emph{old eigenfunctions} are of great importance too because every $\Lap_L$ has large eigenspaces formed by those.

We say that a measure $\mu$ on $T^*\S[d]$ is a \SDM{} of a sequence of point-perturbations $(\Lap_{L_n})_{n \in \N}$ if there exist sequences $(h_n)_{n \in \N} \subseteq (0, \infty)$, $h_n \to 0^+$, and $(u_n)_{n \in \N}$, such that $u_n \in D(\Lap_{L_n})$ and
\begin{equation} \label{Eq: ef eq for seq of point-pert}
    (h_n^{2}\Lap_{L_n}) u_n = u_n \ , \quad \norm{u_n}[L^2(\S[d])] = 1 \ , \qquad \forall \, n \in \N \ ,
\end{equation}
and $\mu$ is an accumulation point of the sequence of Wigner distributions $(W_{h_n}[u_n])_{n \in \N}$ as defined in \eqref{Eq: Wigner distribution}. We denote by $\mathcal{M}_{sc}(\Lap_{L_n})$ the set of all possible \SDMs{} of $(\Lap_{L_n})_{n \in \N}$. Our main result is the following.

\begin{Theorem} \label{Thm: main theorem}
    Assume $d = 2, 3$, and let $Q \subseteq \S[d]$ be finite. 
    \begin{enumerate}
        \item If $Q$ does not contain a pair of antipodal points, for any sequence of point-perturbations $(\Lap_{L_n})_{n \in \N}$ on $Q$
        \[
        \mathcal{M}_{\mathrm{sc}}(\Lap_{L_n}) = \mathcal{P}_{\mathrm{inv}} (S^*\S[d])
        \]
        \item If $Q$ contains a pair of antipodal points, there exists a sequence $(\Lap_{L_n})_{n \in \N}$ of point perturbations of $\Lap$ on $Q$ such that
        \[
        \mathcal{P}_{\mathrm{inv}} (S^*\S[d]) \subsetneq \mathcal{M}_{\mathrm{sc}}(\Lap_{L_n}) \ .
        \]
    \end{enumerate}
\end{Theorem}

\begin{Remark}
    In a related article \cite{Verdasco2026}, we investigated the invariance property of semiclassical measures of point-perturbations of the Laplacian in a general compact Riemannian manifold of dimension two or three. \cite[Theorem 1.1]{Verdasco2026} shows that, under a certain spectral condition on the set of point-scatterers, invariance always holds. Theorem \ref{Thm: main theorem} in this article shows that the aforementioned result is optimal.
\end{Remark}

Theorem \ref{Thm: main theorem} will be obtained from Theorem \ref{Thm: 2nd theorem old ef}, which shows that \emph{old eigenfunctions} (as opposed to the new ones) can be used to obtain all invariant measures as semiclassical measures, and Theorem \ref{Thm: 2nd theorem new ef}, which gives a precise construction using new eigenfunctions that can produce non-invariant semiclassical measures.

\begin{Theorem} \label{Thm: 2nd theorem old ef}
    Assume $d \geq 2$. Let $Q \subseteq \S[d]$ be a finite set, and $\mu \in \mathcal{P}_{\mathrm{inv}}(S^*\S[d])$. There exist an increasing sequence $(\ell_n)_{n \in \N} \subseteq \N$ and a sequence $(u_n)_{n \in \N} \subseteq L^2(\S[d])$ satisfying
    \begin{equation*}
        \Lap u_n = \lambda_{\ell_n}^2 u_n \ , \qquad \norm{u_n}[L^2(\S[d])] = 1 \ , \qquad u_n(q) = 0 \quad \forall \, q \in Q \ ,
    \end{equation*}
    such that for all $a \in \CinfK(T^*\S[d])$,
    \begin{equation*}
        \lim_{n \to \infty} \ip[\Big]{u_n}{ \Op[\lambda_{\ell_n}^{-1}]{a} u_n }[L^2(\S[d])] = \int_{T^*\S[d]} a(x, \xi) \, \mu(\D{x}, \D{\xi}) \ .
    \end{equation*}
    In particular, $(u_n)_{n \in \N}$ satisfies \eqref{Eq: ef eq for seq of point-pert} for any sequence of point-perturbations $(\Lap_{L_n})_{n \in \N}$. 
\end{Theorem}

In order to state our result for sequences of new eigenfunctions, we first need to introduce some notation. For any point $q \in \S[d]$, we define the measure $\nu_{q, 1/2}$ on $T^*\S[d]$ as
\begin{equation*} 
    \int_{T^*\S[d]} a(x, \xi) \nu_{q, 1/2}(\D{x}, \D{\xi}) \coloneqq \int_{S_{q}^*\S[d]} \int_{0}^{\pi} a(\phi_t(q, \xi)) \frac{\D{t}}{2\pi} \frac{\D{\xi}}{\vol(\S[d-1])} \ , \qquad \forall \ a \in \Co(T^*\S[d]) \ .
\end{equation*}
This measure $\nu_{q, 1/2}$ averages the test function $a$ along the geodesic flow from $S_{q}^*\S[d]$ to $S_{-q}^*\S[d]$. These are some properties of the measure $\nu_{q, 1/2}$:
\begin{enumerate}
    \item $\supp \nu_{q, 1/2} \subseteq S^*\S[d]$,
    \item $\nu_{q, 1/2}(S^*\S[d]) = \frac{1}{2}$,
    \item $\nu_{q, 1/2}(S^*\S[d])$ is not invariant under the geodesic flow.
\end{enumerate}
Nevertheless, $\nu_{q} \coloneqq \nu_{q, 1/2} + \nu_{-q, 1/2}$ belongs to $\mathcal{P}_{\mathrm{inv}}(S^*\S[d])$. Actually, $\nu_q$ is the \SDM{} of the sequence of normalized zonal harmonics on $q$ (see Theorem \ref{Thm: SDM of zonal harmonics}).

Assume now that $Q = \{q_1, \dots , q_K, -q_1, \dots, -q_K, p_1, \dots, p_J\}$, $N = 2K + J$. For $m \in [0,2]^N$, $m = (m_{k}^{+}, m_{k}^{-}, m_j)$ such that
\begin{equation} \label{Eq intro: condition on weights ma+ ma- mb}
    \sum_{k = 1}^{K} \frac{1}{2} (m_{k}^{+} + m_{k}^{-}) + \sum_{j = 1}^{J} m_j = 1 \ ,
\end{equation}
define
\begin{equation*} 
    \nu_{Q, m} \coloneqq \sum_{k = 1}^{K} \Big( m_{k}^{+} \nu_{q_k, 1/2} + m_{k}^{-} \nu_{-q_k, 1/2} \Big) + \sum_{j = 1}^{J} m_j \nu_{p_j} \ .
\end{equation*}
\begin{Remark}
    $\nu_{Q, m}$ is not invariant under the geodesic flow as soon as $K \geq 1$ and $m_{k_0}^{+} \neq m_{k_0}^{-}$ for some $k_0 \in \{ 1, \dots, K\}$.
\end{Remark}

\begin{Theorem} \label{Thm: 2nd theorem new ef}
    Assume $d = 2, 3$. Let $Q = \{q_1, \dots , q_K, -q_1, \dots, -q_K, p_1, \dots, p_J\} \subseteq \S[d]$, $N = 2K + J$, and $m \in [0,2]^{N}$ satisfying \eqref{Eq intro: condition on weights ma+ ma- mb}. Then there exists a sequence of point-perturbations $(\Lap_{L_n})_{n \in \N}$ such that
    \[
    \nu_{Q, m} \in \mathcal{M}_{\mathrm{sc}}(\Lap_{L_n}) \ .
    \]
\end{Theorem}

\SSDMs for point-perturbations on tori have been extensively studied in \cite{RudnickUeberschaer2012, KurlbergUeberschaer2014, KurlbergRosenzweig2017, KurlbergUeberschaer2017, KurlbergLesterRosenzweig2023, Yesha2013, Yesha2015, Yesha2018}. In those references, it is shown that, for a single scatterer on the squared torii $\T[2]$ and $\T[3]$, there exist density-one subsequences of new eigenfunctions that equidistribute on phase space and, at the same time, there exist zero-density subsequences of new eigenfunctions that equidistribute on the position variable but scar on the momentum variable, a phenomenon known as \emph{superscarring}. The case of two-dimensional irrational tori for a single point-perturbation was addressed in those articles too: it was proved that there exist density-one subsequences for which superscarring takes place. The case of two scatterers on the squared torus $\T[2]$ was also studied: under some diophantine condition of the scatterers, there exists a density-one subsequence of new eigenfunctions that equidistributes in the configuration variable.

As far as we know, this is the first instance of a perturbed Schrödinger operator whose sequence of eigenfunctions produce \SDMs{} which are non-invariant under the Hamiltonian flow induced by the unperturbed Schrödinger operator. We must note that this kind of phenomenon is very particular to the sphere, as we extensively make use of the fact that eigenfuntions are explicit. At this point, it is not clear whether or not this phenomenon may be reproduced on Zoll manifolds following a strategy similar to the one presented here.

\subsection*{Outline of the paper}
\label{subSec: Introduction. Outline}

In Section \ref{Sec: Point-perturbations of Laplacian on the spheres} we briefly present the theory of point-perturbations of $\Lap$. In Section \ref{Sec: high-energy zonal harmonics} we study zonal harmonics in the high-energy regime, characterizing its set of \SDMs{}. In Section \ref{Sec: Quasimode Green function} we introduce some auxiliary parameters to compute fine estimates of the norms of new eigenfunctions that are fundamental to later finding a sequence of quasimodes for every sequence of new eigenfunctions. In Section \ref{Sec: SDM Green functions} we use these quasimodes to characterize the set of \SDMs{} of sequences of new eigenfunctions. In Section \ref{Sec: proofs of main theorem} we prove the aforementioned theorems, namely Theorem \ref{Thm: 2nd theorem new ef},  Theorem \ref{Thm: 2nd theorem old ef}, and Theorem \ref{Thm: main theorem}, in this order.

\subsection*{Acknowledgments}

The author appreciates Jared Wunsch and Yuzhou Joey Zou's hospitality during his stay at Northwestern University, and warmly thanks them for the insightful conversations on this and related problems that led to clever ideas. The author wants to thank Víctor Arnaiz for the fruitful conversations held at the Institut de Mathématiques de Bordeaux that opened up this problem. Lastly, the author is indebted to Fabricio Macià as this problem was suggested and supervised by him.
This research has been supported by grants PID2021-124195NB-C31 and PID2024-158664NB-C21 from Agencia Estatal de Investigación (Spain), and by grant VPREDUPM22 from Programa Propio UPM.

\section{Point-pertubations of Laplacian on the spheres}
\label{Sec: Point-perturbations of Laplacian on the spheres}

The theory of self-adjoint extensions of a symmetric operator is a well-established theory, and several equivalent descriptions of the family of seld-adjoint extensions are known; for example, we refer the reader to \cite[Chapter X]{ReedSimonII} for a description via unitary operators. Nevertheless, we opt for a description via Lagrangian subspaces (see \cite[Section 3]{HillairetKokotov2012}). In particular, we follow the notation and definitions described in \cite[Section 2]{Verdasco2026} and adapt them to our particular case $M = \S[d]$, $d = 2, 3$.

Fix a finite set of points $Q \subseteq \S[d]$, $N \coloneqq \card Q$, define $A$ as the closure of the symmetric operator $\Lap|_{\CinfK(\S[d] \setminus Q)}$ in $L^2(\S[d])$. A point-perturbation of $\Lap$ in $\S[d]$ is a self-adjoint extension of $A$. Each self-adjoint extension corresponds to a choice of boundary conditions on the set $Q$; in particular, the choice of continuous boundary conditions on $Q$ corresponds to $\Lap$, the trivial self-adjoint extension.

The family of all self-adjoint extensions of $A$ is isomorphic to the Lagrangian-Grassmanian of $\C[N] \times \C[N]$ with its standard symplectic form
\[
\Omega( (A^+, A^-) , (B^+, B^-)) = \ip{A^{+}}{B^{-}}[\C[N]] - \ip{A^{-}}{B^{+}}[\C[N]] \ .
\]
Moreover, there exists a symplectomorphic coordinate map from the set of all possible (self-adjoint) boundary conditions on $Q$ to $(\C[N] \times \C[N], \Omega)$ \cite[Lemma 2.2]{Verdasco2026}.

For a Lagrangian subspace $L$ of $(\C[N] \times \C[N], \Omega)$, define $\Lap_L$ as the corresponding self-adjoint perturbation of $\Lap$. Let $Z_{\ell}^{q}$ denote the zonal harmonic of energy $\lambda_{\ell}^2 \in \Spec(\Lap)$, that is, the unique eigenfunction in $\ker(\Lap {}- \lambda_{\ell}^2)$ such that
\[
\ip{Z_{\ell}^{q}}{u} = u(q) \ , \qquad \forall \ u \in \ker(\Lap {}- \lambda_{\ell}^2) \ .
\]
In addition, for $h > 0$ and $\beta \in \C[N]$ such that
\begin{equation} \label{Eq: condition beta-h}
    \sum_{q \in Q} \beta_{q} \delta_{q} = 0 \qquad \text{on} \qquad \ker(h^2 \Lap {}- 1)
\end{equation}
define $G_{h}^{Q, \beta}$ as follows. If $h^{-2} \notin \Spec(\Lap)$,
\begin{equation} \label{Eq: def GhQbeta good h}
    G_{h}^{Q, \beta} \coloneqq \sum_{ \ell \in \N } \frac{1}{\lambda_{\ell}^2 - h^{-2}} \sum_{q \in Q} \beta_{q} Z_{\ell}^{q} \ .
\end{equation}
but if $h^{-2} = \lambda_{\ell_h}^2$ for some $\ell_h \in \N$,
\begin{equation} \label{Eq: def GhQbeta bad h}
    G_{h}^{Q, \beta} \coloneqq \sum_{ \substack{\ell \in \N \\ \ell \neq \ell_h} } \frac{1}{\lambda_{\ell}^2 - \lambda_{\ell_h}^2} \sum_{q \in Q} \beta_{q} Z_{\ell}^{q} \ .
\end{equation}
Observe that for such $h > 0$ and $\beta \in \C[N]$, $G_{h}^{Q, \beta}$ is a linear combination of Green's function for $\Lap$ at energy $h^{-2}$:
\begin{equation} \label{Eq: GF property of new EF}
    \ip{G_{h}^{Q, \beta}}{(\Lap {}- h^{-2}) u}[L^2(\S[d])] = \sum_{q \in Q} \beta_q u(q) \ , \qquad \forall \ u \in H^2(\S[d]) \ .
\end{equation}
Note that $G_{h}^{Q, \beta}$ is always real-valued because $Z_{\ell}^{q}$ is real-valued for any $\ell \in \N$ (see \cite[Lemma 2.3]{Verdasco2026}).

\begin{Remark}
    Throughout the paper, if some $h > 0$ is under consideration, we will always assume that $\beta \in \C[N]$ is such \eqref{Eq: condition beta-h} holds for that specific $h$. For instance, if we have sequences $(h_n)_{n \in \N} \subseteq (0,\infty)$ and $(\beta_{n})_{n \in \N}$, we will assume that $h_n$ and $\beta_{n}$ together satisfy \eqref{Eq: condition beta-h}.
\end{Remark}

From \cite[Theorem 2.15]{Verdasco2026}, we know that $\Lap_L$ has compact resolvent, and $u_h \in L^2(\S[d])$ satisfies $h^2\Lap_L u_h = u_h$ for some $h > 0$ if and only if $u_h \in \dom(\Lap_L)$ and $u_h$ is of the form
\begin{equation} \label{Eq: eigenfunctions of LapL}
    u_h = w_h + G_{h}^{Q, \beta} \ ,
\end{equation}
for some $w_h \in \ker(h^2\Lap {}- 1)$ and $\beta \in \C[N]$ satisfying \eqref{Eq: condition beta-h}. When $h^{-2} \in \Spec(\Lap)$ and $u_h = w_h$, then $u_h$ is called an \emph{old eigenfunction}; these are true eigenfunctions of $\Lap$. We will study them in Section \ref{Sec: SDM old eigenfunctions}, where we prove Theorem \ref{Thm: 2nd theorem old ef}. When $u_h = G_{h}^{Q, \beta}$ for some $\beta \in \C[N]$ such that \eqref{Eq: condition beta-h} is satisfied, $u_h$ is called a \emph{new eigenfunction}. We will study new eigenfunctions through Section \ref{Sec: Quasimode Green function} and Section \ref{Sec: SDM Green functions}, proving Theorem \ref{Thm: 2nd theorem new ef}. In particular, if $Q = \{q\}$, then $u_h \in D(\Lap_L)$ in \eqref{Eq: eigenfunctions of LapL} is either an old eigenfunction vanishing on $q$ if $h^{-2} \in \Spec(\Lap)$, or $u_h = \beta G_{h}^{q}$ for some $\beta \in \C$ if $h^{-2} \notin \Spec(\Lap)$.

On the other hand, for every $h > 0$ and $\beta \in \C[N]$ satisfying \eqref{Eq: condition beta-h} there exists some point perturbation $\Lap_L$ such that $G_{h}^{Q, \beta} \in D(\Lap_L)$ (see \cite[Lemma 2.16]{Verdasco2026}).

\section{High-energy zonal harmonics}
\label{Sec: high-energy zonal harmonics}

A key feature new eigenfunctions of $\Lap_L$ on the spheres $\S[d]$ have is that they are (infinite) linear combinations of zonal harmonics $Z_{\ell}^{q}$ on the different points of $Q$. The rich structure of the spectrum of $\Lap$ on $\S[d]$ allows us to use the sequence of zonal harmonics $(Z_{\ell}^{q})_{\ell \in \N}$ as skeleton for any sequence of new eigenfunctions. Moreover, thanks to the very nice symmetric properties of the sequence $(Z_{\ell}^{q})_{\ell \in \N}$, one may prove this sequence has a unique \SDM{}; this later allows us to compute \SDMs{} for sequences of new eigenfunctions.

Denote for $\ell \in \N$ and $q \in \S[d]$,
\begin{equation}
    z_{\ell}^{q} \coloneqq \frac{1}{ \norm{Z_{\ell}^{q}}[L^2(\S[d])] } Z_{\ell}^{q} \ .
\end{equation}

\begin{Theorem} \label{Thm: SDM of zonal harmonics}
    Let $d \in \N$. Let $(\ell_n)_{n \in \N} \subseteq \N$ be any increasing sequence of natural numbers, and let $(h_n)_{n \in \N} \subseteq (0, \infty)$ be such that $\lim_{\ell \to \infty} h_n \lambda_{\ell_n} = 1$, then
    \begin{equation} \label{Eq: SDM of zonal harmonics}
        \lim_{n \to \infty} \ip*{z_{\ell_n}^{q}}{\Op[h_n]{a} z_{\ell_n}^{q}} = \int_{T^*\S[d]} a(x, \xi) \nu_{q}(\D{x}, \D{\xi}) \qquad \forall \ a \in \CinfK(T^*\S[d]) \ ,
    \end{equation}
    where $\nu_{q}$ is uniform probability measure along the geodesic flow-out from $S_{q}^{*} \S[d]$:
    \begin{equation} \label{Eq: def of nuq}
         \int_{T^*\S[d]} a(x, \xi) \nu_{q}(\D{x}, \D{\xi}) \coloneqq \int_{S_{q}^*\S[d]} \int_{0}^{2\pi} a(\phi_t(q, \xi)) \frac{\D{t}}{2\pi} \frac{\D{\xi}}{\vol(\S[d-1])} \ , \qquad \forall \, a \in \CinfK(T^*\S[d]) \ .
    \end{equation}
\end{Theorem}

\begin{proof}
    Up to extraction of a subsequence, we may assume there exists a unique measure $\mu$ such that
    \[
    \lim_{n \to \infty} \ip*{z_{\ell_n}^{q}}{\Op[h_n]{a} z_{\ell_n}^{q}} = \int_{T^*\S[d]} a \, \mu \ , \qquad \forall \ a \in \CinfK(T^*\S[d]) \ .
    \]
    We are going to show that $\mu = \nu_{q}$.

    \textbf{I.} Since the sequence $(z_{\ell_n}^{q})_{n \in \N}$ satisfies the following semiclassical pde,
    \[
    ((h_n)^2 \Lap {}- 1) z_{\ell_n}^{q} = \littleo_{L^2}(h_n) \ ,
    \]
    we know, that for $H(x, \xi) \coloneqq \abs{\xi}_x^2$,
    \[
    \int_{T^*\S[d]} a (H - 1) \, \mu = 0 \ , \qquad \int_{T^*\S[d]} \sympf{ a }{ H } \, \mu = 0 \ .
    \]
    From these identities, we read that
    \begin{equation} \label{Eq: prop of zonalharm SDM from Lap}
        \supp \mu \subseteq S^*\S[d] \ , \qquad (\phi_t)_{*} \mu = \mu
    \end{equation}
    where $\phi_t$ denotes the geodesic flow on $T^*\S[d]$.

    Similarly, since the sequence $(z_{\ell_n}^{q})_{n \in \N}$ satisfies (see Proposition \ref{Prop: properties zonal harm})
    \[
    \frac{h_n}{i} \hat{L} z_{\ell_n} = 0
    \]
    for any Killing vector field $\hat{L}$ such that $\hat{L}(q) = 0 \in T_{q}\S[d]$, we infer, that for $L(x, \xi) \coloneqq \xi(\hat{L}(x))$,
    \[
    \int_{T^*\S[d]} a L \, \mu = 0 \ , \qquad \int_{T^*\S[d]} \sympf{ a }{ L } \, \mu = 0 \ ,
    \]
    since $L$ is the symbol of the semiclassical pseudodifferential operator $\frac{h_n}{i} \hat{L}$. We define $\mathcal{L}$ as the subset of smooth functions on $T^*\S[d]$ that comprise all functions $L(x, \xi) \coloneqq \xi(\hat{L}(x))$ for any Killing vector field $\hat{L}$ on $\S[d]$ such that $\hat{L}(q) = 0$. From these identities, we read that
    \begin{equation} \label{Eq: prop of zonalharm SDM from Kfield}
        \supp \mu \subseteq L^{-1}(0) \ , \qquad (\psi_t^{L})_{*} \mu = \mu \ , \qquad \forall \, L \in \mathcal{L} \ ,
    \end{equation}
    where $\psi_t^{L}$ denotes the geodesic flow on $T^*\S[d]$.

    \textbf{II.} Now we show that the only probability measure $\mu$ that satisfy both \eqref{Eq: prop of zonalharm SDM from Lap} and \eqref{Eq: prop of zonalharm SDM from Kfield} is $\nu_q$, that we defined in \eqref{Eq: def of nuq}.

    We first note that the commutation of the operators $\Lap$ and $\hat{L}$ imply that their symbols, $H$ and $L$, Poisson-commute. Since
    \[
    0 = [h^2 \Lap , \tfrac{h}{i} \hat{L}] = \frac{h}{i} \Op[h]{\sympf{H}{L} } + h\Psi_h^1 \ ,
    \]
    and $\frac{h}{i} \Op[h]{\sympf{H}{L} } \in h \Psi_{h}^{2}$, then $\sympf{H}{L} = 0$. This means that the quantities $L \in \mathcal{L}$ are preserved along the geodesic flow.

    \begin{Claim}
        The measure $\mu$ is supported inside the unit-speed geodesic flow-out from $q$, $\supp \mu \subseteq \mathcal{F}_{q}$,
        \begin{equation} \label{Eq: flow-out from q set def}
            \mathcal{F}_{q} = \{ \phi_t(q, \xi) \colon \ \xi \in S_{q}^{*}\S[d] \, , \ t \in \R \} \ .
        \end{equation}
    \end{Claim}

    \begin{proof}[Proof of claim]
        Through stereographic projection from the antipodal point of $q$, geodesics through $q \in \S[d]$ are mapped to straight lines through $0 \in \R[d]$, and Killing fields $\hat{L}$ are transformed into vector fields on $\R[d]$ that are tangent to every sphere $r\S[d-1]$, $r \geq 0$. Moreover, in these coordinates, for every $y \in \R[d] \setminus \{0\}$, $\{ \hat{L}(y) \colon L \in \mathcal{L} = T_{y}\S[d-1]$. This implies that the set of coordinate points $(y, \eta) \in T^*\R[d]$ such that $L(y, \eta) = 0$ for all $L \in \mathcal{L}$ are those such that $T_{y}\S[d - 1] \subseteq \ker(\eta)$. These points on $T^*\R[d]$, when lifted to the cotangent bundle $T^*\S[d]$, correspond to the geodesic flow-out from $q$. Since we know that $\supp \mu \subseteq S^*\S[d]$, we conclude that $\supp \mu \subseteq \mathcal{F}_{q}$.
    \end{proof}
    
    Now we show that $\mu$ must be "uniform" on $\mathcal{F}_{q}$. The set $\mathcal{F}_{q}$ is a $d$-dimensional manifold; it is diffeomorphic to $\T \times \S[d-1]$ indeed. A diffeomorphism is given by
    \[
    F \colon \T \times S_{q}^*\S[d] \to \mathcal{F}_{q} \ , \qquad F(t, \xi) = \phi_t(q, \xi) \ ,
    \]
    where we recall that $\phi_t$ is the geodesic flow on $T^*\S[d]$. The Lie group $\T \times SO(d)$ acts smoothly on $\mathcal{F}_{q}$:
    \[
    (t, g) \cdot \phi_{t'}(q, \xi) = \phi_{t' + t}(q, g\xi) \ .
    \]
    Moreover, for every $L \in \mathcal{L}$, $g^{L} \coloneqq \psi_{1}^{L}|_{S_{q}^* \S[d]}$ can be thought as an element of $SO(d)$. Since $g^{L}$ are rotations on $S_{q}^*\S[d]$---composition of two reflections on $S_{q}^*\S[d]$---, each $g \in SO(d)$ can be expressed as a finite composition of elements of the form $g^{L}$, $g = g^{L_N} \dots g^{L_1}$, because every isometry of $T_{q}^{*}\S[d] \equiv \R[d]$ can be decomposed as a composition of at most $d$ orthogonal reflections (hence $N \leq \frac{d}{2}$). Since flows $\phi_t$ and $\psi_s^L$ commute, the action from above could be rewritten as
    \[
    (t,g) \cdot (x, \xi) = \phi_t \circ \psi_{1}^{L_{N}} \circ \dots \circ \psi_{1}^{L_{1}} (x, \xi) \ .
    \]
    Since $(\phi_t)_{*} \mu = \mu$ and $(\psi_s^{L})_{*} \mu = \mu$, then $\mu$ is invariant under the action of the whole Lie group:
    \begin{align*}
        \int_{\mathcal{F}_{q}} a \big( (t, g) \cdot (x, \xi)\big) \ \mu(\D{x}, \D{\xi}) & = \int_{\mathcal{F}_{q}} a \big( \phi_t \circ \psi_{1}^{L_{N}} \circ \dots \circ \psi_{1}^{L_{1}} (x, \xi) \big) \ \mu(\D{x}, \D{\xi}) = \\
        & = \int_{\mathcal{F}_{q}} a(x, \xi) \ \mu(\D{x}, \D{\xi}) \ .
    \end{align*}
    Since the probability Radon measure $\mu$ is invariant under the (continuous) action of $\T \times SO(d)$, it induces a probability Borel regular measure $\Haarmeasure$ on the group $\T \times SO(d)$ as follows
    \[
    \Haarmeasure(E) \coloneqq \mu(\{ (t, g) \cdot z_0 \colon \, (t, g) \in E \}) \ , \qquad \forall \, E \subseteq \T \times SO(d) \ \text{closed}
    \]
    for any $z_0 \in \mathcal{F}_{q}$. $\Haarmeasure$ is invariant under the group operation, thus it must be the probability Haar measure on $\T \times SO(d)$: $\Haarmeasure = \Haarmeasure_{\T} \otimes \Haarmeasure_{SO(d)}$, the tensor product of the respective probability Haar measures. Therefore, fixing $z_0 = (q, \xi_0)$, for a closed subset $K \subseteq \T \times SO(d)$ of the form $K = F([a, b] \times C)$, 
    \[
    \mu(K) = \Haarmeasure_{\T}([a,b]) \, \Haarmeasure_{SO(d)}(\{ g \in SO(d) \colon \, g \cdot \xi \in C \}) = \frac{b-a}{2\pi} \frac{\sigma^{d-1}(C)}{\vol(\S[d-1])} \ .
    \]
    Here we use the characterization of the uniform spherical measure on $\S[d-1]$, $\sigma^{d-1}$, via the Haar measure on $SO(d)$. This last statement is equivalent to
    \[
    \int_{\mathcal{F}_{q}} a(x, \xi) \, \mu(\D{x}, \D{\xi}) = \int_{S_{q}^*\S[d]} \int_{0}^{2\pi} a(\phi_t(q, \xi)) \ \frac{\D{t}}{2\pi} \frac{\D{\xi}}{\vol(\S[d-1])} \qquad \forall \, a \in \Cont(\mathcal{F}_{q}) \ .
    \]
    This concludes the result.
\end{proof}

The following properties of the measures $\nu_{q}$ are easy to check.
\begin{Prop} \label{Prop: nuq properties}
    Let $d \in \N$. Let $q, p \in \S[d]$, and denote by $-q$ the antipodal point to $q$. The following holds:
    \begin{enumerate}
        \item $\nu_{-q} = \nu_{q}$.
        \item If $p \neq q, -q$, then $\nu_{p}$ and $\nu_{q}$ are mutually singular, $\nu_{p} \perp \nu_{q}$.
    \end{enumerate}
\end{Prop}

The fact that the sequences $(z_{\ell}^{q})_{\ell \in \N}$ and $(z_{\ell}^{p})_{\ell \in \N}$ have mutually singular \SDM{} implies the following algebraic flavored result.

\begin{Prop} \label{Prop: linindep of zonal harm}
    Let $d \in \N$. Let $Q \subseteq \S[d]$ be a fixed finite set of points, $N \coloneqq \card Q$. The following are equivalent:
    \begin{enumerate}
        \item $Q$ contains a pair of antipodal points, $q$ and $-q$.
        \item There exists an increasing sequence $(\ell_n)_{n \in \N}$ such that $\{ z_{\ell_n}^{q}\}_{q \in Q}$ are linearly dependent for all $n \in \N$.
    \end{enumerate}
\end{Prop}

\begin{proof}
    If $Q$ contains a pair of antipodal points, say $q$ and $-q$, then $\{ z_{\ell}^{q}\}_{q \in Q}$ are linearly dependent because $z_{\ell}^{-q} = (-1)^{\ell} z_{\ell}^{q}$ for all $\ell \in \N$ (see Proposition \ref{Prop: properties zonal harm}(6)).

    Assume now that $Q$ does not comprise a pair of antipodal points. Fix an arbitrary increasing sequence $(\ell_n)_{n \in \N} \subseteq \N$, and let $(\beta_{n})_{n \in \N} \subseteq \C[N]$ such that
    \[
    z_{\ell_n}^{Q, \beta_n} \coloneqq \sum_{q \in Q} \beta_{n, q} z_{\ell_n}^{q} = 0 \qquad \forall \, n \in \N \ .
    \]
    Assume that $\beta_{n} \neq 0$ for all $n \in \N$. Due to the homogeneity of the above identity, we may further assume that $\abs{\beta_n}^2 = 1$ for all $n \in \N$. From this we read that
    \begin{equation} \label{Eq: Prop li zonal harm aux}
        0 = \norm*{z_{\ell_n}^{Q, \beta_n}}[L^2(\S[d])]^2 = \sum_{q \in Q} \bigg( \abs{\beta_{n, q}}^2 + \sum_{ \substack{p \in Q \\ p \neq q} } \conj{\beta_{n, p}} \beta_{n, q} \ip{z_{\ell_n}^{p} }{ z_{\ell_n}^{q} }[L^2(\S[d])] \bigg) \ .
    \end{equation}
    Thanks to Theorem \ref{Thm: SDM of zonal harmonics} and Proposition \ref{Prop: nuq properties}, we may apply Lemma \ref{Lemma: orthogonality for sequence from SDM} to the sequences $(z_{\ell_n}^{q})_{n \in \N}$ and $(z_{\ell_n}^{p})_{n \in \N}$ and thus obtain
    \[
    \lim_{n \to \infty} \ip{z_{\ell_n}^{p} }{ z_{\ell_n}^{q} }[L^2(\S[d])] = 0 \qquad \forall \, p \neq q \ .
    \]
    Taking limits $n \to \infty$ in \eqref{Eq: Prop li zonal harm aux}, we get $0 = 1$, a contradiction!
\end{proof}

\begin{Lemma} \label{Lemma: orthogonality for sequence from SDM}
    Let $(u_n)_n$ and $(v_n)_n$ be normalized sequences in $L^2(M)$. Assume that for a fixed sequence of semiclassical parameters $(h_n)_n$, $(u_n)_n$ and $(v_n)_n$ have unique semiclassical defect measure $\mu$ and $\nu$, respectively. Assume the following.
    \begin{enumerate}
        \item $\mu \perp \nu$.
        \item $\mu$ is a probability measure with compact support on $T^*M$.
    \end{enumerate}
    Then
    \[
    \lim_{n \to \infty} \ip{v_n}{u_n} = 0 \ .
    \]
\end{Lemma}

\section{Quasimodes for new eigenfunctions}
\label{Sec: Quasimode Green function}

New eigenfunctions $u_h$ of a point-perturbation of $\Lap$ on $Q$ with positive eigenvalue are linear combinations of Green's functions of $\Lap$ on the points $q \in Q$ at energy $h^{-2}$: for a complex vector $\beta \in \C[N]$, $N \coloneqq \card Q$, such that $\sum_{q \in Q} \beta_{q} \delta_{q} = 0$ on $\ker(h^2 \Lap {}- 1)$, we denote
\[
G_{h}^{Q, \beta} \coloneqq \sum_{q \in Q} \beta_{q} G_{h}^{q} = \sum_{\ell \in \N} \frac{1}{\lambda_{\ell}^2 - h^{-2}} \sum_{q \in Q} \beta_q Z_{\ell}^{q} \ .
\]
If $Q$ is a singleton, $Q = \{q\}$, we simply write $g_{h}^{q} \coloneqq (\norm{G_{h}^{q}}[L^2(\S[d])])^{-1} G_{h}^{q}$.

It turns out that the set of \SDMs{} of the normalized sequence $(g_{h_n}^{Q, \beta_n})_{n \in \N}$,
\begin{equation} \label{Eq: def ghnQbeta}
    g_{h_n}^{Q, \beta_n} \coloneqq \frac{1}{ \norm*{G_{h_n}^{Q, \beta_n}}[L^2(\S[d])] } G_{h_n}^{Q, \beta_{n}} \ ,
\end{equation}
can be inferred from the case $\card Q = 2$; we defer this reasoning to Section \ref{Sec: SDM Green functions} and focus on finding quasimodes for $(g_{h_n}^{Q, \beta_n})_{n \in \N}$ in the case $Q = \{ q, -q\}$.

With the aim of computing these quasimodes, we introduce some parameters that will help us distinguish different high-energy regimes. For a decreasing sequence $(h_n)_{n \in \N} \subseteq (0, 1)$, let $\ell_n \in \N$ and $\sigma_n \in [0,1)$\footnote{One should think that $\sigma_n \in \R / \Z$, because if $\sigma_n$ were to be $1$, one would choose $\ell_n + 1$ and $\sigma_n = 0$ instead. Moreover, sequences $(h_n)_{n \in \N}$ such that $\sigma_n \to \sigma \in \{0, 1\}$ will exhibit the same behavior depending on the asymptotics of $\beta_{n,q} + (-1)^{\sigma} \rho \beta_{n,-q}$.} such that
\begin{equation} \label{Eq: def elln and sigman}
    h_n^{-2} = (\ell_n + \sigma_n) (\ell_n + \sigma_n + d-1) \ .
\end{equation}

The following asymptotics are fundamental.
\begin{Prop} \label{Prop: hn elln approximation}
    Under the notation introduced above, the following estimates hold for all $h_n \in (0,1)$,
    \begin{equation*}
        \abs*{(h_n)^{-1} - [\ell_n + \sigma_n + \tfrac{d - 1}{2}] } \leq h_n \ , \qquad \abs*{h_n - \frac{1}{\ell_n + \sigma_n + \frac{d-1}{2}} } \leq h_n^3 \ .
    \end{equation*}
\end{Prop}

\begin{proof}
    The first inequality follows from a direct computation:
    \[
    \abs*{(h_n)^{-1} - [\ell_n + \sigma_n + \tfrac{d - 1}{2}] } = \frac{\abs*{(h_n)^{-2} - [\ell_n + \sigma_n + \frac{d - 1}{2}]^2 } }{ (h_n)^{-1} + [\ell_n + \sigma_n + \tfrac{d - 1}{2}] } \leq \frac{ \frac{(d-1)^2}{4} }{ h_n^{-1} } \leq h_n \ .
    \]
    The second inequality is a consequence of the first. We see that
    \[
    \ell_n + \sigma_n + \tfrac{d-1}{2} \geq h_n^{-1} \ ,
    \]
    therefore
    \[
    \abs*{h_n - \frac{1}{\ell_n + \sigma_n + \frac{d-1}{2}} } = \frac{ \abs*{(h_n)^{-1} - [\ell_n + \sigma_n + \tfrac{d - 1}{2}] } }{ h_n^{-1} [\ell_n + \sigma_n + \tfrac{d - 1}{2}] } \leq \frac{h_n}{h_n^{-2}} = h_n^3 \ . \qedhere
    \]
\end{proof}

In this case that $Q = \{q, - q\}$, using that $Z_{\ell}^{-q} = (-1)^{\ell} Z_{\ell}^{q}$ for all $\ell \in \N$ (see Proposition \ref{Prop: properties zonal harm}), we have for any sequence $(\beta_{n})_{n \in \N} \subseteq \C[2]$,
\begin{equation} \label{Eq: sum of antipodal Green fun}
    G_{h_n}^{Q, \beta_n} = \beta_{n,q} G_{h_n}^{q} + \beta_{n, -q} G_{h_n}^{-q} = \sum_{\ell \in \N} \frac{\beta_{n,q} + (-1)^{\ell} \beta_{n,-q}}{ \lambda_{\ell}^2 - h_n^{-2} } Z_{\ell_n}^{q} \ .
\end{equation}
Since $G_{h}^{Q, z \beta} = z G_{h}^{Q, \beta}$ for all $z \in \C$, we may assume, without lost of generality, that the sequence $(\beta_{n})_{n \in \N}$ is contained in the complex sphere $\{ \beta \in \C[2] \colon \abs{\beta} = 1\}$.

Up to extraction of a subsequence, we may further assume that the sequences $(\sigma_n)_{n \in \N}$ and $(\beta_n)_{n \in \N}$ converge to some $\sigma \in [0,1]$ and $\beta \in \C[2]$ respectively, and that $\ell_n$ is even for all $n \in \N$, or that $\ell_n$ is odd for all $n \in \N$; in this case we set $\rho \coloneqq (-1)^{\ell_n}$. Four different scenarios are possible:
\begin{enumerate}
    \item $h_n^{-2} \notin \Spec(\Lap)$ for all $n \in \N$, $\sigma_n \to \sigma \in (0,1)$.
    \item $h_n^{-2} \notin \Spec(\Lap)$ for all $n \in \N$, $\sigma_n \to \sigma \in \{0, 1\}$ and $\lim_{n \to \infty} \frac{\abs{ \beta_{n,q} + (-1)^{\sigma} \rho \beta_{n, -q} } }{ \abs{\sigma - \sigma_n} } = \infty$.
    \item $h_n^{-2} \notin \Spec(\Lap)$ for all $n \in \N$, and $\sigma_n \to \sigma \in \{0,1\}$, but $\lim_{n \to \infty} \frac{\beta_{n,q} + (-1)^{\sigma} \rho \beta_{n,-q}}{ \sigma - \sigma_n } \in \C$.
    \item $h_n^{-2} = \lambda_{\ell_n}$ and $\beta_{n,-q} + \rho \beta_{n,q} = 0$ for all $n \in \N$.
\end{enumerate}
Scenarios 1 and 3 are sources of non-invariant phenomena. Under the hypothesis of scenario 2 one may show that (see Theorem \ref{Thm: gh quasimode scenario 2})
\begin{equation*}
    \norm*{g_{h_n}^{Q, \beta_n} - z_{\ell_n + \sigma}^{q}}[L^2(M)] \leq C \bigg[ h_n + \frac{ \abs{\sigma - \sigma_n} }{\abs{\beta_{n,q} + (-1)^{\sigma} \rho \beta_{n,-q}} } \bigg] \qquad \text{for large enough $n$,}
\end{equation*}
and thanks to Theorem \ref{Thm: SDM of zonal harmonics} one concludes
\[
\lim_{n \to \infty} \ip*{ g_{h_n}^{Q, \beta_n}}{\Op[h_n]{a} g_{h_n}^{Q, \beta_n}}[L^2(\S[d])] = \int_{T^*\S[d]} a \, \nu_q \qquad \forall \, a \in \CinfK(T^*\S[d]) \ .
\]
Meanwhile, a sequence that satisfy the hypothesis of scenario 4 can be approximated by a sequences that satisfy the hypothesis of scenario 3 with the added condition $\lim_{n \to \infty} \frac{ \beta_{n,q} + \beta_{n, -q} }{ \sigma - \sigma_n } = 0$, which cancels the non-invariant ingredient of scenario 3. Therefore, sequences of linear combinations of Green's functions satisfying either scenario 2 or scenario 4 do not give rise to non-invariant phenomena; consequently, we give detailed proofs of scenarios 1 and 3, but we just sketch proofs of scenarios 2 and 4.

We introduce a last bit of notation. For any set $A \subseteq \N$, let $\Pi_{A}$ be the spectral eigenprojector onto the subspace $\oplus_{\ell \in A} \ker(\Lap {}- \lambda_{\ell}^2)$, and for $\ell \in \N$, $r \geq 0$, let $I(\ell, r) \coloneqq [\ell - r, \ell + r] \cap \N$. For instance, for $h_n^{-2} \notin \Spec(\Lap)$ and $\Upsilon \in \N$,
\begin{equation} \label{Eq: example cutoff Upsilon GF}
    \Pi_{I(\ell_n, \Upsilon)} G_{h_n}^{Q, \beta} = \sum_{\abs{k} \leq \Upsilon} \frac{\beta_{n,q} + (-1)^{\ell_n  + k} \beta_{n, -q} }{ \lambda_{\ell_n + k}^2 - h_n^{-2} } Z_{\ell_n + k}^{q} \ .
\end{equation}

\begin{Theorem} \label{Thm: Gh tails L2 norm estimates}
    Let $(h_n)_{n \in \N} \subseteq (0,1)$ and $(\beta_{n})_{n \in \N} \subseteq \C[2]$, and for every $n \in \N$, let $\ell_n \in \N$ and $\sigma_n \in [0,1)$ as defined in \eqref{Eq: def elln and sigman}. Assume that $h_n \to 0^+$ and that $(\beta_{n})_{n \in \N}$ is bounded. There exists $C > 1$ such that for all integer $\Upsilon \geq 1$ and all $h_n^{-1} \geq \Upsilon$
    \begin{equation} \label{Eq: Gh tails L2 norm asymptotics}
        \norm*{G_{h_n}^{Q, \beta_n} - \Pi_{I(\ell_n, \Upsilon)} G_{h_n}^{Q, \beta_n}}[L^2(\S[d])]^{2} \leq C \frac{(h_n)^{3-d}}{\Upsilon} \ .
    \end{equation}
\end{Theorem}

\begin{proof}
    Theorem \ref{Thm: Gh tails L2 norm estimates} may be inferred from \cite[Theorem 3.3]{Verdasco2026}, but we present a proof adapted to this setting $M = \S[d]$ for the sake of completeness.
    
    Throughout the proof, $C$ will denote a large positive constant whose value may change from line to line, but at all times independent of $h_n$ and $\Upsilon$.

    Recall that $\norm{Z_{\ell}^{q}}[L^2(\S[d])]^2 = \frac{m_{\ell}}{\vol(\S[d])}$ for $m_{\ell} \coloneqq \dim(\ker(\Lap {}- \lambda_{\ell}^2))$, hence one has (see \eqref{Eq: example cutoff Upsilon GF})
    \begin{equation} \label{Eq: aux norm tails Ghn}
        \norm*{G_{h_n}^{Q, \beta_n} - \Pi_{I(\ell_n, \Upsilon)} G_{h_n}^{Q, \beta_n}}[L^2(\S[d])]^{2} = \bigg[ \sum_{\abs{k} = \Upsilon + 1}^{\ell_n} + \sum_{k = \ell_n + 1}^{\infty} \bigg] \frac{\abs*{\beta_{n, q} + (-1)^{\ell_n + k} \beta_{n, -q}}^2}{\vol(\S[d])} \frac{ m_{\ell_n + k} }{ (\lambda_{\ell_n + k}^2 - h_n^{-2})^2} \ .
    \end{equation}
    First, observe that there exists $C > 1$ such that
    \[
    \frac{\abs*{\beta_{n, q} + (-1)^{\ell_n + k} \beta_{n, -q}}^2}{\vol(\S[d])} \leq C \qquad \forall \, n \in \N
    \]
    because $(\beta_{n})_{n \in \N}$ is bounded. The following lemma, whose proof is is presented below, helps us deal with the terms $\frac{ m_{\ell_n + k} }{ (\lambda_{\ell_n + k}^2 - h_n^{-2})^2}$.

    \begin{Lemma} \label{Lemma: bound for quotient(ln+k) |k|>Upsilon}
        There exists $C > 1$ such that for all $h_n \in (0,1)$, the following holds:
        \begin{enumerate}
            \item If there exists $\delta > 0$ such that $\sigma_n \in [0, 1 - \delta)$ for all $n \in \N$, then, for all $1 \leq \abs{k} \leq \ell_n$,
            \begin{equation*}
                \frac{ m_{\ell_n + k} }{ (\lambda_{\ell_n + k}^2 - h_n^{-2})^2 } \leq C (h_n)^{3-d} \frac{1}{\abs{k}^2} \ .
            \end{equation*}
            \item If there exists $\delta > 0$ such that $\sigma_n \in (\delta, 1)$ for all $n \in \N$, then, for all $1 \leq \abs{k - 1} \leq \ell_n$,
            \begin{equation*}
                \frac{ m_{\ell_n + k} }{ (\lambda_{\ell_n + k}^2 - h_n^{-2})^2 } \leq C (h_n)^{3-d} \frac{1}{\abs{k - 1}^2} \ .
            \end{equation*}
            \item If $k \geq \ell_n + 1$, then
            \begin{equation*}
                \frac{ m_{\ell_n + k} }{ (\lambda_{\ell_n + k}^2 - h_n^{-2})^2 } \leq C \frac{1}{ k^{5-d} } \ .
            \end{equation*}
        \end{enumerate}
    \end{Lemma}

    Assume that there exists $\delta > 0$ such that $\sigma_n \in [0, 1 -\delta)$ for all $n \in \N$. We may apply Lemma \ref{Lemma: bound for quotient(ln+k) |k|>Upsilon} on \eqref{Eq: aux norm tails Ghn} and get for $h_n^{-1} \geq \Upsilon \geq 1$, 
    \begin{align*}
        \norm*{G_{h_n}^{q} - \Pi_{I(\ell_n, \Upsilon)} G_{h_n}^{q}}[L^2(\S[d])]^{2} & \leq C \bigg[ (h_n)^{3-d} \sum_{\abs{k} = \Upsilon + 1}^{\ell_n} \frac{1}{k^2} + \sum_{k = \ell_n + 1}^{\infty} \frac{1}{k^{5-d}} \bigg] \\
        & \leq C \Big[ 2 (h_n)^{3-d} \frac{1}{\Upsilon} + (h_n)^{4-d} \Big] \leq \frac{C}{\Upsilon} (h_n)^{3-d} \ ,
    \end{align*}
    where we used that for all $R \in \N$ and $s > 1$,
    \begin{equation} \label{Eq: upper bound of squared harmonic tail}
        \sum_{k = R + 1}^{\infty} \frac{1}{k^s} \leq \int_{R}^{\infty} \frac{1}{t^s} \D{t} = \frac{s-1}{R} \ .
    \end{equation}
    In a similar fashion, if $\sigma_n \in (\delta, 1)$ for some $\delta > 0$, one gets the same estimate for $h_n^{-1} \geq \Upsilon \geq 1$.
\end{proof}

\begin{proof}[Proof of \texorpdfstring{Lemma \ref{Lemma: bound for quotient(ln+k) |k|>Upsilon}}{Lemma bound quotient}]
    First, since
    \begin{equation} \label{Eq: multiplicities ml}
    m_{\ell} = \begin{cases} 2\ell + 1 & \quad \text{if $d = 2$} \\ (\ell+1)^2 & \quad \text{if $d = 3$} \end{cases} \ ,
    \end{equation}
    we have that for $\abs{k} \leq \ell_n$,
    \[
    m_{\ell_n + k} \leq m_{2\ell_n} \leq C h_n^{1-d} \ .
    \]

    On the other hand, using that
    \[
    x(x + r) - y(y + r) = (x - y) (x + y + r) \ ,
    \]
    with
    \begin{align*}
        \lambda_{\ell_n + k}^2 & = (\ell_n + k) (\ell_n + k + d - 1) \\
        h_n^{-2} & = (\ell_n + \sigma_n) (\ell_n + \sigma_n + d - 1) \ ,
    \end{align*}
    we get
    \begin{equation} \label{Eq: first order approx of lambda(ln+k)}
        \lambda_{\ell_n + k}^2 - h_n^{-2} = (k - \sigma_n) [ 2(\ell_n + \tfrac{d - 1}{2}) + k + \sigma_n ] \ ,
    \end{equation}
    Therefore, if there exists $\delta > 0$ such that $\sigma_n \in [0, 1 - \delta)$ for all $n \in \N$, then
    \[
    \abs*{ \lambda_{\ell_n + k}^2 - h_n^{-2} } \geq \frac{1}{\delta} \abs{k} h_n^{-1} \qquad \forall \, 1 \leq \abs{k} \leq \ell_n \ .
    \]
    If there exists $\delta > 0$ such that $\sigma_n \in (\delta, 1)$ for all $n \in \N$, then
    \[
    \abs*{ \lambda_{\ell_n + k}^2 - h_n^{-2} } \geq \frac{1}{\delta} \abs{k - 1} h_n^{-1} \qquad \forall \, 1 \leq \abs{k - 1} \leq \ell_n \ .
    \]
    Combining the inequalities for $m_{\ell_n + k}$ and $\abs{\lambda_{\ell_n + k}^2 - h_n^{-2}}$ we get
    \[
    \frac{ m_{\ell_n + k} }{ (\lambda_{\ell_n + k}^2 - h_n^{-2})^2 } \leq C \frac{h_n^{1-d}}{\abs{k}^2 h_n^{-2} } \leq C (h_n)^{3-d} \frac{1}{\abs{k}^2} \qquad \forall \, 1 \leq \abs{k} \leq \ell_n \ , \qquad \text{if $\sigma_n \in [0, 1 - \delta)$,}
    \]
    and 
    \[
    \frac{ m_{\ell_n + k} }{ (\lambda_{\ell_n + k}^2 - h_n^{-2})^2} \leq C (h_n)^{3-d} \frac{1}{\abs{k - 1}^2} \qquad \forall \, 1 \leq \abs{k - 1} \leq \ell_n \ , \qquad \text{if $\sigma_n \in (\delta, 1)$.}
    \]

    On the other, if $k \geq \ell_n + 1$, from \eqref{Eq: first order approx of lambda(ln+k)} we have
    \[
    \abs*{ \lambda_{\ell_n + k}^2 - h_n^{-2} } \geq \frac{1}{C} k^2 \ ,
    \]
    and from \eqref{Eq: multiplicities ml},
    \[
    m_{\ell_n + k} \leq m_{2k} \leq C k^{1-d} \ ,
    \]
    therefore
    \[
    \frac{ m_{\ell_n + k} }{ (\lambda_{\ell_n + k}^2 - h_n^{-2})^2} \leq C \frac{1}{k^{5-d} } \qquad \forall \, k \geq \ell_n + 1 \ . \qedhere
    \]
\end{proof}

\subsection{Scenario 1}

\begin{Theorem} \label{Thm: Gh L2 norm estimates scenario 1}
    Let $(h_n)_{n \in \N} \subseteq (0,1)$ and $(\beta_{n})_{n \in \N} \subseteq \C[2]$, and for every $n \in \N$, let $\ell_n \in \N$ and $\sigma_n \in [0,1)$ as defined in \eqref{Eq: def elln and sigman}. Assume the following:
    \begin{enumerate}
        \item $h_n \to 0^+$ and $h_n^{-2} \notin \Spec(\Lap)$ for all $n \in \N$, thus $\sigma_n \in (0,1)$ for all $n \in \N$.
        \item $\beta_n \to \beta \in \C[2]$.
        \item $\ell_n$ is even for all $n \in \N$, or $\ell_n$ is odd for all $n \in \N$. Set $\rho \coloneqq (-1)^{\ell_n}$.
        \item $\sigma_n \to \sigma \in (0,1)$ as $n \to \infty$.
    \end{enumerate}
    Set
    \begin{equation} \label{Eq: Csigmabeta definition scenario 1}
        C_{\sigma, \beta, \rho} \coloneqq \bigg(\sum_{k \in \Z} \frac{ \abs{\beta_{q} + (-1)^{k} \rho \beta_{-q} }^2 }{ (k - \sigma)^2 } \bigg)^{\frac{1}{2}} \ .
    \end{equation}
    There exists $C > 1$ such that for every $\Upsilon \in \N$, there exists $\nu \in \N$ such that for all $n \geq \nu$,
    \begin{equation} \label{Eq: norm mainterm scenario 1}
        \abs*{ \norm*{ \Pi_{I(\ell_n, \Upsilon)} G_{h_n}^{Q, \beta_n} }[L^2(M)]^2 - \frac{(C_{\sigma, \beta, \rho})^2}{2(d-1) \vol(\S[d])}(h_n)^{3-d} } \leq C (h_n)^{3-d} \Big[ \Upsilon h_n + \abs{\sigma_n - \sigma} + \abs{\beta_n - \beta} + \Upsilon^{-1} \Big] \ .
    \end{equation}
\end{Theorem}

\begin{proof}
    Throughout the proof, $C$ will denote a large positive constant whose value may change from line to line and that may depend on $\sigma$ and $\sup_{n \in \N} \abs{\beta_{n}}$, but at all times is independent of $h_n$, $\sigma_n$, and $\Upsilon$. In the same way, $\nu$ will denote a large index whose value may change from line to line and may depend on $\Upsilon$ and how fast does $h_n \to 0$.

    In order to find the main term of the asymptotics of
    \begin{equation} \label{Eq: scenario 1 aux 3}
    \norm*{\Pi_{I(\ell_n, \Upsilon)} G_{h_n}^{Q, \beta_n} }[L^2(\S[d])]^{2} = \sum_{\abs{k} \leq \Upsilon} \abs{\beta_{n,q} + (-1)^{\ell_n + k} \beta_{n, -q}}^2 \frac{ \norm*{ Z_{\ell_n + k}^{q} }[L^2(\S[d])]^2 }{(\lambda_{\ell_n + k}^2 - h_n^{-2})^2} \ ,
    \end{equation}
    it will be useful to know good asymptotics for the terms involved in the previous sum. Recall that $\norm{Z_{\ell}^{q}}[L^2(\S[d])]^2 = \frac{m_{\ell}}{ \vol(\S[d]) }$.

    \begin{Lemma} \label{Lemma: asymp exp of lambda(ln+k)}
        For every $\Upsilon \in \N$, every $\abs{k} \leq \Upsilon$, and every $h_n \in (0,1)$,
        \begin{equation} \label{Eq: asymp exp of lambda(ln+k)}
            \abs*{ \lambda_{\ell_n + k}^2 - h_n^{-2} - 2(k - \sigma_n) h_n^{-1}} \leq 6 \abs{k - \sigma_n} \Upsilon \ .
        \end{equation}
        If one further assumes that $\sigma_n \to \sigma \in [0,1]$, then for all $\abs{k} \leq \Upsilon \leq h_n^{-1}$,
        \begin{equation} \label{Eq: asymp exp of lambda(ln+k) conv sigman}
            \abs*{ \lambda_{\ell_n + k}^2 - h_n^{-2} - 2(k - \sigma) h_n^{-1}} \leq 8 \Big[ \abs{k - \sigma} \Upsilon + \abs{\sigma_n - \sigma} h_n^{-1} \Big] \ .
        \end{equation}
    \end{Lemma}
    \begin{proof}
        From \eqref{Eq: first order approx of lambda(ln+k)},
        \begin{equation*}
            \lambda_{\ell_n + k}^2 - h_n^{-2} = (k - \sigma_n) [ 2(\ell_n + \sigma_n + \tfrac{d-1}{2}) + (k - \sigma_n)] \ ,
        \end{equation*}
        hence, we may subtract $2(k - \sigma_n) h_n^{-1}$ on both sides to obtain
        \begin{equation} \label{Eq: second order approx of lambda(ln+k)}
            \lambda_{\ell_n + k}^2 - h_n^{-2} - 2(k - \sigma_n) h_n^{-1} = (k - \sigma_n) \Big[ 2 (\ell_n + \sigma_n + \tfrac{d-1}{2} - h_n^{-1}) + (k - \sigma_n) \Big] \ .
        \end{equation}
        Applying triangular inequality and then Proposition \ref{Prop: hn elln approximation}, we get
        \[
        \abs*{ \lambda_{\ell_n + k}^2 - h_n^{-2} - 2(k - \sigma_n) h_n^{-1}} \leq 2 \abs{k - \sigma_n} \Big[ h_n + (\Upsilon + 1) \Big] \leq 6 \abs{k - \sigma_n} \Upsilon \ .
        \]

        If we assume that $\sigma_n \to \sigma \in [0,1]$, adding and subtracting $2\sigma h_n^{-1}$ on the left hand side of \eqref{Eq: second order approx of lambda(ln+k)} and then applying triangular inequality, we get
        \[
        \abs*{ \lambda_{\ell_n + k}^2 - h_n^{-2} - 2(k - \sigma) h_n^{-1} } \leq 6 \abs{k - \sigma} \Upsilon + 6 \Upsilon \abs{\sigma - \sigma_n} + 2 h_n^{-1} \abs{\sigma - \sigma_n} \ . \qedhere
        \]
    \end{proof}

    \begin{Lemma} \label{Lemma: asymp exp of m(ln+k)}
        There exists $C > 1$ such that for all $\Upsilon \in \N$, all $\abs{k} \leq \Upsilon$, and all $h_n^{-1} \geq \Upsilon$,
        \begin{equation*}
            \abs*{ m_{\ell_n + k} - \tfrac{2}{d - 1} (h_n)^{1 - d}} \leq C \Upsilon h_n^{2-d}  \ .
        \end{equation*}
    \end{Lemma}
    \begin{proof}
        If $d = 2$, we quickly see that
        \[
        m_{\ell_n + k} = 2(\ell_n + k) + 1 = 2(\ell_n + \sigma_n + \tfrac{1}{2}) + 2(k - \sigma_n) \ .
        \]
        Proposition \ref{Prop: hn elln approximation} imply the result if $d = 2$.

        If $d = 3$, in a similar fashion we have
        \begin{align*}
            m_{\ell_n + k} & = (\ell_n + k + 1)^2 = \big( (\ell_n + \sigma_n + 1) + (k - \sigma_n) \big)^2 \\
            & = (\ell_n + \sigma_n + 1)^2 + 2(k - \sigma_n) (\ell_n + \sigma_n + 1) + (k - \sigma_n)^2 \ .
        \end{align*}
        Once again, Proposition \ref{Prop: hn elln approximation} imply the result if $d = 3$ and $h_n^{-1} \geq \Upsilon$.
    \end{proof}

    \begin{Lemma} \label{Lemma: asymp exp of quotient(ln+k) |k|<Upsilon}
        There exists $C > 1$ such that for all $\Upsilon \in \N$, there exists $\nu \in \N$ such that for all $\abs{k} \leq \Upsilon$ and all $n \geq \nu$,
        \begin{equation} \label{Eq: asymp exp of quotient(ln+k)}
            \abs*{ \frac{ \norm*{ Z_{\ell_n + k}^{q} }[L^2(\S[d])]^2 }{(\lambda_{\ell_n + k}^2 - h_n^{-2})^2} - \frac{(h_n)^{3-d}}{2(d-1) \vol(\S[d])} \frac{1}{(k - \sigma_n)^2} } \leq C \frac{(h_n)^{3-d} }{\abs{k - \sigma_n}^2 } \Big[ \Upsilon h_n \Big] \ .
        \end{equation}
        If one further assumes that $\sigma_n \to \sigma \in [0,1]$, there exists $C > 1$ such that for all $\Upsilon \in \N$, there exists $\nu \in \N$ such that for all $\abs{k} \leq \Upsilon$, $k \neq \sigma$, and $n \geq \nu$,
        \begin{equation} \label{Eq: asymp exp of quotient(ln+k) conv sigman}
            \abs*{ \frac{ \norm*{ Z_{\ell_n + k}^{q} }[L^2(\S[d])]^2 }{(\lambda_{\ell_n + k}^2 - h_n^{-2})^2} - \frac{(h_n)^{3-d}}{2(d-1) \vol(\S[d])} \frac{1}{(k - \sigma)^2} } \leq C \frac{(h_n)^{3-d} }{\abs{k - \sigma}^2 } \bigg[ \Upsilon h_n + \abs{\sigma_n - \sigma} \Big] \ .
        \end{equation}
    \end{Lemma}
    \begin{proof}
        Adding and subtracting
        \[
        \frac{h_n^2}{4(k - \sigma_n)^2} \norm*{Z_{\ell_n + k}^{q} }[L^2(\S[d])]^2 \ ,
        \]
        then triangular inequality implies
        \begin{equation} \label{Eq: first approximation to quotient(ln+k)}
            \begin{aligned}
                \abs*{ \frac{ \norm*{ Z_{\ell_n + k}^{q} }[L^2(\S[d])]^2 }{(\lambda_{\ell_n + k}^2 - h_n^{-2})^2} - \frac{(h_n)^{3-d}}{2(d-1) \vol(\S[d])} \frac{1}{(k - \sigma_n)^2} } & \leq \norm*{Z_{\ell_n + k}^{q} }[L^2(\S[d])]^2 \abs*{ \frac{1}{(\lambda_{\ell_n + k}^2 - h_n^{-2})^2} - \frac{h_n^2}{4(k - \sigma_n)^2} } \\
                & \qquad + \frac{h_n^2}{4 \vol(\S[d]) \abs{k - \sigma_n}^2} \abs*{ m_{\ell_n + k} - \tfrac{2}{d-1} (h_n)^{1-d} } \ .
            \end{aligned}
        \end{equation}
        Thanks to Lemma \ref{Lemma: asymp exp of lambda(ln+k)}, we know that for all $\abs{k} \leq \Upsilon \leq h_n^{-1}$,
        \begin{equation*}
            \abs*{ \lambda_{\ell_n + k}^2 - h_n^{-2} - 2(k - \sigma_n) h_n^{-1} } \leq 6 \abs{k - \sigma_n} \Upsilon \ ,
        \end{equation*}
        and in particular, taking $\nu \in \N$ such that $6 \Upsilon \leq h_{\nu}^{-1}$, since $h_n \to 0^+$, one has that for all $n \geq \nu$,
        \[
        \abs*{ \lambda_{\ell_n + k}^2 - h_n^{-2} } \geq \abs{k - \sigma_n} h_n^{-1} \ , \qquad \text{and} \qquad \abs*{ \lambda_{\ell_n + k}^2 - h_n^{-2} } \leq 3 \abs{k - \sigma_n} h_n^{-1} \ .
        \]
        Using that
        \begin{equation} \label{Eq: aux 1/x2 - 1/y2}
            \abs[\bigg]{ \frac{1}{x^2} - \frac{1}{y^2} } \leq \frac{ \abs{x} + \abs{y}}{\abs{x}^2 \, \abs{y}^2 } \abs{x - y}
        \end{equation}
        with $x = \lambda_{\ell_n + k}^2 - h_n^{-2}$ and $y = 2h^{-1} (k - \sigma_n)$, we get that for all $\abs{k} \leq \Upsilon$ and all $n \geq \nu$,
        \begin{equation*}
            \abs*{ \frac{1}{(\lambda_{\ell_n + k}^2 - h_n^{-2})^2} - \frac{h_n^2}{4(k - \sigma_n)^2} } \leq \frac{5 \abs{k - \sigma_n} h_n^{-1} }{ 4 \abs{k - \sigma_n}^4 h_n^{-4} } \ 6 \abs{k - \sigma_n} \Upsilon \leq \tfrac{15}{2} \frac{\Upsilon h_n^3}{\abs{k - \sigma_n}^2 } \ .
        \end{equation*}
        On the other hand, Lemma \ref{Lemma: asymp exp of m(ln+k)} says that for all $\abs{k} \leq \Upsilon$ and $h_n$ small enough,
        \[
        \abs*{ m_{\ell_n + k} - \tfrac{2}{d-1} (h_n)^{1-d} } \leq C \Upsilon (h_n)^{2-d} \ .
        \]
        Coming back to \eqref{Eq: first approximation to quotient(ln+k)}, we get for all $\abs{k} \leq \Upsilon$ and $n \geq \nu$
        \begin{align*}
            \abs*{ \frac{ \norm{ Z_{\ell_n + k}^{q} }[L^2(\S[d])]^2 }{(\lambda_{\ell_n + k}^2 - h_n^{-2})^2} - \frac{(h_n)^{3-d}}{2(d-1) \vol(\S[d])} \frac{1}{(k - \sigma_n)^2} } & \leq C(h_n)^{1-d} \frac{\Upsilon h_n^3}{\abs{k - \sigma_n}^2 } + C \frac{h_n^2}{\abs{k - \sigma_n}^2} \Upsilon (h_n)^{2-d} \\
            & = \frac{C (h_n)^{3-d}}{\abs{k - \sigma_n}^2} \Big[ h_n \Upsilon \Big] \ .
        \end{align*}
        
        If one further assumes that $\sigma_n \to \sigma \in [0,1]$, the same strategy leads to the following estimate for all $\abs{k} \leq \Upsilon$, $k \neq \sigma$, and all $n \geq \nu$,
        \[
        \abs*{ \frac{1}{(\lambda_{\ell_n + k}^2 - h_n^{-2})^2} - \frac{h_n^2}{4(k - \sigma_n)^2} } \leq C \bigg[ \frac{\Upsilon h_n^3}{ \abs{k - \sigma}^2} + \frac{h_n^2}{ \abs{k - \sigma}^3 } \abs{\sigma_n - \sigma} \bigg] \ ,
        \]
        where $\nu \in \N$ is such that
        \[
        8\bigg[ \Upsilon h_{n} + \frac{\abs{\sigma_n - \sigma}}{\abs{k - \sigma}} \bigg] \leq 1 \qquad \text{for $k \in \{0, 1\} \setminus \{ \sigma \}$, and all $n \geq \nu$.}
        \]
        We may then apply this estimate to \eqref{Eq: first approximation to quotient(ln+k)} obtaining
        \begin{align*}
            & \abs*{ \frac{ \norm{ Z_{\ell_n + k}^{q} }[L^2(\S[d])]^2 }{(\lambda_{\ell_n + k}^2 - h_n^{-2})^2} - \frac{(h_n)^{3-d}}{2(d-1) \vol(\S[d])} \frac{1}{(k - \sigma)^2} } \leq \\
            & \hspace{100pt} \leq C(h_n)^{1-d} \bigg[ \frac{\Upsilon h_n^3}{\abs{k - \sigma}^2 } + \frac{h_n^2}{\abs{k - \sigma}^3} \abs{\sigma_n - \sigma} \bigg] + C \frac{h_n^2}{\abs{k - \sigma}^2} \Upsilon (h_n)^{2-d} \\
            & \hspace{100pt} \leq \frac{C (h_n)^{3-d}}{\abs{k - \sigma}^2} \Big[ h_n \Upsilon + \abs{\sigma_n - \sigma} \Big] \ ,
        \end{align*}
        for all $\abs{k} \leq \Upsilon$, $k \neq \sigma$, and all $n \geq \nu$, for some $C$ that depends on $\sigma \in (0,1)$.
    \end{proof}

    A combination of \eqref{Eq: scenario 1 aux 3},
    \[
    \text{adding and subtracting} \quad \frac{(h_n)^{3-d}}{2(d-1) \vol(\S[d])} \sum_{\abs{k} \leq \Upsilon} \frac{ \abs{\beta_{n,q} + (-1)^{k} \rho \beta_{n, -q}}^2 }{ (k - \sigma)^2 } \ ,
    \]
    and triangular inequality, leads to
    \begin{align*}
        & \abs*{ \norm*{\Pi_{I(\ell_n, \Upsilon)} G_{h_n}^{Q, \beta_n} }[L^2(\S[d])]^{2} - \frac{ (C_{\sigma, \beta, \rho})^2 }{ 2(d-1) \vol(\S[d]) } (h_n)^{3-d} } \leq \\
        & \hspace{70pt} \leq \sum_{\abs{k} \leq \Upsilon} \abs{\beta_{n, q} + (-1)^{k} \rho \beta_{n, -q}}^2 \abs*{ \frac{ \norm{ Z_{\ell_n + k}^{q} }[L^2(\S[d])]^2 }{(\lambda_{\ell_n + k}^2 - h_n^{-2})^2} - \frac{(h_n)^{3-d}}{2(d-1) \vol(\S[d])} \frac{1}{(k - \sigma)^2} } \\
        & \hspace{100pt} + \frac{(h_n)^{3-d}}{2(d-1) \vol(\S[d])} \sum_{\abs{k} \leq \Upsilon} \abs[\Big]{ \abs{\beta_{n,q} + (-1)^{k} \rho \beta_{n, -q}}^2 - \abs{\beta_{q} + (-1)^{k} \rho \beta_{-q}}^2 } \frac{1}{(k - \sigma)^2} \\
        & \hspace{100pt} + \frac{(h_n)^{3-d}}{2(d-1) \vol(\S[d])}  \sum_{ \abs{k} \geq \Upsilon + 1} \frac{ \abs{\beta_{q} + (-1)^{k} \rho \beta_{-q}}^2 }{ (k - \sigma)^2 } \eqqcolon \Sigma_1 + \Sigma_2 + \Sigma_3\ .
    \end{align*}

    Since we are assuming that $\sigma_n \to \sigma \in (0,1)$, Lemma \ref{Lemma: asymp exp of quotient(ln+k) |k|<Upsilon} provides some $C > 1$ (depending on $\sigma \in (0,1)$) such that for all $\Upsilon \in \N$ there exists $\nu \in \N$ such that for all $\abs{k} \leq \Upsilon$ and all $n \geq \nu$,
    \[
    \abs*{ \frac{ \norm{ Z_{\ell_n + k}^{q} }[L^2(\S[d])]^2 }{(\lambda_{\ell_n + k}^2 - h_n^{-2})^2} - \frac{(h_n)^{3-d}}{2(d-1) \vol(\S[d])} \frac{1}{(k - \sigma)^2} } \leq C \frac{(h_n)^{3-d} }{\abs{k - \sigma}^2 } \Big[ \Upsilon h_n + \abs{\sigma_n - \sigma} \Big] \ .
    \]
    If one further assumes that $C \geq \sup_{n \in \N} \abs{\beta_{n,q} + (-1)^{k} \rho \beta_{n,-q}}^2$, then for all $\Upsilon \in \N$, and all $n \geq \nu$,
    \begin{equation} \label{Eq: estimate Sigma1 scenario 1}
        \Sigma_1 \leq C \sum_{\abs{k} \leq \Upsilon} \frac{(h_n)^{3-d} }{\abs{k - \sigma}^2 } \Big[ \Upsilon h_n + \abs{\sigma_n - \sigma} \Big] \leq C (h_n)^{3-d} \Big[ \Upsilon h_n + \abs{\sigma_n - \sigma} \Big] \ .
    \end{equation}

    \begin{Lemma} \label{Lemma: zn wn converging sequences estimate}
        Let $(z_n)_{n \in \N} \subseteq \C$, and $(w_n)_{n \in \N} \subseteq \C$ converging sequences to $z_{\infty}$ and $w_{\infty}$ respectively. There exists $C > 1$ such that
        \[
        \abs[\Big]{ \abs{z_n + w_n}^2 - \abs{z_{\infty} + w_{\infty} }^2 } \leq C \Big[ \abs{z_n - z_{\infty}} + \abs{w_n - w_{\infty}} \Big]
        \]
    \end{Lemma}
    Thanks to Lemma \ref{Lemma: zn wn converging sequences estimate} for the sequences $z_{n} = \beta_{n,q}$ and $w_{n} = \rho \beta_{n,-q}$, there exists $C > 1$ such that for all $\Upsilon \in \N$ and all $h_n \in (0,1)$,
    \begin{equation} \label{Eq: estimate Sigma2 scenario 1}
        \Sigma_2 \leq C (h_n)^{3-d} \abs{\beta_{n} - \beta_{q}} \ .
    \end{equation}
    
    Lastly, using that for all $s > 0$ and all $\Upsilon \in \N$,
    \begin{equation} \label{Eq: upper bound tail of converging series}
        \sum_{k = \Upsilon + 1}^{\infty} \frac{1}{k^{1 + s}} \leq \int_{\Upsilon}^{\infty} x^{-1 - s} \D{x} = \frac{\Upsilon^{s}}{s} \ ,
    \end{equation}
    we get that for all $\Upsilon \in \N$ and all $h_n \in (0,1)$,
    \begin{equation} \label{Eq: estimate Sigma3 scenario 1}
        \Sigma_3 \leq C (h_n)^{3-d} \Upsilon^{-1} \ .
    \end{equation}

    Combining the estimates for $\Sigma_1$ \eqref{Eq: estimate Sigma1 scenario 1}, $\Sigma_2$ \eqref{Eq: estimate Sigma2 scenario 1} and $\Sigma_3$ \eqref{Eq: estimate Sigma3 scenario 1}, we have proved that \eqref{Eq: norm mainterm scenario 1} holds for all $\Upsilon \in \N$ and $n \geq \nu$, for a certain $C > 1$ that depends on $\sigma \in (0,1)$ and an upper bound for the sequence $(\abs{\beta_{n}})_{n \in \N}$.
\end{proof}

Recall that we defined for $\ell \in \N$
\[
z_{\ell}^{q} \coloneqq \frac{1}{\norm{Z_{\ell}^{q}}[L^2(\S[d])]} Z_{\ell}^{q} \ ,
\]
and for $h > 0$ and $\beta \in \C[2]$ such that $\sum_{q \in Q} \beta_{q} \delta_{q} = 0$ on $\ker(h^2 \Lap {}- 1)$,
\[
g_{h_n}^{Q, \beta_n} \coloneqq \frac{1}{ \norm*{G_{h_n}^{Q, \beta_n}}[L^2(\S[d])] } G_{h_n}^{Q, \beta_{n}} \ .
\]

\begin{Theorem} \label{Thm: gh quasimode scenario 1}
    Let $(h_n)_{n \in \N} \subseteq (0,1)$ and $(\beta_{n})_{n \in \N} \subseteq \C[2]$, and for every $n \in \N$, let $\ell_n \in \N$ and $\sigma_n \in [0,1)$ as defined in \eqref{Eq: def elln and sigman}. Assume the following:
    \begin{enumerate}
        \item $h_n \to 0^+$ and $h_n^{-2} \notin \Spec(\Lap)$ for all $n \in \N$, thus $\sigma_n \in (0,1)$ for all $n \in \N$.
        \item $\beta_n \to \beta \in \C[2]$.
        \item $\ell_n$ is even for all $n \in \N$, or $\ell_n$ is odd for all $n \in \N$. Set $\rho \coloneqq (-1)^{\ell_n}$.
        \item $\sigma_n \to \sigma \in (0,1)$ as $n \to \infty$.
    \end{enumerate}
    Set
    \begin{equation*}
        C_{\sigma, \beta, \rho} \coloneqq \bigg(\sum_{k \in \Z} \frac{ \abs{\beta_{q} + (-1)^{k} \rho \beta_{-q} }^2 }{ (k - \sigma)^2 } \bigg)^{\frac{1}{2}} \ .
    \end{equation*}
    There exists $C > 1$ such that for every $\Upsilon \in \N$, there exists $\nu \in \N$ such that for all $n \geq \nu$,
    \begin{equation} \label{Eq: ghQbeta quasimode scenario 1}
        \norm[\bigg]{ g_{h_n}^{Q, \beta_n} - \frac{1}{C_{\sigma, \beta, \rho}} \sum_{\abs{k} \leq \Upsilon} \frac{ \beta_{q} + (-1)^{k} \rho \beta_{-q} }{k - \sigma} z_{\ell_n + k}^{q} }[L^2(\S[d])] \leq C \Big[ \Upsilon h_n + \abs{\sigma_n - \sigma} + \abs{\beta_n - \beta} + \Upsilon^{-\frac{1}{2}} \Big] \ .
    \end{equation}
\end{Theorem}

\begin{proof}
    Adding and subtracting $\Pi_{I(\ell_n, \Upsilon)} g_{h_n}^{Q, \beta_n}$ and then applying triangular inequality, we have
    \begin{multline*}
        \norm*{ g_{h_n}^{Q, \beta_n} - \frac{1}{C_{\sigma, \beta, \rho}} \sum_{\abs{k} \leq \Upsilon} \frac{ \beta_{q} + (-1)^{k} \rho \beta_{-q} }{k - \sigma} z_{\ell_n + k}^{q} }[L^2(\S[d])] \leq \\
        \leq \frac{ \norm*{ G_{h_n}^{Q, \beta_n} - \Pi_{I(\ell_n, \Upsilon)} G_{h_n}^{Q, \beta_n} }[L^2(\S[d])] }{ \norm{G_{h_n}^{Q, \beta_n}}[L^2(\S[d])] } 
        + \norm[\bigg]{ \Pi_{I(\ell_n, \Upsilon)} g_{h_n}^{Q, \beta_n} - \frac{1}{C_{\sigma, \beta, \rho}} \sum_{\abs{k} \leq \Upsilon} \frac{ \beta_{q} + (-1)^{k} \rho \beta_{-q} }{k - \sigma} z_{\ell_n + k}^{q} }[L^2(\S[d])] \ .
    \end{multline*}

    For the first term, thanks to Theorem \ref{Thm: Gh tails L2 norm estimates} and Theorem \ref{Thm: Gh L2 norm estimates scenario 1}, there exists $C > 1$ such that for all $\Upsilon \in \N$, there exists $\nu \in \N$ such that for all $n \geq \nu$,
    \begin{equation} \label{Eq: quasimode thm scnr1 aux1}
        \frac{ \norm*{ G_{h_n}^{Q, \beta_n} - \Pi_{I(\ell_n, \Upsilon)} G_{h_n}^{Q, \beta_n} }[L^2(\S[d])] }{ \norm{G_{h_n}^{Q, \beta_n}}[L^2(\S[d])] } \leq C \Upsilon^{-\frac{1}{2}} \ .
    \end{equation}

    For the second term, adding and subtracting
    \[
    \frac{1}{C_{\sigma, \beta, \rho}} \sum_{\abs{k} \leq \Upsilon} \frac{ \beta_{n, q} + (-1)^{k} \rho \beta_{n, -q} }{ k - \sigma } z_{\ell_n + k}^{q} \ ,
    \]
    and using triangular, we have
    \begin{equation} \label{Eq: quasimode thm scnr1 aux2}
        \begin{aligned}
            & \norm*{ \Pi_{I(\ell_n, \Upsilon)} g_{h_n}^{Q, \beta_n} - \frac{1}{C_{\sigma, \beta, \rho}} \sum_{\abs{k} \leq \Upsilon} \frac{ \beta_{q} + (-1)^{k} \rho \beta_{-q} }{ k - \sigma } z_{\ell_n + k}^{q} }[L^2(\S[d])] \leq \\
            & \hspace{50pt} \leq \left( \sum_{\abs{k} \leq \Upsilon} \abs{ \beta_{n, q} + (-1)^{k} \rho \beta_{n, -q} }^2 \ \abs[\Bigg]{ \frac{ \norm*{Z_{\ell_n + k}^{q}}[L^2(\S[d])] }{ \norm*{G_{h_n}^{Q, \beta_n} }[L^2(\S[d])] (\lambda_{\ell_n + k}^2 - h_n^{-2}) } - \frac{1}{C_{\sigma, \beta, \rho}} \frac{1}{k - \sigma} }^2 \right)^{\frac{1}{2}} \\
            & \hspace{70pt} + \frac{1}{C_{\sigma, \beta, \rho}} \left( \sum_{\abs{k} \leq \Upsilon} \frac{ \abs{ (\beta_{n, q} - \beta_{q}) + (-1)^k \rho (\beta_{n, -q} - \beta_{-q}) }^2 }{(k - \sigma)^2} \right)^{\frac{1}{2}} \ .
        \end{aligned}
    \end{equation}

    \begin{Lemma} \label{Lemma: approx of quotient quasimode scenario 1}
        There exists $C > 1$ such that for all $\Upsilon \in \N$ there exists $\nu \in \N$ such that for all $n \geq \nu$ and all $\abs{k} \leq \Upsilon$,
        \begin{equation*}
            \abs[\Bigg]{ \frac{ \norm*{Z_{\ell_n + k}^{q}}[L^2(\S[d])] }{ \norm*{G_{h_n}^{Q, \beta_n} }[L^2(\S[d])] (\lambda_{\ell_n + k}^2 - h_n^{-2}) } - \frac{1}{C_{\sigma, \beta, \rho}} \frac{1}{k - \sigma} } \leq \frac{C}{\abs{k - \sigma}} \Big[ \Upsilon h_n + \abs{\sigma_n - \sigma} + \abs{\beta_n - \beta} + \Upsilon^{-1} \Big] \ .
        \end{equation*}
    \end{Lemma}
    \begin{proof}[Proof of \texorpdfstring{Lemma \ref{Lemma: approx of quotient quasimode scenario 1}}{Proof of Lemma approx quotient scn1}]
        We add and subtract
        \[
        \frac{1}{\norm*{G_{h_n}^{Q, \beta_n} }[L^2(\S[d])]} \bigg( \frac{(h_n)^{3-d}}{2(d-1) \vol(\S[d])} \bigg)^{\frac{1}{2}} \frac{1}{k - \sigma} \ ,
        \]
        and then apply triangular inequality. 

        On the one hand, thanks to Lemma \ref{Lemma: asymp exp of quotient(ln+k) |k|<Upsilon} and Theorem \ref{Thm: Gh L2 norm estimates scenario 1}, there exists $C > 1$ such that for all $\Upsilon \in \N$ there exists $\nu \in \N$ such that for all $n \geq \nu$ and all $\abs{k} \leq \Upsilon$,
        \[
        \frac{1}{\norm*{G_{h_n}^{Q, \beta_n} }[L^2(\S[d])]} \abs[\Bigg]{ \frac{ \norm*{Z_{\ell_n + k}^{q}}[L^2(\S[d])] }{ \lambda_{\ell_n + k}^2 - h_n^{-2} } - \bigg( \frac{(h_n)^{3-d}}{2(d-1) \vol(\S[d])} \bigg)^{\frac{1}{2}} \frac{1}{k - \sigma}  } \leq \frac{C}{ \abs{k - \sigma} } \Big[ \Upsilon h_n + \abs{\sigma_n - \sigma} \Big] \ .
        \]
        
        On the other hand, Theorem \ref{Thm: Gh L2 norm estimates scenario 1} provides some $C > 1$ such that for all $\Upsilon \in \N$ there exists $\nu \in \N$ such that for all $n \geq \nu$ and all $\abs{k} \leq \Upsilon$,
        \begin{align*}
            & \bigg( \frac{(h_n)^{3-d}}{2(d-1) \vol(\S[d])} \bigg)^{\frac{1}{2}} \frac{1}{\abs{k - \sigma}} \abs[\bigg]{ \frac{1}{\norm*{G_{h_n}^{Q, \beta_n} }[L^2(\S[d])]} - \bigg( \frac{(h_n)^{3-d}}{2(d-1) \vol(\S[d])} \bigg)^{-\frac{1}{2}} \frac{1}{C_{\sigma, \beta, \rho}} } \leq \\
            & \hspace{200pt} \leq \frac{C}{\abs{k - \sigma}} \Big[ \Upsilon h_n + \abs{\sigma_n - \sigma} + \abs{\beta_n - \beta} + \Upsilon^{-1} \Big] \ . \qedhere
        \end{align*}
    \end{proof}

    Meanwhile, using that there exists $C > 1$ such that for all $n \in \N$
    \[
    \abs{ (\beta_{n, q} - \beta_{q}) + (-1)^k \rho (\beta_{n, -q} - \beta_{-q}) } \leq C \abs{\beta_n - \beta} \ ,
    \]
    we get for all $\Upsilon$ and all $n \in \N$,
    \[
    \sum_{\abs{k} \leq \Upsilon} \frac{ \abs{ (\beta_{n, q} - \beta_{q}) + (-1)^k \rho (\beta_{n, -q} - \beta_{-q}) }^2 }{(k - \sigma)^2} \leq C \abs{\beta_n - \beta} \ .
    \]
    Last estimate and Lemma \ref{Lemma: approx of quotient quasimode scenario 1} allows us conclude that there exists $C > 1$ such that for all $\Upsilon \in \N$ there exists $\nu \in \N$ such that for all $n \geq \nu$
    \begin{equation} \label{Eq: quasimode thm scnr1 aux3}
        \norm*{ \Pi_{I(\ell_n, \Upsilon)} g_{h_n}^{Q, \beta_n} - \frac{1}{C_{\sigma, \beta, \rho}} \sum_{\abs{k} \leq \Upsilon} \frac{ \beta_{q} + (-1)^{k} \rho \beta_{-q} }{ k - \sigma } z_{\ell_n + k}^{q} }[L^2(\S[d])] \leq C \Big[ \Upsilon h_n + \abs{\sigma_n - \sigma} + \abs{\beta_n - \beta} + \Upsilon^{-1} \Big] \ .
    \end{equation}
    Theorem \ref{Thm: gh quasimode scenario 1} follows after a combination of \eqref{Eq: quasimode thm scnr1 aux1} and \eqref{Eq: quasimode thm scnr1 aux3}.
\end{proof}

\subsection{Scenario 2}

\begin{Theorem} \label{Thm: Gh L2 norm estimates scenario 2}
    Let $(h_n)_{n \in \N} \subseteq (0,1)$, and $(\beta_{n})_{n \in \N} \subseteq \C[2]$, and for every $n \in \N$, let $\ell_n \in \N$ and $\sigma_n \in [0,1)$ as defined in \eqref{Eq: def elln and sigman}. Assume the following:
    \begin{enumerate}
        \item $h_n \to 0^+$ and $h_n^{-2} \notin \Spec(\Lap)$ for all $n \in \N$, thus $\sigma_n \in (0,1)$ for all $n \in \N$.
        \item $\beta_n \to \beta \in \C[2]$.
        \item $\ell_n$ is even for all $n \in \N$, or $\ell_n$ is odd for all $n \in \N$. Set $\rho \coloneqq (-1)^{\ell_n}$.
        \item $\sigma_n \to \sigma \in \{0,1\}$ as $n \to \infty$.
        \item The following condition holds,
        \begin{equation} \label{Eq: infty condition scenario 2}
            \lim_{n \to \infty} \frac{\abs*{\beta_{n,q} + (-1)^{\sigma} \rho \beta_{n,-q}}}{ \abs{\sigma - \sigma_n} } = \infty \ .
        \end{equation}
    \end{enumerate}
    There exists $C > 1$ and $\nu \in \N$ such that for every $n \geq \nu$,
    \begin{gather} \label{Eq: norm tails scenario 2}
        \norm*{ G_{h_n}^{Q, \beta_n} - \Pi_{I(\ell_n + \sigma, 0)} G_{h_n}^{Q, \beta_n} }[L^2(\S[d])]^2 \leq C (h_n)^{3-d} \ , \\
        \label{Eq: norm mainterm scenario 2}
        \abs*{ \norm*{ \Pi_{I(\ell_n + \sigma, 0)} G_{h_n}^{Q, \beta_n} }[L^2(M)]^2 - \frac{(h_n)^{3-d} }{2(d-1) \vol(\S[d])} \frac{ \abs*{\beta_{n,q} + (-1)^{\sigma} \rho \beta_{n,-q}}^2 }{ \abs{\sigma - \sigma_n}^2 } } \leq C(h_n)^{4-d} \frac{ \abs*{\beta_{n,q} + (-1)^{\sigma} \rho \beta_{n,-q}}^2 }{ \abs{\sigma - \sigma_n}^2 }
    \end{gather} 
\end{Theorem}

\begin{proof}
    \eqref{Eq: norm tails scenario 2} can be proven as \eqref{Eq: Gh tails L2 norm asymptotics} in Theorem \ref{Thm: Gh tails L2 norm estimates}, using that
    \begin{align*}
        \norm*{G_{h_n}^{Q, \beta_n} - \Pi_{I(\ell_n + \sigma, 0)} G_{h_n}^{Q, \beta_n}}[L^2(\S[d])]^2 & \\
        & \hspace{-40pt} = \bigg[ \sum_{1 \leq \abs{k - \sigma} \leq \ell_n + \sigma} + \sum_{k \geq \ell_n + 1 + 2\sigma} \bigg] \frac{\abs*{\beta_{n, q} + (-1)^{k} \rho \beta_{n, -q}}^2}{\vol(\S[d])} \frac{ m_{\ell_n + k} }{ (\lambda_{\ell_n + k}^2 - h_n^{-2})^2} \\
        & \hspace{-40pt} \leq \sum_{1 \leq \abs{k - \sigma} \leq \ell_n + \sigma} C \frac{(h_n)^{3-d}}{\abs{k - \sigma}^2} + \sum_{k \geq \ell_n + 1 + 2\sigma} C \frac{1}{k^{5-d}} \\
        & \hspace{-40pt} \leq C (h_n)^{3-d} \ ,
    \end{align*}
    thanks to Lemma \ref{Lemma: bound for quotient(ln+k) |k|>Upsilon}.

    On the other hand, observe that
    \[
    \norm*{ \Pi_{I(\ell_n + \sigma, 0)} G_{h_n}^{Q, \beta_n} }[L^2(M)]^2 = \abs*{\beta_{n,q} + (-1)^{\sigma} \rho \beta_{n, -q}} \frac{\norm*{Z_{\ell_n + \sigma}^{q}}[L^2(\S[d])]^2 }{ (\lambda_{\ell_n + \sigma}^2 - h_n^{-2})^2 } \ .
    \]
    From Lemma \ref{Lemma: asymp exp of quotient(ln+k) |k|<Upsilon} with $\Upsilon = 1$ and $k = \sigma$, we read that
    \[
    \abs*{ \norm*{ \Pi_{I(\ell_n + \sigma, 0)} G_{h_n}^{Q, \beta_n} }[L^2(M)]^2 - \frac{(h_n)^{3-d}}{2(d-1) \vol(\S[d])} \frac{ \abs*{\beta_{n,q} + (-1)^{\sigma} \rho \beta_{n,-q}}^2 }{ \abs{\sigma - \sigma_n}^2 } } \leq C (h_n)^{4-d} \frac{\abs{\beta_{n,q} + (-1)^{\sigma} \rho \beta_{n,-q}}^2 }{\abs{\sigma - \sigma_n}^2} \ ,
    \]
    as we wanted to prove.
\end{proof}

\begin{Theorem} \label{Thm: gh quasimode scenario 2}
    Let $(h_n)_{n \in \N} \subseteq (0,1)$ and $(\beta_{n})_{n \in \N} \subseteq \C[2]$, and for every $n \in \N$, let $\ell_n \in \N$ and $\sigma_n \in [0,1)$ as defined in \eqref{Eq: def elln and sigman}. Assume the following:
    \begin{enumerate}
        \item $h_n \to 0^+$ and $h_n^{-2} \notin \Spec(\Lap)$ for all $n \in \N$, thus $\sigma_n \in (0,1)$ for all $n \in \N$.
        \item $\beta_n \to \beta \in \C[2]$.
        \item $\ell_n$ is even for all $n \in \N$, or $\ell_n$ is odd for all $n \in \N$. Set $\rho \coloneqq (-1)^{\ell_n}$.
        \item $\sigma_n \to \sigma \in \{0,1\}$ as $n \to \infty$.
        \item The following condition holds,
        \begin{equation*}
            \lim_{n \to \infty} \frac{\abs*{\beta_{n,q} + (-1)^{\sigma} \rho \beta_{n,-q}}}{ \abs{\sigma - \sigma_n} } = \infty \ .
        \end{equation*}
    \end{enumerate}
    There exists $C > 1$ and $\nu \in \N$ such that for every $n \geq \nu$,
    \begin{equation} \label{Eq: ghQbeta quasimode scenario 2}
        \norm*{ g_{h_n}^{Q, \beta_n} - (-1)^{1 - \sigma} z_{\ell_n + \sigma}^{q} }[L^2(\S[d])] \leq C \bigg[ h_n + \frac{ \abs{\sigma - \sigma_n} }{\abs{\beta_{n,q} + (-1)^{\sigma} \rho \beta_{n,-q}} } \bigg] \ .
    \end{equation}
\end{Theorem}

\begin{proof}
    Adding and subtracting $\Pi_{I(\ell_n + \sigma, 0)} g_{h_n}^{Q, \beta_n}$ and then applying triangular inequality we obtain
    \[
    \norm*{ g_{h_n}^{Q, \beta_n} - z_{\ell_n + \sigma}^{q} }[L^2(\S[d])] \leq \frac{\norm*{ G_{h_n}^{Q, \beta_n} - \Pi_{I(\ell_n + \sigma, 0)} G_{h_n}^{Q, \beta_n} }[L^2(\S[d])] }{ \norm*{G_{h_n}^{Q, \beta_n}}[L^2(\S[d])]^2 } + \norm*{\Pi_{I(\ell_n + \sigma, 0)} g_{h_n}^{Q, \beta_n} - (-1)^{1 - \sigma} z_{\ell_n + \sigma}^{q} }[L^2(\S[d])] \ .
    \]
    On the one hand, thanks to Theorem \ref{Thm: Gh L2 norm estimates scenario 2}, there exists $C > 1$ and $\nu \in \N$ such that
    \[
    \frac{\norm*{ G_{h_n}^{Q, \beta_n} - \Pi_{I(\ell_n + \sigma, 0)} G_{h_n}^{Q, \beta_n} }[L^2(\S[d])] }{ \norm*{G_{h_n}^{Q, \beta_n}}[L^2(\S[d])]^2 } \leq C \frac{ \abs{\sigma - \sigma_n} }{\abs*{\beta_{n,q} + (-1)^{\sigma} \rho \beta_{n,-q}}} \ .
    \]
    On the other hand, using that $(-1)^{1 - \sigma} = \frac{\abs{\sigma - \sigma_n}}{\sigma - \sigma_n}$ for all $n \in \N$,
    \[
    \norm*{\Pi_{I(\ell_n + \sigma, 0)} g_{h_n}^{Q, \beta_n} - (-1)^{1 - \sigma} z_{\ell_n + \sigma}^{q} }[L^2(\S[d])] = \abs*{ \frac{ \beta_{n, q} + (-1)^{\sigma} \rho \beta_{n, -q} }{\norm*{G_{h_n}^{Q, \beta_n}}[L^2(\S[d])]} \frac{\norm*{Z_{\ell_n + \sigma}^{q}}[L^2(\S[d])] }{ \lambda_{\ell_n + \sigma}^{2} - h_n^{-2} } - \frac{\abs{\sigma - \sigma_n}}{\sigma - \sigma_n} } \ .
    \]
    Adding and subtracting
    \[
    \frac{ \beta_{n, q} + (-1)^{\sigma} \rho \beta_{n, -q} }{\norm*{G_{h_n}^{Q, \beta_n}}[L^2(\S[d])]} \bigg( \frac{(h_n)^{3-d}}{2(d-1) \vol(\S[d])} \bigg)^{\frac{1}{2}} \frac{1}{\sigma - \sigma_n}
    \]
    and then using triangular inequality in combination with Lemma \ref{Lemma: asymp exp of quotient(ln+k) |k|<Upsilon}, with $\Upsilon = 1$ and $k = \sigma$, and Theorem \ref{Thm: Gh L2 norm estimates scenario 2} one gets that there exists $C > 1$ and $\nu \in \N$ such that for all $n \geq \nu$,
    \[
    \abs*{ \frac{ \beta_{n, q} + (-1)^{\sigma} \rho \beta_{n, -q} }{\norm*{G_{h_n}^{Q, \beta_n}}[L^2(\S[d])]} \frac{\norm*{Z_{\ell_n + \sigma}^{q}}[L^2(\S[d])] }{ \lambda_{\ell_n + \sigma}^{2} - h_n^{-2} } - \frac{\abs{\sigma - \sigma_n}}{\sigma - \sigma_n} } \leq C h_n \ .
    \]
\end{proof}

\subsection{Scenario 3}

\begin{Theorem} \label{Thm: Gh L2 norm estimates scenario 3}
    Let $(h_n)_{n \in \N} \subseteq (0,1)$ and $(\beta_{n})_{n \in \N} \subseteq \C[2]$, and for every $n \in \N$, let $\ell_n \in \N$ and $\sigma_n \in [0,1)$ as defined in \eqref{Eq: def elln and sigman}. Assume the following:
    \begin{enumerate}
        \item $h_n \to 0^+$ and $h_n^{-2} \notin \Spec(\Lap)$ for all $n \in \N$, thus $\sigma_n \in (0,1)$ for all $n \in \N$.
        \item $\beta_n \to \beta \in \C[2]$.
        \item $\ell_n$ is even for all $n \in \N$, or $\ell_n$ is odd for all $n \in \N$. Set $\rho \coloneqq (-1)^{\ell_n}$.
        \item $\sigma_n \to \sigma \in \{0,1\}$ as $n \to \infty$.
        \item $\beta_{q} + (-1)^{\sigma} \rho \beta_{-q} = 0$, so that
        \begin{equation} \label{Eq: def csigmabeta limit}
            c_{\sigma, \beta, \rho} \coloneqq \lim_{n \to \infty} \frac{\beta_{n,q} + (-1)^{\sigma} \rho \beta_{n,-q}}{ \sigma - \sigma_n } \in \C \ .
        \end{equation}
    \end{enumerate}
    Set
    \begin{equation} \label{Eq: Csigmabeta definition limitsigma}
        C_{\sigma, \beta, \rho} \coloneqq \bigg(\abs{c_{\sigma, \beta, \rho}}^2 + \sum_{\substack{ k \in \Z \\ k \neq \sigma} } \frac{ \abs{\beta_{q} + (-1)^{k} \rho \beta_{-q} }^2 }{ (k - \sigma)^2 } \bigg)^{\frac{1}{2}} \ .
    \end{equation}
    There exists $C > 1$ such that for every $\Upsilon \in \N$, there exists $\nu \in \N$ such that for all $n \geq \nu$,
    \begin{multline} \label{Eq: norm mainterm scenario 3}
        \abs*{ \norm*{ \Pi_{I(\ell_n + \sigma, 2\Upsilon - 1)} G_{h_n}^{Q, \beta_n} }[L^2(M)]^2 - \frac{(C_{\sigma, \beta, \rho})^2}{2(d-1) \vol(\S[d])}(h_n)^{3-d} } \leq \\
        \leq C (h_n)^{3-d} \bigg[ \Upsilon h_n + \abs{\sigma_n - \sigma} + \abs{\beta_n - \beta} + \abs*{ c_{\sigma, \beta, \rho} - \frac{\beta_{n,q} + (-1)^{\sigma} \rho \beta_{n,-q}}{ \sigma - \sigma_n } } + \Upsilon^{-1} \bigg] \ .
    \end{multline}
\end{Theorem}

\begin{proof}
    Throughout the proof, $C$ will denote a large positive constant whose value may change from line to line and that may depend on $\sup_{n \in \N} \abs{\beta_{n}}$, but at all times is independent of $h_n$, $\sigma_n$, and $\Upsilon$. In the same way, $\nu$ will denote a large index whose value may change from line to line and may depend on $\Upsilon$ and how fast does $h_n \to 0$.

    We proceed as in the proof of Theorem \ref{Thm: Gh L2 norm estimates scenario 1}. We have for all $n \in \N$,
    \begin{equation*}
        \norm*{ \Pi_{I(\ell_n + \sigma, 2\Upsilon - 1)} G_{h_n}^{Q, \beta_n} }[L^2(M)]^2 = \sum_{\abs{k - \sigma} \leq 2\Upsilon - 1} \abs{\beta_{n,q} + (-1)^{k} \rho \beta_{n, -q}}^2 \frac{ \norm{Z_{\ell_n + k}^{q}}[L^2(\S[d])]^2 }{ \abs{\lambda_{\ell_n + k}^2 + h_n^{-2}}^2 } \ .
    \end{equation*}
    We start applying triangular inequality, splitting the sums into the following categories $k = \sigma$, $1 \leq \abs{k - \sigma} \leq 2\Upsilon - 1$, and $\abs{k - \sigma} \geq 2\Upsilon$:
    \begin{equation} \label{Eq: thm Gh norms scnr3 aux1}
        \abs*{ \norm*{\Pi_{I(\ell_n + \sigma, 2\Upsilon - 1)} G_{h_n}^{Q, \beta_n} }[L^2(\S[d])]^{2} - \frac{ (C_{\sigma, \beta, \rho})^2 }{ 2(d-1) \vol(\S[d]) } (h_n)^{3-d} } \leq \Sigma_{\sigma} + \Sigma_{\Upsilon} + \Sigma_{\infty} \ ,
    \end{equation}
    where
    \begin{align*}
        \Sigma_{\sigma} & \coloneqq \abs[\bigg]{ \abs{\beta_{n, q} + (-1)^{\sigma} \rho \beta_{n, -q}}^2 \frac{ \norm*{Z_{\ell_n}^{q}}[L^2(\S[d])]^2 }{ \abs{ \lambda_{\ell_n + \sigma}^{2} - h_n^{-2} }^2 } - \frac{ (h_n)^{3-d} }{2(d-1) \vol(\S[d]) } \abs{C_{\sigma, \beta, \rho} }^2 } \ ,\\
        \Sigma_{\Upsilon} & \coloneqq \sum_{ \abs{k - \sigma} \leq 2\Upsilon - 1 } \abs[\bigg]{ \abs{\beta_{n,q} + (-1)^{k} \rho \beta_{n, -q}}^2 \frac{ \norm{Z_{\ell_n + k}^{q}}[L^2(\S[d])]^2 }{ \abs{\lambda_{\ell_n + k}^2 + h_n^{-2}}^2 } - \abs{\beta_{q} + (-1)^k \beta_{-q}}^2 \frac{ (h_n)^{3-d} }{ 2(d-1) \vol(\S[d]) \abs{k - \sigma}^2 } } \ , \\
        \Sigma_{\infty} & \coloneqq \frac{(h_n)^{3-d}}{2(d-1) \vol(\S[d])}  \sum_{ \abs{k - \sigma} \geq 2\Upsilon} \frac{ \abs{\beta_{q} + (-1)^{k} \rho \beta_{-q}}^2 }{ (k - \sigma)^2 } \ .
    \end{align*}

    We study $\Sigma_{\sigma}$ first. Adding and subtracting
    \[
    \frac{ (h_n)^{3-d} }{2(d-1) \vol(\S[d]) } \frac{ \abs{\beta_{n, q} + (-1)^{\sigma} \rho \beta_{n, -q}}^2 }{ \abs{\sigma_n - \sigma}^2 }
    \]
    we get
    \begin{align*}
        \Sigma_{\sigma} & \leq \abs*{\beta_{n, q} + (-1)^{\sigma} \rho \beta_{n, -q}}^2 \abs*{ \frac{ \norm*{Z_{\ell_n}^{q}}[L^2(\S[d])]^2 }{ \abs{ \lambda_{\ell_n + \sigma}^{2} - h_n^{-2} }^2 } - \frac{ (h_n)^{3-d} }{2(d-1) \vol(\S[d]) } \frac{1}{\abs{\sigma_n - \sigma}^2} } \\
        & \hspace{40pt} + \frac{ (h_n)^{3-d} }{2(d-1) \vol(\S[d]) } \abs[\bigg]{ \frac{ \abs{\beta_{n, q} + (-1)^{\sigma} \rho \beta_{n, -q}}^2 }{ \abs{\sigma_n - \sigma}^2 } - \abs{c_{\sigma, \beta, \rho}}^2 }
    \end{align*}
    Thanks to Lemma \ref{Lemma: asymp exp of quotient(ln+k) |k|<Upsilon} with $\Upsilon = 1$ and $k = \sigma$, and since \eqref{Eq: def csigmabeta limit} hold, there exists some $C > 1$ and $\nu \in \N$ such that for all $n \geq \nu$,
    \begin{equation} \label{Eq: thm Gh norms scnr3 aux2}
        \begin{aligned}
            \Sigma_{\sigma} & \leq \abs{\beta_{n, q} + (-1)^{\sigma} \rho \beta_{n, -q}}^2 C \frac{ (h_n)^{4-d} }{\abs{\sigma - \sigma_n}^2} + C (h_n)^{3-d} \abs[\bigg]{ \frac{\beta_{n,q} + (-1)^{\sigma} \rho \beta_{n,-q}}{ \sigma - \sigma_n } - c_{\sigma, \beta, \rho} } \\
            & \leq C (h_n)^{3-d} \bigg[ h_n + \abs[\bigg]{ \frac{\beta_{n,q} + (-1)^{\sigma} \rho \beta_{n,-q}}{ \sigma - \sigma_n } - c_{\sigma, \beta, \rho} } \bigg] \ .
        \end{aligned}
    \end{equation}

    Then, for $\Sigma_{\Upsilon}$, we add and subtract for every $1 \leq \abs{k - \sigma} \leq 2\Upsilon - 1$,
    \[
    \abs{\beta_{n,q} + (-1)^{k} \rho \beta_{n, -q}}^2 \frac{ (h_n)^{3-d} }{ 2(d-1) \vol(\S[d]) \abs{k - \sigma}^2 } \ ,
    \]
    and then apply triangular inequality, obtaining
    \begin{align*}
        \Sigma_{\Upsilon} & \leq \sum_{\abs{k - \sigma} = 1}^{2\Upsilon - 1} \abs{\beta_{n,q} + (-1)^{k} \rho \beta_{n, -q}}^2 \abs*{ \frac{ \norm{Z_{\ell_n + k}^{q}}[L^2(\S[d])]^2 }{ \abs{\lambda_{\ell_n + k}^2 + h_n^{-2}}^2 } - \frac{ (h_n)^{3-d} }{ 2(d-1) \vol(\S[d]) \abs{k - \sigma}^2 } } \\
        & \hspace{40pt} + \frac{ (h_n)^{3-d} }{ 2(d-1) \vol(\S[d]) } \sum_{\abs{k - \sigma} = 1}^{2\Upsilon - 1} \abs[\Big]{ \abs{\beta_{n,q} + (-1)^{k} \rho \beta_{n, -q}}^2 - \abs{\beta_{q} + (-1)^{k} \rho \beta_{-q}}^2 } \frac{1}{ \abs{k - \sigma}^2 }
    \end{align*}
    Thanks to Lemma \ref{Lemma: asymp exp of quotient(ln+k) |k|<Upsilon} for all $1 \leq \abs{k - \sigma} \leq \Upsilon$, and Lemma \ref{Lemma: zn wn converging sequences estimate} with $z_{n} = \beta_{n,q}$ and $w_{n} = \rho \beta_{n,-q}$, there exists $C > 1$ such that for all $\Upsilon \in \N$ there exists $\nu \in \N$ such that for all $n \geq \nu$,
    \begin{equation} \label{Eq: thm Gh norms scnr3 aux3}
        \begin{aligned}
            \Sigma_{\Upsilon} & \leq C (h_n)^{3-d} \Big[ \Upsilon h_n + \abs{\sigma - \sigma_n} \Big] \sum_{\abs{k - \sigma} = 1}^{2\Upsilon - 1} \frac{1}{ \abs{k - \sigma}^2 } + C (h_n)^{3-d} \abs{\beta_n - \beta} \sum_{\abs{k - \sigma} = 1}^{2\Upsilon - 1} \frac{1}{ \abs{k - \sigma}^2 } \\
            & \leq C (h_n)^{3-d} \Big[ \Upsilon h_n + \abs{\sigma - \sigma_n} + \abs{\beta_n - \beta} \Big] \ .
        \end{aligned}
    \end{equation}

    Lastly, thanks to \eqref{Eq: upper bound tail of converging series}, there exists $C > 1$ such that for all $\Upsilon \in \N$ and all $n \in \N$,
    \begin{equation} \label{Eq: thm Gh norms scnr3 aux4}
        \Sigma_{\infty} \leq C (h_n)^{3-d} \Upsilon^{-1} \ .
    \end{equation}
    Substituting \eqref{Eq: thm Gh norms scnr3 aux2}, \eqref{Eq: thm Gh norms scnr3 aux3}, and \eqref{Eq: thm Gh norms scnr3 aux4} into \eqref{Eq: thm Gh norms scnr3 aux1}, the result is concluded.
\end{proof}

\begin{Theorem} \label{Thm: gh quasimode scenario 3}
    Let $(h_n)_{n \in \N} \subseteq (0,1)$ and $(\beta_{n})_{n \in \N} \subseteq \C[2]$, and for every $n \in \N$, let $\ell_n \in \N$ and $\sigma_n \in [0,1)$ as defined in \eqref{Eq: def elln and sigman}. Assume the following:
    \begin{enumerate}
        \item $h_n \to 0^+$ and $h_n^{-2} \notin \Spec(\Lap)$ for all $n \in \N$, thus $\sigma_n \in (0,1)$ for all $n \in \N$.
        \item $\beta_n \to \beta \in \C[2]$.
        \item $\ell_n$ is even for all $n \in \N$, or $\ell_n$ is odd for all $n \in \N$. Set $\rho \coloneqq (-1)^{\ell_n}$.
        \item $\sigma_n \to \sigma \in \{0,1\}$ as $n \to \infty$.
        \item $\beta_{q} + (-1)^{\sigma} \rho \beta_{-q} = 0$, so that
        \begin{equation*}
            c_{\sigma, \beta, \rho} \coloneqq \lim_{n \to \infty} \frac{\beta_{n,q} + (-1)^{\sigma} \rho \beta_{n,-q}}{ \sigma - \sigma_n } \in \C \ .
        \end{equation*}
    \end{enumerate}
    Set
    \begin{equation*}
        C_{\sigma, \beta, \rho} \coloneqq \bigg(\abs{c_{\sigma, \beta, \rho}}^2 + \sum_{\substack{ k \in \Z \\ k \neq \sigma} } \frac{ \abs{\beta_{q} + (-1)^{k} \rho \beta_{-q} }^2 }{ (k - \sigma)^2 } \bigg)^{\frac{1}{2}} \ .
    \end{equation*}
    There exists $C > 1$ such that for every $\Upsilon \in \N$, there exists $\nu \in \N$ such that for all $n \geq \nu$,
    \begin{multline} \label{Eq: ghQbeta quasimode scenario 3}
        \norm*{ g_{h_n}^{Q, \beta_n} - \frac{1}{C_{\sigma, \beta, \rho}} \bigg[ c_{\sigma, \beta, \rho} z_{\ell_n + \sigma}^{q} + \sum_{ \abs{k - \sigma} = 1}^{2\Upsilon - 1} \frac{ \beta_{q} + (-1)^{k} \rho \beta_{-q} }{ k - \sigma } z_{\ell_n + k}^{q} \bigg] }[L^2(\S[d])] \\ 
        \leq C \bigg[ \Upsilon h_n + \abs{\sigma_n - \sigma} + \abs{\beta_n - \beta} + \abs[\bigg]{ c_{\sigma, \beta, \rho} - \frac{\beta_{n,q} + (-1)^{\sigma} \rho \beta_{n,-q}}{ \sigma - \sigma_n } } + \Upsilon^{-\frac{1}{2}} \bigg] \ .
    \end{multline}
\end{Theorem}

\begin{proof}
    We mimic the proof of Theorem \ref{Thm: gh quasimode scenario 1}. Adding and subtracting $\Pi_{I(\ell_n + \sigma, 2\Upsilon - 1)} g_{h_n}^{Q, \beta_n}$ and using triangular inequality,
    \begin{align*}
        & \norm*{ g_{h_n}^{Q, \beta_n} - \frac{1}{C_{\sigma, \beta, \rho}} \bigg[ c_{\sigma, \beta, \rho} z_{\ell_n + \sigma}^{q} + \sum_{ \abs{k - \sigma} = 1}^{2\Upsilon - 1} \frac{ \beta_{q} + (-1)^{k} \rho \beta_{-q} }{ k - \sigma } z_{\ell_n + k}^{q} \bigg] }[L^2(\S[d])] \leq \\
        & \hspace{50pt} \leq \frac{ \norm*{ G_{h_n}^{Q, \beta_n} - \Pi_{I(\ell_n + \sigma, 2\Upsilon - 1)} G_{h_n}^{Q, \beta_n} }[L^2(\S[d])] }{ \norm*{G_{h_n}^{Q, \beta_n}}[L^2(\S[d])] } + \\
        & \hspace{60pt} + \norm*{ \Pi_{I(\ell_n + \sigma, 2\Upsilon - 1)} g_{h_n}^{Q, \beta_n} - \frac{1}{C_{\sigma, \beta, \rho}} \bigg[ c_{\sigma, \beta, \rho} z_{\ell_n + \sigma}^{q} + \sum_{ \abs{k - \sigma} = 1}^{2\Upsilon - 1} \frac{ \beta_{q} + (-1)^{k} \rho \beta_{-q} }{ k - \sigma } z_{\ell_n + k}^{q} \bigg] }[L^2(\S[d])] \ .
    \end{align*}
    Thanks to Theorem \ref{Thm: Gh tails L2 norm estimates} and Theorem \ref{Thm: Gh L2 norm estimates scenario 3}, there exists $C > 1$ such that for all $\Upsilon \in \N$ there exists $\nu \in \N$ such that for all $n \geq \nu$,
    \begin{equation} \label{Eq: quasimode thm scnr3 aux1}
        \frac{ \norm*{ G_{h_n}^{Q, \beta_n} - \Pi_{I(\ell_n + \sigma, 2\Upsilon - 1)} G_{h_n}^{Q, \beta_n} }[L^2(\S[d])] }{ \norm*{G_{h_n}^{Q, \beta_n}}[L^2(\S[d])] } \leq C \Upsilon^{-\frac{1}{2}} \ .
    \end{equation}

    For the second term, we first split via the triangular inequality into the indices $k = \sigma$ and $\abs{k} \leq \Upsilon$, $k \neq \sigma$:
    \[
    \norm*{ \Pi_{I(\ell_n + \sigma, 2\Upsilon - 1)} g_{h_n}^{Q, \beta_n} - \frac{1}{C_{\sigma, \beta, \rho}} \bigg[ c_{\sigma, \beta, \rho} z_{\ell_n + \sigma}^{q} + \sum_{\abs{k - \sigma} = 1}^{2\Upsilon - 1} \frac{ \beta_q + (-1)^k \rho \beta_{-q} }{ k - \sigma } z_{\ell_n + k}^{q} \bigg] }[L^2(\S[d])] \leq \Sigma_{\sigma} + \Sigma_{\Upsilon} \ ,
    \]
    where
    \begin{align*}
        \Sigma_{\sigma} & \coloneqq \norm*{ \frac{1}{\norm*{G_{h_n}^{Q, \beta_n}}[L^2(\S[d])]} \frac{ \beta_{n,q} + (-1)^{\sigma} \rho \beta_{n,-q}}{ \lambda_{\ell_n + \sigma}^2 - h_n^{-2} } Z_{\ell_n + \sigma}^{q} - \frac{1}{C_{\sigma, \beta, \rho}} c_{\sigma, \beta, \rho} z_{\ell_n + \sigma}^{q} }[L^2(\S[d])] \ , \\
        \Sigma_{\Upsilon} & \coloneqq \norm*{ \frac{1}{\norm*{G_{h_n}^{Q, \beta_n}}[L^2(\S[d])]} \sum_{ \abs{k - \sigma} = 1}^{2 \Upsilon - 1} \frac{\beta_{n,q} + (-1)^{k} \rho \beta_{n,-q}}{ \lambda_{\ell_n + k}^2 - h_n^{-2} } Z_{\ell_n + k}^{q} - \frac{1}{C_{\sigma, \beta, \rho}} \sum_{\abs{k - \sigma} = 1}^{2\Upsilon - 1} \frac{ \beta_q + (-1)^k \rho \beta_{-q} }{ k - \sigma } z_{\ell_n + k}^{q} }[L^2(\S[d])] \ .
    \end{align*}

    Adding and subtracting
    \[
    \frac{1}{C_{\sigma, \beta, \rho}} \frac{\beta_{n,q} + (-1)^{\sigma} \rho \beta_{n,-q}}{ \sigma - \sigma_n } z_{\ell_n + \sigma}^{q} \ ,
    \]
    into $\Sigma_{\sigma}$, and then using triangular inequality, we get
    \begin{align*}
        \Sigma_{\sigma} & \leq \abs*{\beta_{n,q} + (-1)^{\sigma} \rho \beta_{n,-q}} \abs*{ \frac{\norm*{Z_{\ell_n + \sigma}^{q}}[L^2(\S[d])] }{ \norm*{G_{h_n}^{Q, \beta_n}}[L^2(\S[d])] (\lambda_{\ell_n + \sigma}^2 - h_n^{-2}) } - \frac{1}{C_{\sigma, \beta, \rho}} \frac{1}{\sigma - \sigma_n} } \\
        & \qquad + \frac{1}{C_{\sigma, \beta, \rho}} \abs[\bigg]{ \frac{\beta_{n,q} + (-1)^{\sigma} \rho \beta_{n,-q}}{ \sigma - \sigma_n } - C_{\sigma, \beta, \rho} } \ .
    \end{align*}

    \begin{Lemma} \label{Lemma: approx of quotient quasimode ksigma scenario 3}
        There exists $C > 1$ and $\nu \in \N$ such that for all $n \geq \nu$,
        \begin{multline*}
            \abs*{\beta_{n,q} + (-1)^{\sigma} \rho \beta_{n,-q}} \abs*{ \frac{ \norm*{Z_{\ell_n + \sigma}^{q}}[L^2(\S[d])] }{ \norm*{G_{h_n}^{Q, \beta_n} }[L^2(\S[d])] (\lambda_{\ell_n + \sigma}^2 - h_n^{-2}) } - \frac{1}{C_{\sigma, \beta, \rho}} \frac{1}{k - \sigma} } \leq \\
            \leq C \bigg[ h_n + \abs{\sigma_n - \sigma} + \abs{\beta_n - \beta} + \abs[\bigg]{ c_{\sigma, \beta, \rho} - \frac{\beta_{n,q} + (-1)^{\sigma} \rho \beta_{n,-q}}{ \sigma - \sigma_n } } + \Upsilon^{-1} \bigg] \ .
        \end{multline*}
    \end{Lemma}
    \begin{proof}[Proof of \texorpdfstring{Lemma \ref{Lemma: approx of quotient quasimode ksigma scenario 3}}{Proof of Lemma approx quotient scn3}]
        We add and subtract
        \[
        \frac{1}{\norm*{G_{h_n}^{Q, \beta_n} }[L^2(\S[d])]} \bigg( \frac{(h_n)^{3-d}}{2(d-1) \vol(\S[d])} \bigg)^{\frac{1}{2}} \frac{\beta_{n, q} + (-1)^{\sigma} \rho \beta_{n, -q} }{\sigma - \sigma_n} \ ,
        \]
        and then apply triangular inequality. 

        On the one hand, thanks to Lemma \ref{Lemma: asymp exp of quotient(ln+k) |k|<Upsilon}, with $\Upsilon = 1$ and $k = \sigma$, and Theorem \ref{Thm: Gh L2 norm estimates scenario 3}, and since \eqref{Eq: def csigmabeta limit} holds, there exists $C > 1$ and $\nu \in \N$ such that for all $n \geq \nu$,
        \begin{multline*}
        \frac{ \abs*{\beta_{n,q} + (-1)^{\sigma} \rho \beta_{n,-q}} }{\norm*{G_{h_n}^{Q, \beta_n} }[L^2(\S[d])]} \abs[\Bigg]{ \frac{ \norm*{Z_{\ell_n + \sigma}^{q}}[L^2(\S[d])] }{ \lambda_{\ell_n + \sigma}^2 - h_n^{-2} } - \bigg( \frac{(h_n)^{3-d}}{2(d-1) \vol(\S[d])} \bigg)^{\frac{1}{2}} \frac{1}{\sigma - \sigma_n}  } \leq \\
        \leq C \frac{\abs*{\beta_{n,q} + (-1)^{\sigma} \rho \beta_{n,-q}} }{ \abs{\sigma - \sigma_n} } h_n \leq C h_n \ .
        \end{multline*}
        
        On the other hand, Theorem \ref{Thm: Gh L2 norm estimates scenario 3} provides some $C > 1$ such that for all $\Upsilon \in \N$ there exists $\nu \in \N$ such that for all $n \geq \nu$,
        \begin{multline*}
            \bigg( \frac{(h_n)^{3-d}}{2(d-1) \vol(\S[d])} \bigg)^{\frac{1}{2}} \frac{ \abs*{\beta_{n,q} + (-1)^{\sigma} \rho \beta_{n,-q}} }{\abs{\sigma - \sigma_n }} \abs*{ \frac{1}{\norm*{G_{h_n}^{Q, \beta_n} }[L^2(\S[d])]} - \bigg( \frac{(h_n)^{3-d}}{2(d-1) \vol(\S[d])} \bigg)^{-\frac{1}{2}} \frac{1}{C_{\sigma, \beta, \rho}} } \leq \\
            \leq C \Big[ \Upsilon h_n + \abs{\sigma_n - \sigma} + \abs{\beta_n - \beta} + \abs[\bigg]{ c_{\sigma, \beta, \rho} - \frac{\beta_{n,q} + (-1)^{\sigma} \rho \beta_{n,-q}}{ \sigma - \sigma_n } } + \Upsilon^{-1} \Big] \ ,
        \end{multline*}
        where we used that \eqref{Eq: def csigmabeta limit} holds.
    \end{proof}
    Therefore, thanks to Lemma \ref{Lemma: approx of quotient quasimode ksigma scenario 3} and using that \eqref{Eq: def csigmabeta limit} holds, we infer that
    \begin{equation} \label{Eq: quasimode thm scnr3 aux2}
        \Sigma_{\sigma} \leq C \bigg[ h_n + \abs{\sigma_n - \sigma} + \abs{\beta_n - \beta} + \abs[\bigg]{ c_{\sigma, \beta, \rho} - \frac{\beta_{n,q} + (-1)^{\sigma} \rho \beta_{n,-q}}{ \sigma - \sigma_n } } + \Upsilon^{-1} \bigg] \ .
    \end{equation}

    On the other hand, for $\Sigma_{\Upsilon}$, we add and subtract
    \[
    \frac{1}{C_{\sigma, \beta, \rho}} \sum_{ \abs{k - \sigma} = 1}^{2\Upsilon -1} \frac{\beta_{n, q} + (-1)^{k} \rho \beta_{n,-q}}{ k - \sigma } z_{\ell_n + k}^{q} \ ,
    \]
    and then apply triangular inequality and Parseval's identity, obtaining
    \begin{align*}
        \Sigma_{\Upsilon} & \leq \left( \sum_{ \abs{k - \sigma} = 1}^{2\Upsilon - 1} \abs*{\beta_{n,q} + (-1)^{\sigma} \rho \beta_{n,-q}}^2 \abs*{ \frac{ \norm*{Z_{\ell_n + k}^{q}}[L^2(\S[d])] }{ \norm*{G_{h_n}^{Q, \beta_n} }[L^2(\S[d])] (\lambda_{\ell_n + k}^2 - h_n^{-2}) } - \frac{1}{C_{\sigma, \beta, \rho}} \frac{1}{k - \sigma} }^2  \right)^{\frac{1}{2}} \\
        & \qquad + \frac{1}{C_{\sigma, \beta, \rho}} \left( \sum_{\abs{k - \sigma} = 1}^{2\Upsilon - 1} \frac{ \abs{ (\beta_{n, q} - \beta_{q}) + (-1)^k \rho (\beta_{n, -q} - \beta_{-q}) }^2 }{(k - \sigma)^2} \right)^{\frac{1}{2}}
    \end{align*}

    \begin{Lemma} \label{Lemma: approx of quotient quasimode scenario 3}
        There exists $C > 1$ such that for all $\Upsilon \in \N$ there exists $\nu \in \N$ such that for all $n \geq \nu$ and all $1 \leq \abs{k - \sigma} \leq 2\Upsilon - 1$,
        \begin{multline*}
            \abs*{ \frac{ \norm*{Z_{\ell_n + k}^{q}}[L^2(\S[d])] }{ \norm*{G_{h_n}^{Q, \beta_n} }[L^2(\S[d])] (\lambda_{\ell_n + k}^2 - h_n^{-2}) } - \frac{1}{C_{\sigma, \beta, \rho}} \frac{1}{k - \sigma} } \leq \\
            \leq \frac{C}{\abs{k - \sigma}} \bigg[ \Upsilon h_n + \abs{\sigma_n - \sigma} + \abs{\beta_n - \beta} + \abs[\bigg]{ c_{\sigma, \beta, \rho} - \frac{\beta_{n,q} + (-1)^{\sigma} \rho \beta_{n,-q}}{ \sigma - \sigma_n } } + \Upsilon^{-1} \bigg] \ .
        \end{multline*}
    \end{Lemma}
    \begin{proof}[Proof of \texorpdfstring{Lemma \ref{Lemma: approx of quotient quasimode scenario 3}}{Proof of Lemma approx quotient scn3}]
        Let $k \in \Z$ such that $1 \leq \abs{k - \sigma} \leq 2\Upsilon - 1$. We add and subtract
        \[
        \frac{1}{\norm*{G_{h_n}^{Q, \beta_n} }[L^2(\S[d])]} \bigg( \frac{(h_n)^{3-d}}{2(d-1) \vol(\S[d])} \bigg)^{\frac{1}{2}} \frac{1}{k - \sigma} \ ,
        \]
        and then apply triangular inequality. 

        On the one hand, thanks to Lemma \ref{Lemma: asymp exp of quotient(ln+k) |k|<Upsilon} and Theorem \ref{Thm: Gh L2 norm estimates scenario 3}, there exists $C > 1$ such that for all $\Upsilon \in \N$ there exists $\nu \in \N$ such that for all $n \geq \nu$ and all $1 \leq \abs{k - \sigma} \leq 2\Upsilon - 1$,
        \[
        \frac{1}{\norm*{G_{h_n}^{Q, \beta_n} }[L^2(\S[d])]} \abs[\Bigg]{ \frac{ \norm*{Z_{\ell_n + k}^{q}}[L^2(\S[d])] }{ \lambda_{\ell_n + k}^2 - h_n^{-2} } - \bigg( \frac{(h_n)^{3-d}}{2(d-1) \vol(\S[d])} \bigg)^{\frac{1}{2}} \frac{1}{k - \sigma}  } \leq \frac{C}{ \abs{k - \sigma} } \Big[ \Upsilon h_n + \abs{\sigma_n - \sigma} \Big] \ .
        \]
        
        On the other hand, Theorem \ref{Thm: Gh L2 norm estimates scenario 3} provides some $C > 1$ such that for all $\Upsilon \in \N$ there exists $\nu \in \N$ such that for all $n \geq \nu$,
        \begin{align*}
            & \bigg( \frac{(h_n)^{3-d}}{2(d-1) \vol(\S[d])} \bigg)^{\frac{1}{2}} \frac{1}{\abs{k - \sigma}} \abs*{ \frac{1}{\norm*{G_{h_n}^{Q, \beta_n} }[L^2(\S[d])]} - \bigg( \frac{(h_n)^{3-d}}{2(d-1) \vol(\S[d])} \bigg)^{-\frac{1}{2}} \frac{1}{C_{\sigma, \beta, \rho}} } \leq \\
            & \hspace{50pt} \leq \frac{C}{\abs{k - \sigma}} \Big[ \Upsilon h_n + \abs{\sigma_n - \sigma} + \abs{\beta_n - \beta} + \abs[\bigg]{ c_{\sigma, \beta, \rho} - \frac{\beta_{n,q} + (-1)^{\sigma} \rho \beta_{n,-q}}{ \sigma - \sigma_n } } + \Upsilon^{-1} \Big] \ . \qedhere
        \end{align*}
    \end{proof}
    Applying Lemma \ref{Lemma: approx of quotient quasimode scenario 3} and using that $\beta_n \to \beta$, there exists $C > 1$ such that for all $\Upsilon \in \N$ there exists $\nu \in \N$ such that for all $n \geq \nu$,
    \begin{equation} \label{Eq: quasimode thm scnr3 aux3}
        \begin{aligned}
            \Sigma_{\Upsilon} & \leq C \bigg[ \Upsilon h_n + \abs{\sigma_n - \sigma} + \abs{\beta_n - \beta} + \abs[\bigg]{ c_{\sigma, \beta, \rho} - \frac{\beta_{n,q} + (-1)^{\sigma} \rho \beta_{n,-q}}{ \sigma - \sigma_n } } + \Upsilon^{-1} \bigg] \bigg( \sum_{ \substack{ \abs{k} \leq \Upsilon \\ k \neq \sigma} } \frac{1}{ \abs{k - \sigma}^2} \bigg)^{\frac{1}{2}} \\
            & \leq C \bigg[ \Upsilon h_n + \abs{\sigma_n - \sigma} + \abs{\beta_n - \beta} + \abs[\bigg]{ c_{\sigma, \beta, \rho} - \frac{\beta_{n,q} + (-1)^{\sigma} \rho \beta_{n,-q}}{ \sigma - \sigma_n } } + \Upsilon^{-1} \bigg]
        \end{aligned}
    \end{equation}
    
    Combining \eqref{Eq: quasimode thm scnr3 aux1}, \eqref{Eq: quasimode thm scnr3 aux2}, and \eqref{Eq: quasimode thm scnr3 aux3}, we conclude the result.
\end{proof}

\subsection{Scenario 4}

\begin{Theorem} \label{Thm: Gh L2 norm estimates scenario 4}
    Let $(h_n)_{n \in \N} \subseteq (0,1)$ and $(\beta_{n})_{n \in \N} \subseteq \C[2]$, and for every $n \in \N$, let $\ell_n \in \N$ and $\sigma_n \in [0,1)$ as defined in \eqref{Eq: def elln and sigman}. Assume the following:
    \begin{enumerate}
        \item $h_n \to 0^+$ and $h_n^{-2} = \lambda_{\ell_n}^2$ for all $n \in \N$, thus $\sigma_n = 0$ for all $n \in \N$.
        \item $\ell_n$ is even for all $n \in \N$, or $\ell_n$ is odd for all $n \in \N$. Set $\rho \coloneqq (-1)^{\ell_n}$.
        \item $\beta_{n,-q} + \rho \beta_{n,q} = 0$ for all $n \in \N$, thus $\beta_{-q} + \rho \beta_{q} = 0$.
    \end{enumerate}
    Set
    \begin{equation} \label{Eq: Dbeta definition}
        D_{\beta} \coloneqq \bigg( \sum_{\substack{ k \in \Z \\ k \neq 0} } \frac{ \abs{\beta_{q} + (-1)^{k} \rho \beta_{-q} }^2 }{ k^2 } \bigg)^{\frac{1}{2}} \ .
    \end{equation}
    There exists $C > 1$ such that for every $\Upsilon \in \N$, there exists $\nu \in \N$ such that for all $n \geq \nu$,
    \begin{equation} \label{Eq: norm mainterm scenario 4}
        \abs*{ \norm*{ \Pi_{I(\ell_n, 2\Upsilon - 1)} G_{h_n}^{Q, \beta_n} }[L^2(M)]^2 - \frac{(D_{\beta})^2}{2(d-1) \vol(\S[d])}(h_n)^{3-d} } \leq C (h_n)^{3-d} \bigg[ \Upsilon h_n + \Upsilon^{-1} \bigg] \ .
    \end{equation}
\end{Theorem}

\begin{proof}
    Since $h_n^{-2} = \lambda_{\ell_n}^2$, $\ell_n$ is even, and $\beta_{n,-q} + \rho \beta_{n,q} = 0$ for all $n \in \N$, we note that
    \[
    \norm*{ \Pi_{I(\ell_n, 2\Upsilon - 1)} G_{h_n}^{Q, \beta_n} }[L^2(M)]^2 = \sum_{\abs{k} = 1}^{2\Upsilon - 1} \frac{ \norm*{Z_{\ell_n + k}^{q}}[L^2(\S[d])]^2 }{ \abs{\lambda_{\ell_n + k}^2 - \lambda_{\ell_n}^2 }^2 } \ .
    \]
    Adding and subtracting
    \[
    \frac{(h_n)^{3-d}}{2(d-1) \vol(\S[d])} \sum_{\abs{k} = 1}^{2\Upsilon - 1} \frac{ \abs{\beta_{q} + (-1)^{k} \rho \beta_{-q} }^2}{k^2} \ ,
    \]
    triangular inequality gives us
    \begin{align*}
        & \abs*{ \norm*{ \Pi_{I(\ell_n, 2\Upsilon - 1)} G_{h_n}^{Q, \beta_n} }[L^2(M)]^2 - \frac{(D_{\beta})^2}{2(d-1) \vol(\S[d])}(h_n)^{3-d} } \\
        & \hspace{100pt} \leq \sum_{\abs{k} = 1}^{2\Upsilon - 1} \abs*{\beta_{q} + (-1)^{k} \rho \beta_{-q}}^2 \abs*{ \frac{ \norm*{Z_{\ell_n + k}^{q}}[L^2(\S[d])]^2 }{ \abs{\lambda_{\ell_n + k}^2 - \lambda_{\ell_n}^2 }^2 } - \frac{(h_n)^{3-d}}{2(d-1) \vol(\S[d])} \frac{1}{k} } \\
        & \hspace{120pt} + \frac{(h_n)^{3-d}}{2(d-1) \vol(\S[d])} \bigg[ \sum_{k = -\ell_n}^{-2\Upsilon} + \sum_{k = 2\Upsilon}^{\infty} \bigg] \frac{\abs{\beta_{q} + (-1)^{k} \rho \beta_{q}}^2 }{\abs{k}^2} \ .
    \end{align*}
    Now, thanks to Lemma \ref{Lemma: asymp exp of quotient(ln+k) |k|<Upsilon}, there exists $C > 1$ such that for every $\Upsilon \in \N$ there exists $\nu \in \N$ such that for every $n \geq \nu$,
    \begin{align*}
        \sum_{\abs{k} = 1}^{2\Upsilon - 1} \abs{\beta_{q} + (-1)^{k} \rho \beta_{-q} }^2 \abs*{ \frac{ \norm*{Z_{\ell_n + k}^{q}}[L^2(\S[d])]^2 }{ \abs{\lambda_{\ell_n + k}^2 - \lambda_{\ell_n}^2 }^2 } - \frac{(h_n)^{3-d}}{2(d-1) \vol(\S[d])} \frac{1}{k^2} } & \leq C \Big[ \Upsilon h_n \Big] \sum_{\abs{k} = 1}^{2\Upsilon - 1} \frac{(h_n)^{3-d}}{ \abs{k}^2 } \\
        & \leq C (h_n)^{3-d} \Big[\Upsilon h_n \Big] \ .
    \end{align*}
    Using $\sum_{k = 2\Upsilon} \frac{1}{\abs{k}^2} \leq \Upsilon^{-1}$, we infer
    \[
    \frac{(h_n)^{3-d}}{2(d-1) \vol(\S[d])} \bigg[ \sum_{k = -\ell_n}^{-2\Upsilon} + \sum_{k = 2\Upsilon}^{\infty} \bigg] \frac{\abs{\beta_{q} + (-1)^{k} \rho \beta_{q}}^2 }{\abs{k}^2} \leq C (h_n)^{3-d} \Upsilon^{-1} \ .
    \]
    Thus, the result follows.
\end{proof}

\begin{Theorem} \label{Thm: gh quasimode scenario 4}
    Let $(h_n)_{n \in \N} \subseteq (0,1)$ and $(\beta_{n})_{n \in \N} \subseteq \C[2]$, and for every $n \in \N$, let $\ell_n \in \N$ and $\sigma_n \in [0,1)$ as defined in \eqref{Eq: def elln and sigman}. Assume the following:
    \begin{enumerate}
        \item $h_n \to 0^+$ and $h_n^{-2} = \lambda_{\ell_n}^2$ for all $n \in \N$, thus $\sigma_n = 0$ for all $n \in \N$.
        \item $\ell_n$ is even for all $n \in \N$, or $\ell_n$ is odd for all $n \in \N$. Set $\rho \coloneqq (-1)^{\ell_n}$.
        \item $\beta_{n,-q} + \rho \beta_{n,q} = 0$ for all $n \in \N$, thus $\beta_{-q} + \rho \beta_{q} = 0$.
    \end{enumerate}
    Set
    \begin{equation*}
        D_{\beta} \coloneqq \bigg( \sum_{\substack{ k \in \Z \\ k \neq 0} } \frac{ \abs{\beta_{q} + (-1)^{k} \rho \beta_{-q} }^2 }{ k^2 } \bigg)^{\frac{1}{2}} \ .
    \end{equation*}
    There exists $C > 1$ such that for every $\Upsilon \in \N$, there exists $\nu \in \N$ such that for all $n \geq \nu$,
    \begin{equation} \label{Eq: ghQbeta quasimode scenario 4}
        \norm*{ g_{h_n}^{Q, \beta_{n}} - \frac{1}{D_{\beta}} \sum_{\abs{k} = 1}^{2\Upsilon - 1} \frac{\beta_{q} + (-1)^{k} \rho \beta_{-q}}{ k } z_{\ell_n + k}^{q} }[L^2(M)] \leq C \Big[ \Upsilon h_n^{-1} + \Upsilon^{-\frac{1}{2}} \Big] \ .
    \end{equation}
\end{Theorem}

\begin{proof}
    Triangular inequality implies that
    \begin{multline*}
        \norm*{ g_{h_n}^{Q, \beta_{n}} - \frac{1}{D_{\beta}} \sum_{\abs{k} = 1}^{2\Upsilon - 1} \frac{\beta_{q} + (-1)^{k} \rho \beta_{-q}}{ k } z_{\ell_n + k}^{q} }[L^2(M)] \leq \\
        \leq \frac{ \norm*{ G_{h_n}^{Q, \beta_{n}} - \Pi_{I(\ell_n, 2\Upsilon - 1)} G_{h_n}^{Q, \beta_{n}} }[L^2(\S[d])] }{ \norm*{ G_{h_n}^{Q, \beta_n}}[L^2(\S[d])] } + \norm*{ \Pi_{I(\ell_n, 2\Upsilon - 1)} G_{h_n}^{Q, \beta_{n}} - \frac{1}{D_{\beta}} \sum_{\abs{k} = 1}^{2\Upsilon - 1} \frac{\beta_{q} + (-1)^{k} \rho \beta_{-q}}{ k } z_{\ell_n + k}^{q} }[L^2(\S[d])] \ .
    \end{multline*}
    Thanks to Theorem \ref{Thm: Gh tails L2 norm estimates} and Theorem \ref{Thm: Gh L2 norm estimates scenario 4}, there exists $C > 1$ such that for all $\Upsilon \in \N$ there exists $\nu \in \N$ such that for all $n \geq \nu$,
    \begin{equation} \label{Eq: quasimode thm scnr4 aux1}
        \frac{ \norm*{ G_{h_n}^{Q, \beta_{n}} - \Pi_{I(\ell_n, 2\Upsilon - 1)} G_{h_n}^{Q, \beta_{n}} }[L^2(\S[d])] }{ \norm*{ G_{h_n}^{Q, \beta_n}}[L^2(\S[d])] } \leq C \Upsilon^{-\frac{1}{2}} \ .
    \end{equation}
    On the other hand
    \begin{multline*}
    \norm*{ \Pi_{I(\ell_n, 2\Upsilon - 1)} G_{h_n}^{Q, \beta_{n}} - \frac{1}{D_{\beta}} \sum_{\abs{k} = 1}^{2\Upsilon - 1} \frac{\beta_{q} + (-1)^{k} \rho \beta_{-q}}{ k } z_{\ell_n + k}^{q} }[L^2(\S[d])] = \\
    = \left( \sum_{\abs{k} = 1}^{2\Upsilon - 1} \abs*{\beta_{q} + (-1)^{k} \rho \beta_{-q}}^2 \abs*{ \frac{1}{\norm*{G_{h_n}^{Q, \beta_n}}[L^2(\S[d])]} \frac{\norm*{Z_{\ell_n + k}^{q}}[L^2(\S[d])]}{ \lambda_{\ell_n + k}^2 - \lambda_{2\ell_n}^2 } - \frac{1}{D_{\beta}} \frac{1}{k} }^2 \right)^{\frac{1}{2}}
    \end{multline*}
    One then uses the following lemma, that is consequence of Lemma \ref{Lemma: asymp exp of quotient(ln+k) |k|<Upsilon} and Theorem \ref{Thm: Gh L2 norm estimates scenario 4}.

    \begin{Lemma} \label{Lemma: approx of quotient quasimode scenario 4}
        There exists $C > 1$ such that for all $\Upsilon \in \N$ there exists $\nu \in \N$ such that for all $n \geq \nu$ and all $0 \leq \abs{k} \leq 2\Upsilon - 1$,
        \begin{equation*}
            \abs*{ \frac{1}{\norm*{G_{h_n}^{Q, \beta_n}}[L^2(\S[d])]} \frac{\norm*{Z_{\ell_n + k}^{q}}[L^2(\S[d])]}{ \lambda_{\ell_n + k}^2 - \lambda_{\ell_n}^2 } - \frac{1}{D_{\beta}} \frac{1}{k} } \leq \frac{C}{\abs{k}} \Big[ \Upsilon h_n + \Upsilon^{-1} \Big] \ .
        \end{equation*}
    \end{Lemma}

    Thanks to Lemma \ref{Lemma: approx of quotient quasimode scenario 4}, there exists $C > 1$ such that for every $\Upsilon \in \N$ there exists $\nu \in \N$ such that for every $n \geq \nu$,
    \begin{equation} \label{Eq: quasimode thm scnr4 aux2}
        \norm*{ \Pi_{I(\ell_n, 2\Upsilon - 1)} G_{h_n}^{Q, \beta_{n}} - \frac{1}{D_{\beta}} \sum_{\abs{k} = 1}^{2\Upsilon - 1} \frac{\beta_{q} + (-1)^{k} \rho \beta_{-q}}{k} z_{\ell_n + k}^{q} }[L^2(\S[d])] \leq C \Big[ \Upsilon h_n + \Upsilon^{-1} \Big] \ .
    \end{equation}
    \eqref{Eq: quasimode thm scnr4 aux1} and \eqref{Eq: quasimode thm scnr4 aux2} together imply the result.
\end{proof}

\section{Semiclassical defect measures for Green's functions}
\label{Sec: SDM Green functions}

In Section \ref{Sec: Quasimode Green function} we introduce a large parameter $\Upsilon \in \N$ that allowed us to find arbitrarily good quasimodes of linear combination of Green's functions, $G_{h_n}^{Q, \beta_n}$, on antipodal points, $q$ and $-q$, in the high-energy regime, $h_n \searrow 0$. These quasimodes are linear combinations of normalized zonal harmonics on $q$ on a large window centered on $\ell_n$ and of width $\Upsilon$, and the coefficients of these linear combination depend on the accumulation points of the vector of weights $\beta_{n} \in \C[2]$ and the \emph{intensity} $\sigma_n \in [0,1]$ (see Theorem \ref{Thm: gh quasimode scenario 1} and Theorem \ref{Thm: gh quasimode scenario 3}).

We start this section proving that for such quasimodes with $\Upsilon \in \N$ fixed, we may find a semiclassical pseudodifferential operator that capture such linear combination of zonal harmonics on $q$ in the high-energy regime. We choose a Weyl quantization defined through charts $\Opp_{h}$.

We introduce necessary notation for this section. For a fixed point $q \in \S[d]$, let $r \coloneqq d(x, q)$ for $x \in \S[d]$, and define the smooth symbols on $T^*\S[d]$
\begin{equation} \label{Eq: definition symbols kappaq varsigmaq}
    \kappa^{q}(x, \xi) \coloneqq \cos r \ , \qquad \varsigma^{q}(x, \xi) \coloneqq \xi (\sin r \cdot \partial_{r}) \ .
\end{equation}
In addition, for a square-summable sequence $X \in \ell_2(\Z)$ and any finite subset $S \subseteq \Z$, set
\[
S_{-} \coloneqq \{k \in S \colon k \leq -1 \} \ , \qquad S_{+} \coloneqq \{k \in S \colon k \geq 1\} \ ,
\]
and let $\Gamma^{q, X, S}$ be the following smooth symbol on $T^*\S[2]$
\begin{equation} \label{Eq:definition-symbol-GammaqXS}
    \Gamma^{q, X, S} \coloneqq \sum_{k \in S_{-}} X_{k} (\kappa^{q} - i \varsigma^{q})^{-k} + \charf_{S}(0) X_{0} + \sum_{k \in S_{+}} X_{k} (\kappa^{q} + i \varsigma^{q})^{k} \ .
\end{equation}

\begin{Theorem} \label{Thm: approx of uqUpsilonxn by PsDO on zln}
    Let $q \in \S[d]$, and let $(h_n)_{n \in \N} \subseteq (0,1)$, $(\ell_n)_{n \in \N} \subseteq \N$ be sequences such that
    \begin{equation*}
        \lim_{n \to \infty} h_n = 0 \ , \qquad \exists \delta > 0 \ \text{s.t.} \quad \delta \leq h_n \lambda_{\ell_n} \leq \delta^{-1} \qquad \forall \, n \in \N \ .
    \end{equation*}
    Let $(x_n)_{n \in \N} \subseteq \ell_2(\Z)$ be a bounded sequence and let $X \in \ell_2(\Z)$ fixed. For any finite subset $S \subseteq \Z$, let $\charf_{S}$ be the characteristic function of $S$ on $\Z$. There exist $C(S) > 1$ and $N = N(S) \in \N$ such that for all $n \geq N$,
    \begin{equation} \label{Eq:symbol-lincomb-zonalharmonics}
        \norm*{ \sum_{k \in S} x_{n,k} z_{\ell_n + k}^{q} - \Op[h_n]{ \Gamma^{q, X, S} } z_{\rho_n}^{q} }[L^2(\S[d])] \leq \norm*{ \charf_{S} (x_n - X)}[\ell_2(\Z)] + \norm*{\charf_{S}X}[\ell_1]  C(S) h_n \ .
    \end{equation}
\end{Theorem}

\begin{proof}
    We start defining
    \begin{align*}
        u_{n}^{q, x_n, S} & \coloneqq \sum_{k \in S} x_{n,k} z_{\ell_n + k}^{q} \ , \\
        u_{n}^{q, X, S} & \coloneqq \sum_{k \in S} X_{k} z_{\ell_n + k}^{q} \ .
    \end{align*}
    Since the following holds
    \begin{equation*}
        \norm[\Big]{ u_{n}^{q, x_n, S} -  u_{n}^{q, X, S} }[L^2(\S[d])] \leq \bigg( \sum_{x \in S} \abs*{ x_{n,k} - X_k }^2 \bigg)^{\frac{1}{2}} = \norm*{ \charf_{S} (x_n - X)}[\ell_2(\Z)] \ ,
    \end{equation*}
    we have
    \begin{equation} \label{Eq: uqUpsilonxn - uqUpsilonX norm}
        \norm*{ u_{n}^{q, x_n, S} - \Op[h_n]{\Gamma^{q, X, S}} z_{\ell_n}^{q} }[L^2(\S[d])] \leq \norm*{ \charf_{S} (x_n - X)}[\ell_2(\Z)]  + \norm*{u_{n}^{q, X, S} - \Op[h_n]{\Gamma^{q, X, S}} z_{\ell_n}^{q} }[L^2(\S[d])] \ .
    \end{equation}
    Proving an estimate for the second term requires of several lemma; it is eventually achieved on \eqref{Eq: Proposition symbol of lincomb end}, in page \pageref{Eq: Proposition symbol of lincomb end}.

    From Equation \ref{Eq: identity zonal harm - Gegen poly} we read that for every $\ell \in \N$,
    \[
    z_{\ell}^{q}(x) = \frac{1}{\norm{Z_{\ell}^{q}}} Z_{\ell}^{q} =  \sqrt{ \frac{m_{\ell}}{\vol(\S[d])} } \big[ C_{\ell}^{(\frac{d-1}{2})} (1) \big]^{-1} C_{\ell}^{(\frac{d-1}{2})} (\cos r)
    \]
    where $C_{\ell}^{(\alpha)}$ is the $\ell$-th Gegenbauer polynomial and $r = d(x, q)$ is the geodesic distance from $q$ to $x$. Combining this identity with Proposition \ref{Prop: ladder operators for Gegenbauer poly} we get
    
    \begin{Prop} \label{Lemma: ladder op for zonal harm}
        The operators
        \begin{equation} \label{Eq: crea and anni op for zonal harm}
            K_{\ell}^{q, \pm} \coloneqq \sqrt{ \frac{m_{\ell \pm 1}}{m_{\ell}} } \Big[ \cos r \pm \frac{\sin r}{(\ell + \frac{d-1}{2}) \pm \frac{d-1}{2} } \partial_r \Big] \ ,
        \end{equation}
        are creation and annihilation operators respectively for the normalized zonal harmonics on $q \in \S[d]$, that is,
        \begin{equation} \label{Eq: cre and anni identities for zonal harm}
            z_{\ell+1}^{q} = K_{\ell}^{q, +} z_{\ell}^{q} \ , \qquad z_{\ell-1}^{q} = K_{\ell}^{q, -}z_{\ell}^{q} \ .
        \end{equation}
    \end{Prop}
    
    With these ladder operators at hand, we may define
    \begin{equation*}
        \begin{aligned}
        \mathcal{K}_{n}^{q, X, S, -} & \coloneqq \sum_{k \in S_{-}} X_{k} \Big[ K_{\ell_n - (k - 1)}^{q,-} \dots K_{\ell_n}^{q,-} \Big] \ , \\
        \mathcal{K}_{n}^{q, X, S, +} & \coloneqq \sum_{k \in S_{+}} X_{k} \Big[ K_{\ell_n + (k - 1)}^{q,+} \dots K_{\ell_n}^{q,+} \Big] \ ,
        \end{aligned}
    \end{equation*}
    and thus rewrite $u_{n}^{q, X, S}$ as follows
    \begin{equation} \label{Eq: uqUpsilonX into ladder operators}
        u_{n}^{q, X, S} = \Big[ \mathcal{K}_{n}^{q, X, S, -} \, z_{\ell_n}^{q} + X_{0} z_{\ell_n}^{q} + \mathcal{K}_{n}^{q, X, S, +} \, z_{\ell_n}^{q} \Big] \ .
    \end{equation}
    
    We see that for a fixed $\ell \in \N$ the operator $K_{\ell}^{q,\pm}$ is a pseudodifferential operator; however, in the semiclassical regime $h_n \searrow 0$, its symbol depends on the sequences $(h_n)_{n \in \N}$ and $(\ell_n)_{n \in \N}$. More specifically, we will show that the symbol of $K_{\ell_n}^{q, \pm}$ is $\sigma(K_{\ell_n}^{q, \pm}) = (\alpha \pm i \beta)$ provided $\ell_n h_n \sim 1$; however, the size of the error term will depend on the distance $\abs{h_n - \ell_n^{-1}}$. In order to estimate the error term, we introduce a scale of semiclassical Sobolev spaces for $s \in \R$,
    \[
    H_h^s(\S[d]) = \{ u \in \Dist(\S[d]) \colon \, (h^2 \Lap {}+ 1)^{s/2} u \in L^2(\S[d]) \} \ , \qquad \norm{u}[H_h^s(\S[d])] = \norm{(h^2 \Lap {}+ 1)^{s/2} u}[L^2(\S[d])] \ ,
    \]
    and three technical lemmas.

    \begin{Lemma} \label{Lemma: ladder operators are hPsDO}
        For any $\ell \in \N$, any $h > 0$, and any choice of $\pm$, there exists a pseudodifferential operator $R_{l, h}^{q, \pm}$ such that
        \begin{equation} \label{Eq: ladder operators are hPsDO}
            K_{\ell}^{q, \pm} = \Op[h]{\kappa^{q} \pm i \varsigma^{q}} + R_{l, h}^{q, \pm} \ ,
        \end{equation}
        and for every $s \in \R$, there exists a constant $C_{1, s} > 0$ such that
        \[
        \norm*{R_{l, h}^{q, \pm}}[H_h^{s} \to H_{h}^{s-1}] \leq C_{1, s} \bigg[ \tfrac{1}{\ell + \frac{d-1}{2} } + h + \abs[\bigg]{ \frac{1}{ (\ell + \frac{d - 1}{2}) \pm \frac{d-1}{2} } - h } \bigg] \ ,
        \]
    \end{Lemma}

    \begin{proof}
        From the definition of Weyl quantization, we have the following identities
        \begin{align*}
            \Op[h]{\kappa^{q}} & = \cos r \ ,\\
            \Op[h]{\varsigma^{q}} & = \frac{h}{i} (\sin r\, \partial_r) + R_h \ ,
        \end{align*}
        for a certain pseudodifferential operator $R_h$ such that $\norm{R_h}[H_h^s \to H_h^{s}] \leq C_{s} h$ for all $s \in \R$ and all $h > 0$. Hence,
        \begin{align*}
            K_{\ell}^{q, \pm} - \Op[h]{\kappa^{q} \pm i\varsigma^{q}} & = \Big[ \sqrt{ \frac{m_{\ell \pm 1}}{m_{\ell}} } - 1 \Big] \cos r \pm \Big[ \sqrt{ \frac{m_{\ell \pm 1}}{m_{\ell}} } \frac{1}{ (\ell + \frac{d - 1}{2}) \pm \frac{d-1}{2} } - h \Big] \sin r \, \partial_r + R_h \\
            & = \BigO_{H_h^{s} \to H_{h}^{s}} \Big( \tfrac{1}{\ell + \frac{d-1}{2} } + h \Big) + \BigO_{H_h^{s} \to H_{h}^{s - 1}} \bigg( \abs[\bigg]{ \frac{1}{ (\ell + \frac{d - 1}{2}) \pm \frac{d-1}{2} } - h } + \frac{h}{\ell + \frac{d-1}{2} }\bigg)
        \end{align*}
        where we used that
        \begin{equation} \label{Eq: estimate on quotient of mellpm1}
            \abs*{ \frac{m_{\ell \pm 1}}{m_{\ell}} - 1} \leq C \frac{1}{\ell + \tfrac{d-1}{2} } \ .
        \end{equation}
    \end{proof}
    
    \begin{Lemma} \label{Lemma: ladder operators are continuous on Hhs}
        For any $s \in \R$, there exists positive constants $C_{2, s}, C_{3, s} > 0$ such that
        \begin{align*}
            \norm{\Op[h]{\kappa^{q} \pm i \varsigma^{q}}}[H_{h}^{s} \to H_{h}^{s-1}] & \leq C_{2, s} \\
            \norm*{\Op[h]{\kappa^{q} \pm i \varsigma^{q}}^k - \Op[h]{(\kappa^{q} \pm i \varsigma^{q})^k}}[H_{h}^{s-1} \to L^2] & \leq C_{3, s} h
        \end{align*}
    \end{Lemma}

    \begin{proof}
        For the first inequality, thanks to the symbolic calculus and Calderon-Vaillancourt theorem, for any $u \in H^{s}(\S[d])$
        \begin{align*}
            \norm*{ \Op[h]{\kappa^{q} \pm i\varsigma^{q}} u }[H_{h}^{s-1}(\S[d])] & = \norm*{ (h^2 \Lap {}+ 1)^{\frac{s-1}{2}} \Op[h]{\kappa^{q} \pm i \varsigma^{q}} u }[L^2(\S[d])] \\
            & = \norm*{ \Op[h]{\kappa^{q} \pm i \varsigma^{q}} (h^2 \Lap {}+ 1)^{\frac{s-1}{2}}  u }[L^2(\S[d])] + \BigO_{L^2(\S[d])}(h) \\
            & \leq \norm*{  \Op[h]{(\abs{\xi}^2 + 1)^{-\frac{1}{2}} (\kappa^{q} \pm i \varsigma^{q}) } }[L^2 \to L^2] \norm*{ (h^2 \Lap {}+ 1)^{\frac{s}{2}} u }[L^2(\S[d])] + \BigO_{L^2(\S[d])}(h) \\
            & \leq C \norm{u}[H_h^{s}(\S[d])] + \BigO(h) \ .
        \end{align*}
        The second inequality is, once again, a consequence of symbolic calculus and Calderon-Vaillancourt theorem.
    \end{proof}

    \begin{Lemma} \label{Lemma: elliptic estimate for zonal harmonics}
        For every $s \in \R$, there exists $C_s^4 > 0$ such that
        \[
        \norm*{z_{\ell_n}^{q}}[H_{h_n}^s(\S[d])] \leq C_s^{4} \qquad \forall \ n \in \N \ .
        \]
    \end{Lemma}

    \begin{proof}
        This is a direct computation:
        \begin{align*}
            \norm*{z_{\ell_n}^{q}}[H_{h_n}^{s}] & = \norm*{[(h_n)^2 \Lap {}+ 1]^{s/2} z_{\ell_n}^{q} }[L^2(\S[2])] \\
            & = \abs*{(h_n)^2 \lambda_{\ell_n}^2 + 1}^{s/2} \leq C^{s/2}
        \end{align*}
        where we used that $h_n \lambda_{\ell_n} \leq \delta^{-1}$ for every $n \in \N$.
    \end{proof}

    We are finally in position to prove Theorem \ref{Thm: approx of uqUpsilonxn by PsDO on zln}. Let $S_{\max} \coloneqq \max \{\abs{k} \colon k \in S\}$, and let $C_{S} > 0$ such that
    \[
    C_{S} \geq \max_{j = 1, \dots, S_{\max} } \{C_{1, j}, C_{2, j}, C_{3, j}, C_{4, j}\} \ ,
    \]
    and, thanks to $\delta \leq h_n \lambda_{\ell_n}$ for all $n \in \N$, such that
    \begin{equation} \label{Eq: Prop 5.2 aux 10}
        C_{S} \geq h_n^{-1}\bigg[ \tfrac{1}{\ell_n + k + \frac{d-1}{2} } + h_n + \abs[\bigg]{ \frac{1}{ (\ell_n + k + \frac{d - 1}{2}) \pm \frac{d-1}{2} } - h_n } \bigg] \qquad \forall \, \abs{k} \leq S_{\max} \ , \quad \forall \, n \in \N \ .
    \end{equation}
    We study the operator with positive indices $\mathcal{K}_{n}^{q, X, S, +}$. Thanks to \eqref{Eq: ladder operators are hPsDO}, we may write the $k$-creation operator as a linear combination of semiclassical pseudodifferential operators,
    \begin{equation} \label{Eq: Prop 5.2 aux 11}
        K_{\ell_n + (k - 1)}^{q,+} \dots K_{\ell_n}^{q,+} = \sum_{\mathrm{len}(w) = k} \prod_{w} w \Big( \Op[h_n]{\kappa^{q} + i \varsigma^{q}}, R_{\bullet,h_n}^{q,+} \Big) \ .
    \end{equation}
    Thanks to Lemma \ref{Lemma: ladder operators are hPsDO} and \eqref{Eq: Prop 5.2 aux 10},
    \[
    \norm*{R_{\ell_n + k, h_n}^{q, +}}[H_{h_n}^{j} \to H_{h_n}^{j-1}] \leq (C_{S})^2 h_n \qquad \forall \, 1 \leq k, j \leq S_{\max} \ , \quad n \in \N \ ;
    \]
    and thanks to Lemma \ref{Lemma: ladder operators are continuous on Hhs},
    \[
    \norm*{\Op[h_n]{\kappa^{q} + i \varsigma^{q}}}[H_{h_n}^{j} \to H_{h_n}^{j-1}] \leq C_{S} \qquad \forall \, 1 \leq j \leq S_{\max} \ , \quad n \in \N \ ;
    \]
    we get
    \[
    \norm*{ K_{\ell_n + (k - 1)}^{q,+} \dots K_{\ell_n}^{q,+} - \big( \Op[h_n]{\kappa^{q} + i \varsigma^{q}} \big)^{k} }[H_{h}^{k-1} \to L^2] \leq \sum_{j = 1}^{k} \binom{k}{j} (C_{S})^{k - j} \big( (C_{S})^2 h_n \big)^j \qquad \forall \, n \in \N \ .
    \]
    This yields that there exists $\nu = \nu(S) \in \N$ such that for all $n \geq \nu$ and all $k \in S_{+}$, the following holds
    \[
    \norm*{ K_{\ell_n + (k - 1)}^{q,+} \dots K_{\ell_n}^{q,+} - \big( \Op[h_n]{\kappa^{q} + i \varsigma^{q}} \big)^{k} }[H_{h}^{k-1} \to L^2] \leq (C_{S})^{2k} [ (1 + h_n)^{k} - 1] \leq (C_{S})^{2k} h_n \ .
    \]

    Thanks to Lemma \ref{Lemma: ladder operators are continuous on Hhs},
    \[
    \norm*{ \big( \Op[h_n]{\kappa^{q} + i \varsigma^{q} } \big)^{k} - \Op[h_n]{(\kappa^{q} + i \varsigma^{q})^k} }[H_{h_n}^{k-1} \to L^2] \leq C_{S} h_n \qquad \forall \, k \in S_{+} \ , \quad \forall \, n \in \N \ ,
    \]
    hence, triangular inequality tells that there exists $\nu = \nu(S) \in \N$ such that for all $n \geq \nu$ and all $k \in S_{+}$,
    \[
    \norm*{ K_{\ell_n + (k - 1)}^{q,+} \dots K_{\ell_n}^{+} - \Op[h_n]{(\kappa^{q} + i \varsigma^{q})^k} }[H_{h_n}^{k-1} \to L^2] \leq (C_{S})^{2k} h_n + C_{S} h_n \ .
    \]

    We may now apply Lemma \ref{Lemma: elliptic estimate for zonal harmonics} to control semiclassical Sobolev norms of the sequence $(z_{\ell_n}^{q})_{n \in \N}$,
    \[
    \norm*{z_{\ell_n}^{q}}[H_{h_n}^{k-1}] \leq C_{S} \qquad \forall \, k \in S_{+} \ , \quad \forall \, n \in \N \ ,
    \]
    thus obtaining
    \[
    \norm*{ \Big[K_{\ell_n + (k - 1)}^{q,+} \dots K_{\ell_n}^{q,+} - \Op[h_n]{(\kappa^{q} + i \varsigma^{q})^k} \Big] z_{\ell_n}^{q} }[L^2(\S[d])] \leq \Big[ (C_{S})^{2k+1} + (C_{S})^2 \Big] h_n \ ,
    \]
    for all $k \in S_{+}$, and all $n \geq \nu$. Thanks to triangular inequality, we have shown there exists $\nu = \nu(S) \in \N$ such that for all $n \geq \nu$,
    \[
    \norm*{ \Big[ \mathcal{K}_{n}^{q, X, S, +} - \Opp_{h_n}\Big( {\textstyle \sum_{k \in S_{+}} X_{k} (\kappa^{q} + i \varsigma^{q})^k} \Big) \Big]z_{\ell_n}^{q} }[L^2(\S[d])] \leq \sum_{k \in S_{+}} \abs*{X_k} \Big[ (C_{S})^{2k+1} + (C_{S})^2 \Big] h_n \ .
    \]
    
    In a similar fashion, one may prove as well there exists $\nu = \nu(S) \in \N$ such that for all $n \geq \nu$,
    \[
    \norm*{ \Big[ \mathcal{K}_{n}^{q, X, S, -} - \Opp_{h_n}\Big( {\textstyle \sum_{k \in S_{-}} X_{k} (\kappa^{q} - i \varsigma^{q})^{-k} } \Big) \Big] z_{\ell_n}^{q} }[L^2(\S[d])] \leq \sum_{-k \in S_{-}} \abs*{X_{-k}} \Big[ (C_{S})^{2k+1} + (C_{S})^2 \Big] h_n \ ,
    \]
    Recalling \eqref{Eq: uqUpsilonX into ladder operators}, we have shown there exists $C(S) > 0$ and $\nu = \nu(S) \in \N$ such that for all $n \geq \nu$,
    \begin{equation} \label{Eq: Proposition symbol of lincomb end}
        \norm[\bigg]{ u_{n}^{q, X, S} - \Op[h_n]{\Gamma^{q, X, S}} z_{\ell_n}^{q} }[L^2(\S[d])] \leq \norm*{\charf_{S} X}[\ell_1] C(S) h_n \qquad \forall \, n \geq \nu \ ,
    \end{equation}
    and the symbol
    \begin{equation*}
        \Gamma^{q, X, S} \coloneqq \sum_{k \in S_{-}} X_{k}(\kappa^{q} - i \varsigma^{q})^{-k} + X_0 + \sum_{k \in S_{+}} X_{k} (\kappa^{q} + i \varsigma^{q})^{k} \ .
    \end{equation*}
    Combining this with \eqref{Eq: uqUpsilonxn - uqUpsilonX norm}, Theorem \ref{Thm: approx of uqUpsilonxn by PsDO on zln} is proven.
\end{proof}

Theorem \ref{Thm: approx of uqUpsilonxn by PsDO on zln} is the key ingredient to compute semiclassical defect measures of linear combinations of Green's functions. Recall the following notation introduced back in the introduction. For $q \in \S[d]$, define the measure $\nu_{q, 1/2}$ on $T^*\S[d]$ as
\begin{equation} \label{Eq: half measure on flow-out defintion}
    \int_{T^*\S[d]} a(x, \xi) \nu_{q, 1/2}(\D{x}, \D{\xi}) \coloneqq \int_{S_{q}^*\S[d]} \int_{0}^{\pi} a(\phi_t(q, \xi)) \frac{\D{t}}{2\pi} \frac{\D{\xi}}{\vol(\S[d-1])} \ , \qquad \forall \ a \in \Co(T^*\S[d]) \ .
\end{equation}

\begin{Prop} \label{Prop: nuqhalf properties}
    Let $q \in \S[d]$ and let $H(x, \xi) \coloneqq \abs{\xi}_{x}^2$ be Hamiltonian function that generates the geodesic flow on $T^*\S[d]$. The following holds:
    \begin{enumerate}
        \item $\supp \nu_{q, 1/2} \subseteq S^*\S[d]$.
        \item $\nu_{q, 1/2}(S^*\S[d]) = \frac{1}{2}$.
        \item $\nu_{q} = \nu_{q, 1/2} + \nu_{-q, 1/2}$ (see \eqref{Eq: def of nuq}).
        \item For every $a \in \CinfK(T^*\S[d])$,
        \begin{equation} \label{Eq: nuqhalf geodesic derivative}
            \int_{T^*\S[d]} \{ a, \, H\} \nu_{q, 1/2} = \frac{1}{2\pi} \int_{S_{-q}^*\S[d]} a(-q, \xi) \frac{\D{\xi}}{\vol(\S[d-1])} - \frac{1}{2\pi} \int_{S_{q}^*\S[d]} a(q, \xi) \frac{\D{\xi}}{\vol(\S[d-1])} \ .
        \end{equation}
    \end{enumerate}
\end{Prop}
\begin{proof}
    Properties 1 through 3 are trivially satisfied. We prove property 4. Let $a \in \CinfK(T^*\S[d])$, and recall that if $\phi_t$ denotes the geodesic flow on $T^*\S[d]$ at time $t$, 
    \[
    \{ a, \, H\} (x, \xi) = \partial_{s} [a \circ \phi_s(x, x)] \qquad \forall \, (x, \xi) \in T^*\S[d] \ .
    \]
    Plugging this into the definition of $\nu_{q, 1/2}$,
    \begin{align*}
        \int_{T^*\S[d]} \{ a, \, H\} \nu_{q, 1/2} & = \int_{S_{q}^*\S[d]} \int_{0}^{\pi} \partial_{s} [a \circ \phi_{t + s}(q, \xi)] \frac{\D{t}}{2\pi} \frac{\D{\xi}}{\vol(\S[d-1])} \\
        & = \int_{S_{q}^*\S[d]} \frac{1}{2\pi} \Big[ (a \circ \phi_{\pi})(q, \xi) - (a \circ \phi_{0}) (q, \xi) \Big] \frac{\D{\xi}}{\vol(\S[d-1])} \ ,
    \end{align*}
    and \eqref{Eq: nuqhalf geodesic derivative} follows.
\end{proof}

\subsection{Antipodal points case}
We start by proving the simpler case that $Q$ just contains a pair of antipodal points, $Q = \{q, -q\}$; the case $Q = \{q\}$ is included in this by taking $\beta_{n,-q} = 0$ for all $n \in \N$ (see Corollary \ref{Cor: SDM of seq of single GF} below).

\begin{Theorem} \label{Thm: characterization of SDM of ghQbeta}
    Let $Q = \{q, -q\} \subseteq \S[d]$. Let $(h_n)_{n \in \N} \subseteq (0,1)$ such that $h_n \to 0^+$, and $(\beta_{n})_{n \in \N} \subseteq \C[2]$. Assume there exists a unique Radon probability measure $\mu$ on $T^*\S[d]$ such that
    \[
    \lim_{n \to \infty} \ip*{ g_{h_n}^{Q, \beta_n} }{\Op[h_n]{a} g_{h_n}^{Q, \beta_n} }[L^2(\S[d])] = \int_{T^*\S[d]} a \, \mu \qquad \forall \, a \in \CinfK(T^*\S[d]) \ .
    \]
    Then $\mu = m_{q} \nu_{q, 1/2} + m_{-q} \nu_{-q, 1/2}$ for some $(m_{q}, m_{-q}) \in [0,2]^2$ such that $\frac{1}{2}(m_{q} + m_{-q}) = 1$ (see \eqref{Eq: mu is lincomb of halfnuq scenario1}, \eqref{Eq: mu is lincomb of halfnuq scenario2}, \eqref{Eq: mu is lincomb of halfnuq scenario3}, and \eqref{Eq: mu is lincomb of halfnuq scenario4}).
\end{Theorem}

\begin{proof}
    Let $(h_n)_{n \in \N} \subseteq (0,1)$, $(\beta_{n})_{n \in \N} \subseteq \C[2]$, and $\mu$ be as stated in the theorem.

    Since $G_{h}^{Q, z \beta} = z G_{h}^{Q, \beta}$ for any $z \in \C$, we may assume that the sequence $(\beta_{n})_{n \in \N}$ is contained in the complex sphere $\{ \beta \in \C[2] \colon \abs{\beta} = 1\}$. Let $(\ell_n)_{n \in \N} \subseteq \N$ and $(\sigma_n)_{n \in \N} \subseteq [0,1)$ be defined by
    \[
    h_n^{-2} = (\ell_n + \sigma_n) (\ell_n + \sigma_n + d - 1) \ .
    \]
    Up to extraction of a subsequence, we may assume that
    \begin{enumerate}
        \item $h_n^{-2} \notin \Spec(\Lap)$ for all $n \in \N$, or $h_n^{-2} \in \Spec(\Lap)$ for all $n \in \N$.
        \item $\beta_n \to \beta \in \C[2]$.
        \item $\sigma_n \to \sigma \in [0,1]$.
        \item $\ell_n$ is even for all $n \in \N$, or odd for all $n \in \N$. We set $\rho \coloneqq +1$ if $\ell_n$ is always even, $\rho \coloneqq -1$ if $\ell_n$ is always odd.
    \end{enumerate}
    Recall the definition of the symbols $\kappa^{q}$ and $\varsigma^{q}$ back in \eqref{Eq: definition symbols kappaq varsigmaq} for $r \coloneqq d(x, q)$.
    \begin{equation} \label{Eq: def kappaq varsigmaq thm cardQ=2}
        \kappa^{q}(x, \xi) \coloneqq \cos r \ , \qquad \varsigma^{q}(x, \xi) \coloneqq \xi(\sin r \cdot \partial_{r}) \ .
    \end{equation}
    We give a detailed proof under the hypothesis of scenario 1, and sketch the proof under the hypothesis of scenarios 2,3 and 4 due to their large similarities.

    \textbf{Scenario 1.} Assume that $h_n^{-2} \notin \Spec(\Lap)$ for all $n \in \N$, and that $\sigma \in (0,1)$. We prove that
    \begin{equation} \label{Eq: mu is lincomb of halfnuq scenario1}
        \mu = m_{q, +}^{\sigma, \beta, \rho} \nu_{q, 1/2} + m_{q, -}^{\sigma, \beta, \rho} \nu_{-q, 1/2} \ ,
    \end{equation}
    for the weights
    \begin{equation} \label{Eq: mqpmsigmabetarho definition}
        m_{q, \pm}^{\sigma, \beta, \rho} \coloneqq \frac{ \abs*{ \beta_{q} + \rho \beta_{-q} e^{\pm i\pi \sigma} }^2 }{\frac{1}{2}\abs*{ \beta_{q} + \rho \beta_{-q} e^{i\pi \sigma} }^2  + \frac{1}{2} \abs*{ \beta_{q} + \rho \beta_{-q} e^{-i\pi \sigma} }^2} \ .
    \end{equation}

    Set
    \begin{equation*}
        C_{\sigma, \beta, \rho} \coloneqq \bigg(\sum_{k \in \Z} \frac{ \abs{\beta_{q} + (-1)^{k} \rho \beta_{-q} }^2 }{ (k - \sigma)^2 } \bigg)^{\frac{1}{2}} \ .
    \end{equation*}
    Let $\Upsilon \in \N$ be large. Thanks to Theorem \ref{Thm: gh quasimode scenario 1}, there exist $C > 1$ and $\nu \in \N$ (depending on $\Upsilon$) such that for all $n \geq \nu$
    \[
    \norm[\bigg]{ g_{h_n}^{Q, \beta_n} - \frac{1}{C_{\sigma, \beta, \rho}} \sum_{\abs{k} \leq \Upsilon} \frac{ \beta_{q} + (-1)^{k} \rho \beta_{-q} }{k - \sigma} z_{\ell_n + k}^{q} }[L^2(\S[d])] \leq C \Big[ \Upsilon h_n + \abs{\sigma_n - \sigma} + \abs{\beta_n - \beta} + \Upsilon^{-\frac{1}{2}} \Big] \ .
    \]
    Let $X = (X_k)_{k \in \Z} \in \ell_2(\Z)$ be defined by
    \begin{equation} \label{Eq: seq xnk def scenario1}
        X_{k} \coloneqq \frac{1}{C_{\sigma, \beta, \rho}} \frac{ \beta_{q} + (-1)^{k} \rho \beta_{-q} }{k - \sigma} \qquad  k \in \Z \ ,
    \end{equation}
    and let the smooth symbol
    \begin{equation} \label{Eq: symbol Gamma scenario1}
        \Gamma^{q, \beta, \Upsilon, \sigma, \rho} \coloneqq \frac{1}{C_{\sigma, \beta, \rho}} \Bigg[ \sum_{k = 0}^{\Upsilon} \frac{ \beta_{q} + (-1)^{k} \rho \beta_{-q} }{-k - \sigma} (\kappa^{q} - i \varsigma^{q})^{k} + \sum_{k = 1}^{\Upsilon} \frac{ \beta_{q} + (-1)^{k} \rho \beta_{-q} }{k - \sigma} (\kappa^{q} + i \varsigma^{q})^{k} \Bigg] \ .
    \end{equation}
    Thanks to Theorem \ref{Thm: approx of uqUpsilonxn by PsDO on zln}, there exist $C(X, \Upsilon) > 1$ such that 
    \[
    \norm*{ \frac{1}{C_{\sigma, \beta, \rho}} \sum_{\abs{k} \leq \Upsilon} \frac{ \beta_{q} + (-1)^{k} \rho \beta_{-q} }{k - \sigma} z_{\ell_n + k}^{q} - \Op[h_n]{ \Gamma^{q, \beta, \Upsilon, \sigma, \rho} } z_{\ell_n}^{q} }[L^2(\S[d])] \leq C(X, \Upsilon) h_n \ .
    \]
    Combining these estimates, we get that there exists $C > 1$, $C(X, \Upsilon) > 1$, and $\nu \in \N$ such that for all $n \geq \nu$
    \[
    \norm[\bigg]{ g_{h_n}^{Q, \beta_n} - \Op[h_n]{ \Gamma^{q, \beta, \Upsilon, \sigma, \rho} } z_{\ell_n}^{q} }[L^2(\S[d])] \leq C \Big[ C(X, \Upsilon) h_n + \abs{\sigma_n - \sigma} + \abs{\beta_n - \beta} + \Upsilon^{-\frac{1}{2}} \Big] \ .
    \]
    Therefore, by symbolic calculus, for all $n \geq \nu$,
    \begin{multline*}
        \abs*{ \ip{ g_{h_n}^{Q, \beta_n} }{\Op[h_n]{a} g_{h_n}^{Q, \beta_n} }[L^2(\S[d])] - \ip{z_{\ell_n}^{q}}{ \Op[h_n]{ a \, \abs{ \Gamma^{q, \beta, \Upsilon, \sigma, \rho}}^2 } z_{\ell_n}^{q} }[L^2(\S[d])] } \leq \\
        \leq C \Big[ C(X, \Upsilon) h_n + \abs{\sigma_n - \sigma} + \abs{\beta_n - \beta} + \Upsilon^{-\frac{1}{2}} \Big] \ .
    \end{multline*}
    Thanks to Theorem \ref{Thm: SDM of zonal harmonics}, taking limits $n \to \infty$ we infer
    \begin{equation} \label{Eq: approximation mu by symbolnuq}
        \abs[\bigg]{ \int_{T^*\S[d]} a \, \mu - \int_{T^*\S[d]} a \cdot \abs{ \Gamma^{q, \beta, \Upsilon, \sigma, \rho}}^2 \nu_{q} } \leq C \Upsilon^{-\frac{1}{2}} \ .
    \end{equation}
    Taking $\Upsilon \to \infty$ would allow us to characterize the measure $\mu$, but for that, we need to study the \weakstar limit of $\abs{ \Gamma^{q, \beta, \Upsilon, \sigma, \rho}}^2 \nu_{q}$ as $\Upsilon \to \infty$ first.

    If $\phi_t$ denotes the geodesic flow on $T^*\S[d]$ at time $t$, it is not difficult to check (see \eqref{Eq: def kappaq varsigmaq thm cardQ=2}) that
    \begin{equation} \label{Eq: kappaq varsigmaq evaluated on flow-out from q}
        \kappa^{q}(\phi_t(q, \xi)) = \cos t \ , \qquad \varsigma^{q}(\phi_t(q, \xi)) = \sin t \qquad \forall \, \xi \in S_{q}^*\S[d] \ ,
    \end{equation}
    thus,
    \begin{align*}
        \Gamma^{q, \beta, \Upsilon, \sigma, \rho}(\phi_{t}(q, \xi)) & = \frac{1}{C_{\sigma, \beta, \rho}} \sum_{\abs{k} \leq \Upsilon} \frac{\beta_{q} + (-1)^{k} \rho \beta_{-q} }{k - \sigma} e^{ikt} \\
        & = \frac{\beta_{q}}{C_{\sigma, \beta, \rho}} \sum_{\abs{k} \leq \Upsilon} \frac{1}{k - \sigma} e^{ikt} + \frac{ \rho \beta_{-q}}{C_{\sigma, \beta, \rho}} \sum_{\abs{k} \leq \Upsilon} \frac{(-1)^{k}}{k - \sigma} e^{ikt} \qquad \forall \, \xi \in S_{q}^{*}\S[d] \ .
    \end{align*}
    If $\T \coloneqq \R / (2\pi \Z)$, let $\gamma_{\sigma}, [S_{-\pi} \gamma^{\sigma}] \in L^2(\T)$ be defined by
    \[
    \gamma^{\sigma}(t) \coloneqq \frac{2\pi i}{1 - e^{i2\pi\sigma} } e^{i \sigma t} \ , \quad t \in [0,2\pi) \ , \qquad [S_{-\pi}\gamma^{\sigma}] (t) \coloneqq \begin{dcases} \gamma^{\sigma}(t + \pi) & \qquad t \in [0, \pi) \ , \\ \gamma^{\sigma}(t - \pi) & \qquad t \in [\pi, 2\pi) \ . \end{dcases}
    \]
    We observe that $\Gamma^{q, \beta, \Upsilon, \sigma, \rho}(\phi_{t}(q, \xi))$ is the partial Fourier series of
    \[
    \gamma^{\sigma, \beta, \rho} \coloneqq \frac{1}{C_{\sigma, \beta, \rho}} \Big[ \beta_{q} \gamma^{\sigma} + \rho \beta_{-q} [S_{-\pi}\gamma^{\sigma}] \Big] \ .
    \]
    Carleson theorem states that the partial Fourier series of some $f \in L^2(\T)$ converges a.e. $t \in \T$ to $f$, hence, Carleson theorem, in combination with dominated convergence theorem, implies
    \begin{equation} \label{Eq: limitUpsilon symbolnuq scenario1}
        \lim_{\Upsilon \to \infty} \int_{T^*\S[d]} a \cdot \abs{ \Gamma^{q, \beta, \Upsilon, \sigma, \rho}}^2 \nu_{q} = \int_{S_{q}^*\S[d]} \int_{0}^{2\pi} a(\phi_t(q, \xi)) \abs*{ \gamma^{\sigma, \beta, \rho}(t)}^2 \frac{\D{t}}{2\pi} \frac{\D{\xi}}{\vol(\S[d-1])} \qquad \forall \, a \in \CinfK(T^*\S[d]) \ .
    \end{equation}

    \begin{Claim}
        For every $\sigma \in (0,1)$, $\beta \in \C[2]$ and $\rho \in \{-1, +1\}$, the following identity holds
        \begin{equation} \label{Eq: Csigmabetarho value}
            (C_{\sigma, \beta, \rho})^2 = \frac{(2\pi)^2}{\abs{1 - e^{i2\pi \sigma}}^2 } \Big[ \frac{1}{2}\abs*{ \beta_{q} + \rho \beta_{-q} e^{i\pi \sigma} }^2  + \frac{1}{2} \abs*{ \beta_{q} + \rho \beta_{-q} e^{-i\pi \sigma} }^2  \Big] \ .
        \end{equation}
        Moreover, for $m_{q, \pm}^{\sigma, \beta, \rho}$ defined in \eqref{Eq: mqpmsigmabetarho definition} one has
        \begin{equation} \label{Eq: Claim aux33}
            \abs*{ \gamma^{\sigma, \beta, \rho}(t)}^2 = \begin{dcases}
                m_{q, +}^{\sigma, \beta, \rho}& \quad t \in [0,\pi) \ , \\
                m_{q, -}^{\sigma, \beta, \rho} & \quad t \in [\pi, 2\pi) \ .
            \end{dcases}
        \end{equation}
    \end{Claim}
    \begin{proof}[Proof of claim]
        We see that for $t \in [0,\pi)$,
        \begin{equation} \label{Eq: Claim aux34}
            \gamma^{\sigma, \beta, \rho}(t) = \frac{1}{C_{\sigma, \beta, \rho}} \frac{2\pi i}{1 - e^{i 2\pi \sigma}} [ \beta_{q} + \rho \beta_{-q} e^{i\pi \sigma} ] e^{i \sigma t} \ ,
        \end{equation}
        and for $t \in [\pi, 2\pi)$,
        \begin{equation} \label{Eq: Claim aux35}
        \gamma^{\sigma, \beta, \rho}(t) = \frac{1}{C_{\sigma, \beta, \rho}} \frac{2\pi i}{1 - e^{i 2\pi \sigma}} [ \beta_{q} + \rho \beta_{-q} e^{-i\pi \sigma} ] e^{i \sigma t} \ ,
        \end{equation}
        thus
        \[
        \norm*{\gamma^{\sigma, \beta, \rho}}[L^2(\T)]^2 = \frac{1}{(C_{\sigma, \beta, \rho})^2} \frac{(2\pi)^2}{\abs{1 - e^{i2\pi \sigma}}^2 } \Big[ \abs*{ \beta_{q} + \rho \beta_{-q} e^{i\pi \sigma} }^2 \pi + \abs*{ \beta_{q} + \rho \beta_{-q} e^{-i\pi \sigma} }^2 \pi \Big] \ .
        \]
        Meanwhile, Parseval's identity gives
        \[
        \norm*{\gamma^{\sigma, \beta, \rho}}[L^2(\T)]^2 = 2\pi \ .
        \]
        \eqref{Eq: Csigmabetarho value} follows from these identities.

        Now, using \eqref{Eq: Csigmabetarho value} in combination with \eqref{Eq: Claim aux34} and \eqref{Eq: Claim aux35}, and the definition of $m_{q, \pm}^{\sigma, \beta, \rho}$, we obtain \eqref{Eq: Claim aux33}.
    \end{proof}

    Recall the measures $\nu_{q, 1/2}$ from \eqref{Eq: half measure on flow-out defintion}. Applying \eqref{Eq: Claim aux33} into \eqref{Eq: limitUpsilon symbolnuq scenario1}, one gets
    \[
    \lim_{\Upsilon \to \infty} \int_{T^*\S[d]} a \cdot \abs{ \Gamma^{q, \beta, \Upsilon, \sigma, \rho}}^2 \nu_{q} = \int_{T^*\S[d]} a \Big( m_{q, +}^{\sigma, \beta, \rho} \nu_{q, 1/2} + m_{q, -}^{\sigma, \beta, \rho} \nu_{-q, 1/2} \Big) \ .
    \]
    Therefore, taking limits $\Upsilon \to \infty$ in \eqref{Eq: approximation mu by symbolnuq} we obtain
    \[
    \int_{T^*\S[d]} a \, \mu = \int_{T^*\S[d]} a \, \Big( m_{q, +}^{\sigma, \beta, \rho} \nu_{q, 1/2} + m_{q, -}^{\sigma, \beta, \rho} \nu_{-q, 1/2} \Big) \qquad \forall \, a \in \CinfK(T^*\S[d]) \ .
    \]

    \textbf{Scenario 2.} Assume that $h_n^{-2} \notin \Spec(\Lap)$ for all $n \in \N$ and that
    \[
    \sup_{n \in \N} \frac{\abs*{\beta_{n,q} + (-1)^{\sigma} \rho \beta_{n,-q}}}{ \abs{\sigma - \sigma_n} } = \infty \ .
    \]
    We prove that
    \begin{equation} \label{Eq: mu is lincomb of halfnuq scenario2}
        \mu = \nu_q = \nu_{q, 1/2} + \nu_{-q, 1/2} \ .
    \end{equation}
    
    Up to extraction of a subsequence we may assume that
    \[
    \lim_{n \to \infty} \frac{\abs*{\beta_{n,q} + (-1)^{\sigma} \rho \beta_{n,-q}}}{ \abs{\sigma - \sigma_n} } = \infty \ ,
    \]
    thus we are under the hypothesis of Theorem \ref{Thm: gh quasimode scenario 2}. Thanks to Theorem \ref{Thm: gh quasimode scenario 2} there exist $C > 1$ and $\nu \in \N$ such that for all $n \geq \nu$,
    \[
    \norm*{g_{h_n}^{Q, \beta_n} - (-1)^{1 - \sigma} z_{\ell_n + \sigma}}[L^2(\S[d])]^2 \leq C \bigg[ h_n + \frac{ \abs{\sigma - \sigma_n} }{\abs*{\beta_{n,q} + (-1)^{\sigma} \rho \beta_{n,-q}}} \bigg] \ .
    \]
    This implies that
    \begin{equation*}
        \abs*{ \ip{g_{h_n}^{Q, \beta_n} }{ \Op[h_n]{a} g_{h_n}^{Q, \beta_n} }[L^2(\S[d])] - \ip{z_{\ell_n + \sigma}^{q}}{ \Op[h_n]{a} z_{\ell_n + \sigma} }[L^2(\S[d])] } \leq C \bigg[ h_n + \frac{ \abs{\sigma - \sigma_n} }{\abs*{\beta_{n,q} + (-1)^{\sigma} \rho \beta_{n,-q}}} \bigg] \ .
    \end{equation*}
    Thanks to Theorem \ref{Thm: SDM of zonal harmonics}, we may take limits $n \to \infty$ above obtaining
    \[
    \int_{T^*M} a \, \mu = \int_{T^*M} a \, \nu_{q} \qquad \forall \, a \in \CinfK(T^*\S[d]) \ .
    \]

    \textbf{Scenario 3.} Assume that $h_n^{-2} \notin \Spec(\Lap)$ for all $n \in \N$ and that
    \[
    \sup_{n \in \N} \frac{\abs*{\beta_{n,q} + (-1)^{\sigma} \rho \beta_{n,-q}}}{ \abs{\sigma - \sigma_n} } < \infty \ .
    \]
    We prove that
    \begin{equation} \label{Eq: mu is lincomb of halfnuq scenario3}
        \mu = m_{q, +}^{\sigma, \beta, \rho} \nu_{q, 1/2} + m_{q, -}^{\sigma, \beta, \rho} \nu_{-q, 1/2}
    \end{equation}
    for the weights
    \begin{equation} \label{Eq: mqpmsigmabetarho definition scenarion3}
            m_{q, \pm}^{\sigma, \beta, \rho} \coloneqq \frac{ \abs*{ c_{\sigma, \beta, \rho} \pm i \pi \beta_q}^2}{ \abs{c_{\sigma, \beta, \rho}}^2 + \pi^2 \abs{\beta_{q}}^2 } \ .
    \end{equation}
    
    Up to extraction of a subsequence we may assume that
    \[
    c_{\sigma, \beta, \rho} \coloneqq \lim_{n \to \infty} \frac{\abs*{\beta_{n,q} + (-1)^{\sigma} \rho \beta_{n,-q}}}{ \abs{\sigma - \sigma_n} } \in \C \ ,
    \]
    thus we are under the hypothesis of Theorem \ref{Thm: gh quasimode scenario 3}. Set
    \begin{equation*}
        C_{\sigma, \beta, \rho} \coloneqq \bigg(\abs{c_{\sigma, \beta, \rho}}^2 + \sum_{\substack{ k \in \Z \\ k \neq \sigma} } \frac{ \abs{\beta_{q} + (-1)^{k} \rho \beta_{-q} }^2 }{ (k - \sigma)^2 } \bigg)^{\frac{1}{2}} \ .
    \end{equation*}
    Let $\Upsilon \in \N$ large. A combination of Theorem \ref{Thm: gh quasimode scenario 3} and Theorem \ref{Thm: approx of uqUpsilonxn by PsDO on zln} for the symbol (note that summation index $k = 0$ corresponds with the eigenvalue index $\ell_n + \sigma$)
    \begin{equation} \label{Eq: symbol Gamma scenario3}
        \Gamma^{q, \beta, \Upsilon, \sigma, \rho} \coloneqq \frac{1}{C_{\sigma, \beta, \rho}} \Bigg[ \sum_{k = 1}^{2\Upsilon - 1} \frac{ \beta_{q} + (-1)^{\sigma - k} \rho \beta_{-q} }{-k} (\kappa^{q} - i \varsigma^{q})^{k} + c_{\sigma, \beta, \rho} + \sum_{k = 1}^{2\Upsilon - 1} \frac{ \beta_{q} + (-1)^{\sigma + k} \rho \beta_{-q} }{k} (\kappa^{q} + i \varsigma^{q})^{k} \Bigg] \ .
    \end{equation}
    give the following: there exists $C > 1$, $C(X, \Upsilon) > 1$, and $\nu \in \N$ such that for all $n \geq \nu$
    \begin{multline*}
        \norm[\bigg]{ g_{h_n}^{Q, \beta_n} - \Op[h_n]{ \Gamma^{q, \beta, \Upsilon, \sigma, \rho} } z_{\ell_n}^{q} }[L^2(\S[d])] \leq \\
        \leq C \bigg[ C(X, \Upsilon) h_n + \abs{\sigma_n - \sigma} + \abs{\beta_n - \beta} + \abs[\bigg]{ c_{\sigma, \beta, \rho} - \frac{\beta_{n,q} + (-1)^{\sigma} \rho \beta_{n,-q}}{ \sigma - \sigma_n } } + \Upsilon^{-\frac{1}{2}} \bigg] \ .
    \end{multline*}

    Now, symbolic calculus implies
    \begin{multline*}
        \abs*{ \ip{ g_{h_n}^{Q, \beta_n} }{\Op[h_n]{a} g_{h_n}^{Q, \beta_n} }[L^2(\S[d])] - \ip{z_{\ell_n}^{q}}{ \Op[h_n]{ a \, \abs{ \Gamma^{q, \beta, \Upsilon, \sigma, \rho}}^2 } z_{\ell_n}^{q} }[L^2(\S[d])] } \leq \\
        C \Big[ C(X, \Upsilon) h_n + \abs{\sigma_n - \sigma} + \abs{\beta_n - \beta} + \abs[\bigg]{ c_{\sigma, \beta, \rho} - \frac{\beta_{n,q} + (-1)^{\sigma} \rho \beta_{n,-q}}{ \sigma - \sigma_n } } + \Upsilon^{-\frac{1}{2}} \Big] \ ,
    \end{multline*}
    and thanks to Theorem \ref{Thm: SDM of zonal harmonics}, we can take limits $n \to \infty$ and obtain
    \begin{equation} \label{Eq: approximation mu by symbolnuq rep}
        \abs[\bigg]{ \int_{T^*\S[d]} a \, \mu - \int_{T^*\S[d]} a \cdot \abs{ \Gamma^{q, \beta, \Upsilon, \sigma, \rho}}^2 \nu_{q} } \leq C \Upsilon^{-\frac{1}{2}} \ .
    \end{equation}
    Once again, we want to compute the \weakstar limit of $\abs{ \Gamma^{q, \beta, \Upsilon, \sigma, \rho}}^2 \nu_{q}$ as $\Upsilon \to \infty$.
    
    Recall that $\beta_{q} + (-1)^{\sigma} \rho \beta_{-q} = 0$, hence, for $k \in \Z$,
    \begin{equation} \label{Eq: simplification betaq + -1sigmarhobeta-q = 0}
        \beta_{q} + (-1)^{\sigma} \rho \beta_{-q} = \begin{dcases}
            0 & \text{if $k \equiv \sigma \mathrm{(mod \, 2)}$} \\
            2\beta_{q} & \text{if $k \equiv 1 - \sigma \mathrm{(mod \, 2)}$}
        \end{dcases}
    \end{equation}
    Therefore, applying identities \eqref{Eq: kappaq varsigmaq evaluated on flow-out from q} and \eqref{Eq: simplification betaq + -1sigmarhobeta-q = 0} later, we get that for all $\xi \in S_{q}^*\S[d]$,
    \begin{align*}
        \Gamma^{q, \beta, \Upsilon, \sigma, \rho}(\phi_{t}(q, \xi)) & = \frac{c_{\sigma, \beta, \rho}}{ C_{\sigma, \beta, \rho} } \charf_{[0,2\pi)}(t) + \frac{1}{C_{\sigma, \beta, \rho}} \sum_{\abs{k} = 1}^{2\Upsilon - 1} \frac{\beta_{q} + (-1)^{\sigma + k} \rho \beta_q }{k} e^{ikt} \\
        & = \frac{c_{\sigma, \beta, \rho}}{ C_{\sigma, \beta, \rho} } \charf_{[0,2\pi)}(t) + \frac{4\beta_{q} i}{C_{\sigma, \beta, \rho}} \sum_{k = 1}^{\Upsilon} \frac{ \sin (2k - 1) t }{2k - 1} \ .
    \end{align*}
    We observe that, for $\xi \in S_{q}^*\S[d]$, $\Gamma^{q, \beta, \Upsilon, \sigma, \rho}(\phi_t(q, \xi))$ is the partial Fourier series of $\gamma^{\sigma, \beta, \rho} \in L^2(\T)$, given by
    \[
    \gamma^{\sigma, \beta, \rho} \coloneqq \frac{1}{C_{\sigma, \beta, \rho}} \Big[ c_{\sigma, \beta, \rho} \charf_{[0,2\pi)} + i \pi \beta_{q} \mathcal{S} \Big] \ , \qquad \mathcal{S}(t) = \begin{dcases} +1 & \text{if $t \in [0,\pi)$} \\ -1 & \text{if $t \in [\pi, 2\pi)$}
    \end{dcases} \ .
    \]
    Meanwhile, taking into account that definition of $C_{\sigma, \beta, \rho}$ and the fact that $\sum_{k = 1}^{\infty} \frac{1}{(2k - 1)^2} = \frac{\pi^2}{8}$, one infers that for $\sigma \in \{0, 1\}$, $\beta \in \C[2]$, and $\rho \in \{-1, +1\}$,
    \begin{equation} \label{Eq: Csigmabetarho value scenario3}
        C_{\sigma, \beta, \rho} = \bigg( \abs{c_{\sigma, \beta, \rho}}^2 + \pi^2 \abs{\beta_{q}}^2 \bigg)^{\frac{1}{2}} \ .
    \end{equation}
    This let us write $\abs{\gamma^{\sigma, \beta, \rho}(t)}^2$ in terms of the previously defined weights $m_{q, \pm}^{\sigma, \beta, q}$ \eqref{Eq: mqpmsigmabetarho definition scenarion3}
    \begin{equation} \label{Eq: Claim aux63}
        \abs*{ \gamma^{\sigma, \beta, \rho}(t)}^2 = \begin{dcases}
            m_{q, +}^{\sigma, \beta, \rho} & \quad t \in [0,\pi) \ ,\\
            m_{q, -}^{\sigma, \beta, \rho} & \quad t \in [\pi, 2\pi) \ .
        \end{dcases}
    \end{equation}
    Using that $\abs{a + b}^2 + \abs{a - b}^2 = 2\abs{a} + 2\abs{b}^2$ for any $a,b \in \C$, one can show that $\norm*{\gamma^{\sigma, \beta, \rho}}[L^2(\T)]^2 = 2\pi$.
    
    Summing up, Carleson theorem and dominated convergence theorem together give
    \[
    \lim_{\Upsilon \to \infty} \int_{T^*\S[d]} a \cdot \abs*{ \Gamma^{q, \beta, \Upsilon, \sigma, \rho}}^2 \nu_{q} = \int_{T^*\S[d]} a \, \Big( m_{q, +}^{\sigma, \beta, \rho} \nu_{q, 1/2} + m_{q, -}^{\sigma, \beta, \rho} \nu_{-q, 1/2} \Big) \qquad \forall \, a \in \CinfK(T^*\S[d]) \ ,
    \]
    therefore, taking limits $\Upsilon \to \infty$ in \eqref{Eq: approximation mu by symbolnuq rep} we infer
    \[
    \int_{T^*\S[d]} a \, \mu = \int_{T^*\S[d]} a \, \Big( m_{q, +}^{\sigma, \beta, \rho} \nu_{q, 1/2} + m_{q, -}^{\sigma, \beta, \rho} \nu_{-q, 1/2} \Big) \qquad \forall \, a \in \CinfK(T^*\S[d]) \ .
    \]

    \textbf{Scenario 4.} Assume that $h_n^{-2} = \lambda_{\ell_n}^2$ (thus $\sigma_n = 0$) and $\beta_{n,q} + \rho \beta_{n,-q} = 0$ for all $n \in \N$ so that $G_{h_n}^{Q, \beta_n}$ is a well-defined $L^2$-function (see \eqref{Eq: condition beta-h}. Set
    \[
    D_{\beta} \coloneqq \bigg( \sum_{\substack{ k \in \Z \\ k \neq 0} } \frac{ \abs{\beta_{q} + (-1)^{k} \rho \beta_{-q} }^2 }{ k^2 } \bigg)^{\frac{1}{2}} \ .
    \]
    The same reasoning as in scenario 3 with the added hypothesis $c_{\sigma, \beta, \rho} = 0$ works in scenario 4. This requires the combination of Theorem \ref{Thm: gh quasimode scenario 4} with Theorem \ref{Thm: approx of uqUpsilonxn by PsDO on zln} for the symbol
    \begin{equation} \label{Eq: symbol Gamma scenario4}
        \Gamma^{q, \beta, \Upsilon, \rho} \coloneqq \frac{1}{D_{\beta}} \Bigg[ \sum_{k = 1}^{2\Upsilon - 1} \frac{ \beta_{q} + (-1)^{\sigma - k} \rho \beta_{-q} }{-k} (\kappa^{q} - i \varsigma^{q})^{k} + \sum_{k = 1}^{2\Upsilon - 1} \frac{ \beta_{q} + (-1)^{\sigma + k} \rho \beta_{-q} }{k} (\kappa^{q} + i \varsigma^{q})^{k} \Bigg] \ .
    \end{equation}
    This leads to the identity
    \begin{equation} \label{Eq: mu is lincomb of halfnuq scenario4}
        \mu = \nu_{q} = \nu_{q, 1/2} + \nu_{-q, 1/2} \ .
    \end{equation}
\end{proof}

\begin{Remark} \label{Rmk: surjectivity of weights from GF}
    Any non-invariant measure of the form $m_{q} \nu_{q, 1/2} + m_{-q} \nu_{-q, 1/2}$ can be produced from a strong perturbation regime, that is $\sigma \in (0,1)$. If one sees the weight $m_{q, +}^{\sigma, \beta, \rho}$ \eqref{Eq: mqpmsigmabetarho definition} as a function of $\beta \in \C[2]$ for $\sigma$ and $\rho$ fixed, then $m$ is surjective thanks to intermediate value Theorem: say $\rho = +1$, then $m$ is continuous, $m(1, e^{-i\pi \sigma}) = 2$, and $m(-1, e^{-i\pi \sigma}) = 0$.
\end{Remark}

If $Q \subseteq \S[d]$ is a singleton, $Q = \{q\}$, then any semiclassical defect measure arising from a sequence of new eigenfunctions is invariant under the geodesic flow, as the following results shows.

\begin{Corollary} \label{Cor: SDM of seq of single GF}
    Let $q \in \S[d]$. Let $(h_n)_{n \in \N} \subseteq (0,1)$ such that $h_n \to 0^+$ and $h_n^{-2} \notin \Spec(\Lap)$ for all $n \in \N$. Assume there exists a unique Radon probability measure $\mu$ on $T^*\S[d]$ such that
    \[
    \lim_{n \to \infty} \ip{ g_{h_n}^{q} }{\Op[h_n]{a} g_{h_n}^{q} }[L^2(\S[d])] = \int_{T^*\S[d]} a \, \mu \qquad \forall \, a \in \CinfK(T^*\S[d]) \ .
    \]
    Then $\mu = \nu_q$.
\end{Corollary}

\begin{proof}
    This is a consequence of Theorem \ref{Thm: characterization of SDM of ghQbeta} by simply setting $\beta_n = (1, 0) \in \C[2]$ for all $n \in \N$. Under this conditions, scenario 3 and 4 cannot hold. If the sequence $(h_n)_{n \in \N}$ is such that scenario 2 holds, then $\mu = \nu_q$ immediately thanks to \eqref{Eq: mu is lincomb of halfnuq scenario2}. Meanwhile, if $(h_n)_{n \in \N}$ satisfies the conditions of scenario 1, then \eqref{Eq: mu is lincomb of halfnuq scenario1} and \eqref{Eq: mqpmsigmabetarho definition} imply that for this specific choice of $\beta_{n} \to \beta = (1,0)$,
    \[
    \mu = \nu_{q, 1/2} + \nu_{q, 1/2} = \nu_q \ . \qedhere
    \]
\end{proof}

\subsection{General case}
For a set $X \subseteq \S[d]$, define the set of antipodal points of $X$ as $-X \coloneqq \{ -x \in \S[d] \colon x \in X\}$. Let $A, B \subseteq \S[d]$ be a pair of finite subsets such that $A$, $B$, $(-A)$, and $(-B)$ are pairwise disjoint. Assume that the set $Q$ has the form
\begin{equation} \label{Eq: decomposition Q into A -A B}
    Q = A \cup (-A) \cup B \ ,
\end{equation}
$N \coloneqq \card Q = 2 (\card A) + \card B$. For $m \in [0,2]^N$ such that
\begin{equation} \label{Eq: condition on weights ma+ ma- mb}
    \sum_{q \in A} \tfrac{1}{2} (m_{q} + m_{-q}) + \sum_{p \in B} m_{p} = 1 \ ,
\end{equation}
define
\begin{equation} \label{Eq: non-inv SDM with weights ma+ ma- mb}
    \nu_{Q, m} \coloneqq \sum_{q \in A} \Big( m_{q} \nu_{q, 1/2} + m_{-q} \nu_{-q, 1/2} \Big) + \sum_{p \in B} m_{p} \nu_{p} \ .
\end{equation}

\begin{Theorem} \label{Thm: characterization of SDM of ghQbeta general case}
    Let $Q = A \cup (-A) \cup B \subseteq \S[d]$ be a finite subset as above. Let $(h_n)_{n \in \N} \subseteq (0,1)$ such that $h_n \to 0^+$, and $(\beta_{n})_{n \in \N} \subseteq \C[N]$. Assume there exists a unique Radon probability measure $\mu$ on $T^*\S[d]$ such that
    \[
    \lim_{n \to \infty} \ip*{ g_{h_n}^{Q, \beta_n} }{\Op[h_n]{a} g_{h_n}^{Q, \beta_n} }[L^2(\S[d])] = \int_{T^*\S[d]} a \, \mu \qquad \forall \, a \in \CinfK(T^*\S[d]) \ .
    \]
    Then $\mu = \nu_{Q, m}$ for some $m \in [0,2]^{N}$ satisfying \eqref{Eq: condition on weights ma+ ma- mb}.
\end{Theorem}

\begin{proof}
    Let $(h_n)_{n \in \N} \subseteq (0,1)$, $(\beta_{n})_{n \in \N} \subseteq \C[2]$, and $\mu$ as stated in the theorem. Up to extraction of a subsequence we may assume that $\beta_{n} \to \beta \in \C[N]$.

    For every $q \in A$, let $P_q \coloneqq \{q, -q\}$ and $\alpha_n^{q} \coloneqq (\beta_{n, q}, \beta_{n, -q}) \in \C[2]$. We decompose $g_{h_n}^{Q, \beta_n}$ as follows
    \[
    g_{h_n}^{Q, \beta_n} = \sum_{q \in A} \frac{ \norm*{G_{h_n}^{P_q, \alpha_n^{q}} }[L^2(\S[d])] }{ \norm*{G_{h_n}^{Q, \beta_n}}[L^2(\S[d])]} g_{h_n}^{P_q, \alpha_n^{q}} + \sum_{p \in B} \frac{ \norm*{G_{h_n}^{p} }[L^2(\S[d])] }{ \norm*{G_{h_n}^{Q, \beta_n}}[L^2(\S[d])]} g_{h_n}^{p} \ .
    \]
    Up to extraction of a subsequence, thanks to Theorem \ref{Thm: characterization of SDM of ghQbeta}, for every $q \in A$ there exist weights $\tilde{m}_{q}, \tilde{m}_{-q} \in [0,2]$ such that $\frac{1}{2}(\tilde{m}_{q} + \tilde{m}_{-q}) = 1$ and
    \begin{equation} \label{Eq: aux89}
        \lim_{n \to \infty} \ip*{g_{h_n}^{P_q, \alpha_{n}^{q}} }{\Op[h_n]{a} g_{h_n}^{P_q, \alpha_{n}^{q}} }[L^2(\S[d])] = \int_{T^*\S[d]} a \big(\tilde{m}_{q} \nu_{q, 1/2} + \tilde{m}_{-q} \nu_{-q, 1/2} \big) \qquad \forall \, a \in \CinfK(T^*\S[d]) \ .
    \end{equation}
    Meanwhile, up to extraction of a subsequence, thanks to Corollary \ref{Cor: SDM of seq of single GF}, we know that for all $q \in B$,
    \begin{equation} \label{Eq: aux90}
        \lim_{n \to \infty} \ip*{g_{h_n}^{q} }{\Op[h_n]{a} g_{h_n}^{q} }[L^2(\S[d])] = \int_{T^*\S[d]} a \, \nu_q \qquad \forall \, a \in \CinfK(T^*\S[d]) \ .
    \end{equation}
    Since $A \cup B$ is finite, we may find a subsequence such that \eqref{Eq: aux89} and \eqref{Eq: aux90} hold.

    We now observe that the measures $(m_{q} \nu_{q, 1/2} + m_{-q} \nu_{-q, 1/2})$ for $q \in A$, and $\nu_{q}$ for $q \in B$, are pairwise mutually singular and have compact support. Therefore, Lemma \ref{Lemma: orthogonality for sequence from SDM} imply that
    \begin{equation*}
        1 = \norm{g_{h_n}^{Q, \beta_n}}[L^2(\S[d])]^2 = \sum_{q \in A} \frac{ \norm*{G_{h_n}^{P_q, \alpha_n^{q}} }[L^2(\S[d])]^2 }{ \norm*{G_{h_n}^{Q, \beta_n}}[L^2(\S[d])]^2 }  + \sum_{p \in B} \frac{ \norm*{ G_{h_n}^{p} }[L^2(\S[d])]^2 }{ \norm*{G_{h_n}^{Q, \beta_n}}[L^2(\S[d])]^2} + \littleo_{n}(1) \ .
    \end{equation*}
    Up to extraction of a subsequence, we may assume that
    \begin{align*}
        w_{q} & \coloneqq \lim_{n \to \infty} \frac{ \norm*{G_{h_n}^{P_q, \alpha_n^{q}} }[L^2(\S[d])]^2 }{ \norm*{G_{h_n}^{Q, \beta_n}}[L^2(\S[d])]^2 } \in [0,1] \qquad \forall \, q \in A \ , \\
        w_{q} & \coloneqq \lim_{n \to \infty} \frac{ \norm*{ G_{h_n}^{p} }[L^2(\S[d])]^2 }{ \norm*{G_{h_n}^{Q, \beta_n}}[L^2(\S[d])]^2} \in [0,1] \qquad \forall \, q \in B \ .
    \end{align*}
    Now, Lemma \ref{Lemma: orthogonality for crossed SDM} and the definition of the weights $\{w_{q}\}_{q \in A \cup B}$ allow us to do the following computation
    \begin{align*}
        \ip*{g_{h_n}^{Q, \beta_n}}{ \Op[h_n]{a} g_{h_n}^{Q, \beta_n} }[L^2(\S[d])] & = \sum_{q \in A} w_{q} \ip*{g_{h_n}^{P_q, \alpha_{n}^{q} } }{ \Op[h_n]{a} g_{h_n}^{P_q, \alpha_{n}^{q} } }[L^2(\S[d])] \\
        & \qquad + \sum_{p \in B} w_{p} \ip*{g_{h_n}^{p} }{ \Op[h_n]{a} g_{h_n}^{p} }[L^2(\S[d])] + \littleo_n(1) \ .
    \end{align*}
    Finally, we may invoke identities \eqref{Eq: aux89} and \eqref{Eq: aux90} to obtain the desired result.
\end{proof}

\begin{Lemma} \label{Lemma: orthogonality for crossed SDM}
    Let $(u_n)_n$ and $(v_n)_n$ two normalized sequences in $L^2(M)$. Assume that for a fixed sequence of semiclassical parameters $(h_n)_n$, $(u_n)_n$ and $(v_n)_n$ have unique semiclassical defect measure $\mu$ and $\nu$ respectively. If $\mu \perp \nu$, then
    \[
    \lim_{n \to \infty} \ip{u_n}{\Op[h_n]{a} v_n} = 0 \qquad \qquad \forall \ a \in \CinfK(T^*M) \ .
    \]
\end{Lemma}

\section{Proofs of main theorems}
\label{Sec: proofs of main theorem}

\subsection{Proof of \texorpdfstring{Theorem \ref{Thm: 2nd theorem new ef}}{Theorem intro new EF}}
\label{Sec: Proof of 2nd theorem new EF}

    Assume that $Q = A \cup (-A) \cup B$ as described in \eqref{Eq: decomposition Q into A -A B}. Let $m \in [0,2]^N$ satisfying \eqref{Eq: condition on weights ma+ ma- mb}. Choose some $\sigma \in (0,1)$ and $\rho \in \{-1, +1\}$. Choose an increasing sequence $(\ell_n)_{n \in \N} \subseteq \N$ such that $(-1)^{\ell_n} = \rho$ for all $n \in \N$, and define $(h_n)_{n \in \N} \subseteq (0,1)$ as follows (see \eqref{Eq: def elln and sigman}):
    \[
    h_n^{-2} \coloneqq (\ell_n + \sigma) (\ell_n + \sigma_n + d - 1) \ .
    \]
    For each $q \in A$, let $\beta_q, \beta_{-q} \in \C[2]$ such that (see Remark \ref{Rmk: surjectivity of weights from GF})
    \begin{equation} \label{Eq: choice of betaq in terms of mq} 
        m_{\pm q} = \abs{\beta_{q} + \rho \beta_{-q} e^{\pm i \pi \sigma}}^2 .
    \end{equation}
    For each $p \in B$, let $\beta_{p} \in \C$ such that $\abs{\beta_{p}}^2 = m_{p}$. We claim the sequence $(g_{h_n}^{Q, \beta})_{n \in \N}$ has $\nu_{Q,m}$ as \SDM{}.

    We recycle the argument used in the proof of Theorem \ref{Thm: characterization of SDM of ghQbeta general case} with the added condition that $\sigma \in (0,1)$. For every $q \in A$, let $P_q \coloneqq \{q, -q\}$, and $\alpha^{q} \coloneqq (\beta_q, \beta_{-q}) \in \C[2]$. We decompose $g_{h_n}^{Q, \beta_n}$ as follows
    \begin{equation} \label{Eq: decomposition of ghQbeta general into ghQbeta in A B}
        g_{h_n}^{Q, \beta} = \sum_{q \in A} \frac{ \norm*{G_{h_n}^{P_q, \alpha^{q}} }[L^2(\S[d])] }{ \norm*{G_{h_n}^{Q, \beta}}[L^2(\S[d])]} g_{h_n}^{P_q, \alpha^{q}} + \sum_{p \in B} \frac{ \norm*{G_{h_n}^{p} }[L^2(\S[d])] }{ \norm*{G_{h_n}^{Q, \beta}}[L^2(\S[d])]} g_{h_n}^{p} \ .
    \end{equation}
    Since the sequences $(g_{h_n}^{P_q, \alpha^{q}})_{n \in \N}$ for $q \in A$ and $(g_{h_n}^{p})_{n \in \N}$ for $p \in B$ have mutually singular semiclassical defect measures, thanks to Lemma \ref{Lemma: orthogonality for sequence from SDM} we get
    \[
    \norm*{G_{h_n}^{Q, \beta}}[L^2(\S[d])]^2 = \sum_{q \in A} \norm*{G_{h_n}^{P_q, \alpha^{q}} }[L^2(\S[d])]^2 + \sum_{p \in B} \norm*{G_{h_n}^{p} }[L^2(\S[d])]^2  + \littleo_n \bigg( \norm*{G_{h_n}^{Q, \beta}}[L^2(\S[d])]^2 \bigg) \qquad \text{as $n \to \infty$.}
    \]
    Now, combining this estimate with Theorem \ref{Thm: Gh L2 norm estimates scenario 1} and Theorem \ref{Thm: Gh tails L2 norm estimates}, we get that for all $\eps > 0$ there exists $\nu \in \N$ such that for all $n \geq \nu$,
    \begin{equation} \label{Eq: aux 103}
        \abs*{ \norm*{G_{h_n}^{Q, \beta}}[L^2(\S[d])]^2 - \frac{(h_n)^{3-d}}{2(d-1) \vol(\S[d])} \bigg[ \sum_{q \in A} (C_{\sigma, \alpha^{q}, \rho})^2 + \sum_{p \in B} (C_{\sigma, (\beta_{p}, 0), \rho})^2 \bigg] } \leq \eps (h_n)^{3-d} \ ,
    \end{equation}
    where $C_{\sigma, \beta, \rho}$ was defined for $\beta \in \C[2]$ back in \eqref{Eq: Csigmabeta definition scenario 1}:
    \[
    C_{\sigma, \beta, \rho} \coloneqq \bigg(\sum_{k \in \Z} \frac{ \abs{\beta_{q} + (-1)^{k} \rho \beta_{-q} }^2 }{ (k - \sigma)^2 } \bigg)^{\frac{1}{2}} \ .
    \]
    Next, we use that for $\sigma \in (0,1)$, $\beta \in \C[2]$, and $\rho \in \{-1, +1\}$, $(C_{\sigma, \beta, \rho})^2$ takes the value \eqref{Eq: Csigmabetarho value}
    \[
    (C_{\sigma, \beta, \rho})^2 = \frac{(2\pi)^2}{\abs{1 - e^{i2\pi \sigma}}^2 } \Big[ \tfrac{1}{2}\abs*{ \beta_{q} + \rho \beta_{-q} e^{i\pi \sigma} }^2  + \tfrac{1}{2} \abs*{ \beta_{q} + \rho \beta_{-q} e^{-i\pi \sigma} }^2  \Big] \ .
    \]
    Due to our choice of $\beta_{q} \in \C$ for $q \in Q$, we get
    \begin{align*}
        (C_{\sigma, \alpha^{q}, \rho})^2 & = \frac{(2\pi)^2}{\abs{1 - e^{i2\pi \sigma}}^2 } \frac{ m_{q} + m_{-q} }{2} \qquad \forall \, q \in A \ , \\
        (C_{\sigma, (\beta_{p}, 0), \rho})^2 & = \frac{(2\pi)^2}{\abs{1 - e^{i2\pi \sigma}}^2 } m_{p} \qquad \forall \, p \in B \ .
    \end{align*}
    Substituting these identities into \eqref{Eq: aux 103} we find that for all $\eps > 0$ there exists $\nu \in \N$ such that for all $n \in \N$,
    \begin{equation} \label{Eq: aux 104}
        \abs*{ \norm*{G_{h_n}^{Q, \beta}}[L^2(\S[d])]^2 - \frac{(h_n)^{3-d}}{2(d-1) \vol(\S[d])} \frac{(2\pi)^2}{\abs{1 - e^{i2\pi \sigma}}^2 } } \leq \eps (h_n)^{3-d}
    \end{equation}
    where we made use of the fact that $m$ satisfies \eqref{Eq: condition on weights ma+ ma- mb}.

    Coming back to \eqref{Eq: decomposition of ghQbeta general into ghQbeta in A B}, we read that
    \[
    g_{h_n}^{Q, \beta} = \sum_{q \in A} \bigg(\frac{ m_{q} + m_{-q} }{2}\bigg)^{\frac{1}{2}} g_{h_n}^{P_q, \alpha^{q}} + \sum_{p \in B} m_{p}^{1/2} g_{h_n}^{p} + \littleo_{L^2}(1) \qquad \text{as $n \to \infty$.}
    \]
    Consequently, combining this estimate with Lemma \ref{Lemma: orthogonality for crossed SDM} we get
    \begin{equation} \label{Eq: aux 105}
        \begin{aligned}
            \ip*{g_{h_n}^{Q, \beta}}{\Op[h_n]{a} g_{h_n}^{Q, \beta}}[L^2(\S[d])] & = \sum_{q \in A} \frac{ m_{q} + m_{-q} }{2} \ip*{g_{h_n}^{P_q, \alpha^{q}}}{\Op[h_n]{a} g_{h_n}^{P_q, \alpha^{q}} }[L^2(\S[d])] \\
            & \qquad + \sum_{p \in B} m_{p} \ip*{ g_{h_n}^p}{\Op[h_n]{a} g_{h_n}^{p}}[L^2(\S[d])] + \littleo_n(1) \ .
        \end{aligned}
    \end{equation}
    On the one hand, for every $q \in A$, due to the choice we have made of $\beta_q, \beta_{-q} \in \C$, Theorem \ref{Thm: characterization of SDM of ghQbeta}, equation \eqref{Eq: mu is lincomb of halfnuq scenario1}, it gives us
    \[
    \lim_{n \to \infty} \ip*{g_{h_n}^{P_q, \alpha^{q}}}{\Op[h_n]{a} g_{h_n}^{P_q, \alpha^{q}} }[L^2(\S[d])] = \int_{T^*\S[d]} a \, \Big( \frac{m_{q}}{\frac{ m_{q} + m_{-q} }{2}} \nu_{q, 1/2} + \frac{m_{-q}}{\frac{ m_{q} + m_{-q} }{2}} \nu_{-q, 1/2}\Big)
    \]
    On the other hand, for every $p \in B$, Corollary \ref{Cor: SDM of seq of single GF} gives us
    \[
    \lim_{n \to \infty} \ip*{ g_{h_n}^p}{\Op[h_n]{a} g_{h_n}^{p}}[L^2(\S[d])] = \int_{T^*\S[d]} a \, \nu_{q} \ .
    \]
    Taking limits $n \to \infty$ in \eqref{Eq: aux 105} we find that for every $a \in \CinfK(T^*\S[d])$,
    \begin{equation*}
    \lim_{n \to \infty} \ip*{g_{h_n}^{Q, \beta}}{\Op[h_n]{a} g_{h_n}^{Q, \beta}}[L^2(\S[d])] = \sum_{q \in A} \int_{T^*\S[d]} a \, \Big( m_{q} \nu_{q, 1/2} + m_{-q} \nu_{-q, 1/2}\Big) + \sum_{p \in B} \int_{T^*\S[d]} a \, \nu_{p} \ .
    \end{equation*}
    \qed

\subsection{Proof of \texorpdfstring{Theorem \ref{Thm: 2nd theorem old ef}}{Theorem intro old EF}}
\label{Sec: SDM old eigenfunctions}

Assume that $d \geq 2$. For any given measure $\mu \in \mathcal{P}_{\mathrm{inv}}(S^*\S[d])$, that is, a probability Radon measure $\mu$ on $T^*\S[d]$ supported inside $S^*\S[d]$ and invariant under the geodesic flow, we want to find sequences $(\ell_n)_{n \in \N} \nearrow +\infty$ and $(u_n)_{n \in \N}$ such that
\begin{equation} \label{Eq: conditions old EF vanishing on Q}
    \Lap u_n = \lambda_{\ell_n}^2 u_n \ , \qquad \norm{u_{n}}[L^2(\S[d])] = 1 \ , \qquad u_n(q) = 0 \quad \forall \, q \in Q \ , \qquad \forall \, n \in \N \ ,
\end{equation}
and for $h_n \coloneqq \lambda_{\ell_n}^{-1}$,
\begin{equation} \label{Eq: SDM of sequence un old EF}
    \lim_{n \to \infty} \ip{u_n}{ \Op[h_n]{a} u_n}[L^2(\S[d])] = \int_{T^*\S[d]} a \, \mu \qquad \forall \, a \in \CinfK(T^*\S[d]) \ .
\end{equation}
To this end, we follow the standard procedure used to show that any $\mu \in \mathcal{P}_{\mathrm{inv}}(S^*\S[d])$ is a \SDM{} of $\Lap$ on $\S[d]$. One first proves that for any linear convex combination of $\delta_{\gamma}$,
\begin{equation} \label{Eq: deltagamma definition}
    \int_{T^*\S[d]} a \, \delta_{\gamma} = \int_{0}^{2\pi} a(\gamma(s)) \frac{\D{s}}{2\pi} \ ,
\end{equation}
there exists a sequence $(u_n)_{n \in \N}$ satisfying \eqref{Eq: conditions old EF vanishing on Q} and \eqref{Eq: SDM of sequence un old EF}. Then, Krein-Milman theorem together with a diagonal extraction arguments shows that for any $\mu \in \mathcal{P}_{\mathrm{inv}}(S^*\S[d])$ there exists a sequence $(u_n)_{n \in \N}$ satisfying \eqref{Eq: conditions old EF vanishing on Q} and \eqref{Eq: SDM of sequence un old EF} (see the proof of Theorem 4 in \cite{Macia2008}). Therefore, we just need to focus on proving the following theorem
\begin{Theorem}
    Let $T = \sum_{\gamma \in F} b_{\gamma} \delta_{\gamma}$, for some finite collection of closed geodesics $F$ on $S^*\S[d]$ and coefficients $b_{\gamma} \in [0,1]$ such that $\sum_{\gamma \in F} b_{\gamma} = 1$. There exists $(u_n)_{n \in \N}$ satisfying \eqref{Eq: conditions old EF vanishing on Q} such that \eqref{Eq: SDM of sequence un old EF} holds for $\mu = T$.
\end{Theorem}

\begin{proof}
    Assume that $\S[d]$ is embedded in $\R[d+1]$ as the set of unit length vectors:
    \[
    \S[d] = \{ (x_1, \dots, x_{d+1}) \in \R[d+1] \colon \abs{x_1}^2 + \dots + \abs{x_{d+1}}^2 = 1\} \ .
    \]
    Let $P_0 \coloneqq \{x_3 = \dots = x_{d+1} = 0\}$, and let $L_0(x, \xi) \coloneqq \xi_2 x_1 - \xi_1 x_2$ be the normalized angular momentum function with respect to that plane $P$, and let $H(x, \xi) \coloneqq \abs{\xi}_x^2$. Let $\gamma_0$ be the unique closed geodesic on $S^*\S[d]$ such that $\mathrm{im}(\gamma) = H^{-1}(1) \cap L^{-1}(1)$. The following lemma is well-known (see \cite{JakobsonZelditch1999})
    \begin{Lemma}
        The sequence $(Y_{\ell}^{\gamma_0})_{\ell \in \N}$ given by
        \[
        Y_{\ell}^{\gamma_0} \coloneqq \frac{1}{2\pi^{\frac{d+1}{4}} } \bigg( \frac{\Gamma(\ell + \frac{d+1}{2})}{ \Gamma(\ell + 1) }\bigg)^{\frac{1}{2}} (x_1 + i x_2)^{\ell} \ ,
        \]
        enjoys the following properties:
        \begin{enumerate}
            \item $\norm{Y_{\ell}^{\gamma_0}}[L^2(\S[d])] = 1$ for all $\ell \in \N$.
            \item $\Lap Y_{\ell}^{\gamma_0} = \lambda_{\ell}^2 Y_{\ell}^{\gamma_0}$ for all $\ell \in \N$.
            \item If $h_{\ell} = \lambda_{\ell}^{-1}$, then
            \[
            \lim_{\ell \to \infty} \ip*{Y_{\ell}^{\gamma_0}}{ \Op[h_\ell]{a} Y_{\ell}^{\gamma_0}}[L^2(\S[d])] = \int_{T^*\S[d]} a \, \delta_{\gamma_0} \qquad \forall \, a \in \CinfK(T^*\S[d]) \ .
            \]
        \end{enumerate}
    \end{Lemma}
    Since the group of isometries of $\S[d]$ acts transitively on the space of geodesics of $S^*\S[d]$, for any closed geodesic $\gamma$ there exists an isometry $g$ such that $g \gamma_0 = \gamma$. Moreover, the eigenspaces $\ker(\Lap {}- \lambda_{\ell}^2)$ are invariant under the pull-back of any isometry of $\S[d]$, therefore, for any isometry $g$ such that $g \gamma_0 = \gamma$, the sequence $(Y_{\ell}^{\gamma_0} \circ g^{-1})_{\ell \in \N}$ enjoys the same properties as $(Y_{\ell}^{\gamma_0})_{\ell \in \N}$ except for
    \[
    \lim_{\ell \to \infty} \ip*{Y_{\ell}^{g \gamma_0} \circ g^{-1} }{ \Op[h_\ell]{a} Y_{\ell}^{g \gamma_0} \circ g^{-1} }[L^2(\S[d])] = \int_{T^*\S[d]} a \, \delta_{\gamma} \qquad \forall \, a \in \CinfK(T^*\S[d]) \ .
    \]
    We see that the sequence $(Y_{\ell}^{\gamma_0} \circ g_1^{-1})_{\ell \in \N}$ and $(Y_{\ell}^{\gamma_0} \circ g_2^{-1})_{\ell \in \N}$, have the same \SDM{} as long as $g_1 \gamma_0 = g_2 \gamma_0$, hence we will simply denote $Y_{\ell}^{\gamma}$ for any such sequence.
    
    Since for distinct geodesics $\gamma_1 \neq \gamma_2$ the measures $\delta_{\gamma_1}$ and $\delta_{\gamma_2}$ are mutually singular, Lemma \ref{Lemma: orthogonality for sequence from SDM} tells us that $(Y_{\ell}^{\mathcal{T}})_{\ell \in \N}$ given by $Y_{\ell}^{\mathcal{T}} \coloneqq \sum_{\gamma \in F} (b_{\gamma})^{1/2} Y_{\ell}^{\gamma}$ satisfies
    \[
    \lim_{\ell \to \infty} \norm*{Y_{\ell}^{\mathcal{T}}}[L^2(\S[d])] = 1 \ ,
    \]
    and thanks to Lemma \ref{Lemma: orthogonality for crossed SDM},
    \[
    \lim_{\ell \to \infty} \ip*{Y_{\ell}^{\mathcal{T}} }{ \Op[h_\ell]{a} Y_{\ell}^{\mathcal{T}}}[L^2(\S[d])] = \int_{T^*\S[d]} a \, \mathcal{T}  \qquad \forall \, a \in \CinfK(T^*\S[d]) \ .
    \]
    However, the eigenfunctions $Y_{\ell}^{\mathcal{T}}$ may not vanish on $Q$. We fix that making small perturbations to the functions $Y_{\ell}^{\mathcal{T}}$.

    \textbf{I.} Assume first that $\supp T \cap (\cup_{q \in Q} S_{q}^{*}\S[d]) = \varnothing$. This condition implies that for every $q \in Q$, $\gamma \in F$, there exists some $\eps_{\gamma, q} \in (0,1)$ such that
    \[
    \abs{Y_{\ell}^{\gamma}(q)} = \frac{1}{2\pi^{\frac{d+1}{4}} } \bigg( \frac{\Gamma(\ell + \frac{d+1}{2})}{ \Gamma(\ell + 1) }\bigg)^{\frac{1}{2}} (1 - \eps_{\gamma, q})^{\ell} \ .
    \]
    Since $\Gamma(\ell + \frac{d+1}{2}) \sim \ell^{\frac{d-1}{2}} \Gamma(\ell + 1)$ as $\ell \to \infty$, we read that
    \begin{equation} \label{Eq: estimate YlT on q}
        \abs{Y_{\ell}^{\gamma}(q)} \sim \frac{1}{2\pi^{\frac{d+1}{4}} } \ell^{\frac{d-1}{2}} (1 - \eps_{\gamma, q})^{\ell} \qquad \forall \, q \in Q \, , \quad \ell \in \N \ .
    \end{equation}

    For every $\ell \in \N$, let $\alpha_{\ell} \in \C[N]$ be such that
    \begin{equation} \label{Eq: aux 202}
        w_{\ell}^{\mathcal{T}} \coloneqq Y_{\ell}^{\mathcal{T}} - \sum_{q \in Q} \alpha_{\ell, q} z_{\ell}^{q} \qquad \implies \qquad w_{\ell}(p) = 0 \qquad \forall \, p \in Q \ .
    \end{equation}
    We conclude that $\alpha_{\ell, q}$ are the solutions to the following system of linear equations
    \begin{equation} \label{Eq: linear system for coefficients of perturbation}
        \sum_{q \in Q} \alpha_{\ell, q} z_{\ell}^{q}(p) = Y_{\ell}(p) \qquad \forall \, p \in Q \ . 
    \end{equation}
    If we define the matrix $\mathcal{Z}_{\ell} \coloneqq (z_{\ell}^{q}(p))_{p, q}$ and the vector $\vec{Y}_{\ell}^{\mathcal{T}} \coloneqq (Y_{\ell}^{\mathcal{T}}(p))_{p \in Q} \in \C[N]$, we have that $\alpha_{\ell}$ is the solution to the linear equation $\mathcal{Z}_{\ell} \alpha_{\ell} = \vec{Y}_{\ell}^{\mathcal{T}}$. Thanks to Lemma \ref{Lemma: properties of Zl matrix} we know that $\mathcal{Z}_{\ell}$ is invertible for all large enough $\ell$, and $\norm{\mathcal{Z}_{\ell}^{-1}} = \BigO(\ell^{ - \frac{d-1}{2}})$. Consequently, for all large enough $\ell$, \eqref{Eq: linear system for coefficients of perturbation} has a unique solution given by $\alpha_{\ell} = \mathcal{Z}_{\ell}^{-1} \vec{Y}_{\ell}^{\mathcal{T}}$, with the continuity inequality
    \begin{equation} \label{Eq: aux 101}
        \norm{\alpha_\ell} \leq \norm{\mathcal{Z}_{\ell}^{-1}}  \cdot \norm{\vec{Y}_{\ell}^{\mathcal{T}} } \lesssim \ell^{-\frac{d-1}{2}} \norm{\vec{Y}_{\ell}^{\mathcal{T}} } \ .
    \end{equation}
    Thanks to \eqref{Eq: estimate YlT on q}, we know that for some $\eps > 0$, $\norm{\vec{Y}_{\ell}^{\mathcal{T}}} \sim \ell^{\frac{d-1}{2}} (1 - \eps)^{\ell}$ as $\ell \to \infty$, therefore $\norm{\alpha_\ell} = \BigO( (1- \eps)^{\ell})$ as $\ell \to \infty$. From this estimate, we find that
    \[
    \norm*{w_{\ell}^{\mathcal{T}} - Y_{\ell}^{\mathcal{T}} }[L^2(\S[d])] \leq \sum_{q \in Q} \abs{\alpha_{\ell, q}} = \BigO((1 - \eps)^{\ell}) \ .
    \]
    This allows us to conclude that $\lim_{n \to \infty} \norm{w_{\ell}^{\mathcal{T}} }[L^2(\S[d])] = 1$ and
    \[
    \lim_{n \to \infty} \ip*{w_{\ell}^{\mathcal{T}} }{ \Op[h_{\ell}]{a} w_{\ell}^{\mathcal{T}} }[L^2(\S[d])] = \int_{T^*\S[d]} a \, \mathcal{T} \qquad \forall \, a \in \CinfK(S^*\S[d]) \ .
    \]

    \textbf{II.} Assume now that $\supp T \cap (\cup_{q \in Q} S_{q}^{*}\S[d]) \neq \varnothing$. Define the space of closed geodesics as the quotient space $\mathcal{G} = S^*\S[d] / \sim$ for the equivalence relation
    \[
    (x, \xi) \sim (y, \eta) \qquad \text{iff} \qquad \text{$\exists \, T \in [0, 2\pi)$ such that $\phi_T(x, \xi) = (y, \eta)$ \ .}
    \]
    We endow $S^*\S[d]$ with the Sasaki metric. $\mathcal{G}$ inherits a metric because every equivalence class is a closed subset of $S^*\S[d]$ for the Sasaki metric.
    
    If $F = \{ \gamma_1, \dots, \gamma_K\}$ for some $K \in \N$, for each $1\leq k \leq K$ let $(\gamma_{k, n})_{n \in \N}$ be a sequence of closed geodesics that converge to $\gamma_{k}$ in $\mathcal{G}$ and such that $\gamma_{k, n} \cap (\cup_{q \in Q} S_{q}^*\S[d]) = \varnothing$ for all $n \in \N$. Define $\mathcal{T}_{n} \coloneqq \sum_{k = 1}^{K} b_{\gamma_{k}} \delta_{\gamma_{k, n}}$ for all $n \in \N$. For each fixed $n \in \N$, $\supp \mathcal{T}_n \cap (\cup_{q \in Q} S_{q}^{*}\S[d]) = \varnothing$ holds, so part \textbf{I} provides us with a sequence of eigenfunctions $(w_{\ell}^{\mathcal{T}_n})_{\ell \in \N}$ that vanish on $Q$ for all $\ell \in \N$ and
    \[
    \lim_{n \to \infty} \ip*{w_{\ell}^{\mathcal{T}_n} }{ \Op[h_{\ell}]{a} w_{\ell}^{\mathcal{T}_n} }[L^2(\S[d])] = \int_{T^*\S[d]} a \, \mathcal{T}_n \qquad \forall \, a \in \CinfK(S^*\S[d]) \ .
    \]
    Let $D \subseteq \CinfK(S^*\S[d])$ be a countable subset that is dense in $\Co(S^*\S[d])$. Let us enumerate it as $D = \{ a_j \colon j \in \N\}$. We now perform a diagonal-extraction argument. Let us write for some measure $\nu$ on $T^*\S[d]$, $\ip{\nu}{a}[\Dist] \coloneqq \int_{T^*\S[d]} a \, \nu$.
    
    For $j = 1$ and $n = 1$, let $L_{1,1} \in \N$ be such that
    \[
    \abs*{ \ip*{w_{\ell}^{\mathcal{T}_1}}{\Op[h_{\ell}]{a_1} w_{\ell}^{\mathcal{T}_1}}[L^2(\S[d])] - \ip{ \mathcal{T}_1 }{a_1}[\Dist] } \leq 1 \qquad \forall \, \ell \geq L_{1,1} \ .
    \]
    Now for $n \geq 2$, we define $L_{1,n}$ iteratively as some integer $L_{1, n} > L_{1, n-1}$ such that
    \[
    \abs*{ \ip*{w_{\ell}^{\mathcal{T}_n}}{\Op[h_{\ell}]{a_1} w_{\ell}^{\mathcal{T}_n}}[L^2(\S[d])] - \ip*{\mathcal{T}_n}{a_1}[\Dist]} \leq \frac{1}{n} \qquad \forall \, \ell \geq L_{1,n} \ .
    \]
    In this way, for $j = 1$ we have defined a sequence of indices $(L_{1,n})_n \subseteq \N$ with the property that any sequence of eigenfunctions $(w_{\ell_n}^{\mathcal{T}_n})_n$ with $\ell_n \geq L_{1,n}$ satisfies 
    \[
    \abs*{ \ip*{w_{\ell_n}^{\mathcal{T}_n}}{\Op[h_{\ell_n}]{a_1} w_{\ell_n}^{\mathcal{T}_n} }[L^2(\S[d])] - \ip{\mathcal{T}_n}{a_1}[\Dist] } \leq \frac{1}{n} \qquad \forall \ n \in \N \ .
    \]
    We repeat this process for all $a_j \in D$, obtaining sequences of indices $(L_{j,n})_n \in \N$ with the property that for any $J \in \N$ and any sequence of eigenfunctions $(w_{\ell_n}^{\mathcal{T}_n})_n$ with $\ell_n \geq L_{J,n}$ satisfy 
    \[
    \abs*{ \ip*{w_{\ell_n}^{\mathcal{T}_n}}{\Op[h_{\ell_n}]{a_j} w_{\ell_n}^{\mathcal{T}_n} }[L^2(\S[d])] - \ip{\mathcal{T}_n}{a_j}[\Dist] } \leq \frac{1}{n} \qquad \forall \ n \in \N \ , \quad j \leq J \ .
    \]
    For $n \in \N$, let $\hat{L}_n \coloneqq L_{n, n} \geq n$ (recall that $L_{j,n} > L_{j,n-1}$ for all $n \in \N$). $(w_{\hat{L}_n}^{\mathcal{T}_n})_{n \in \N}$ is the sequence we were looking for (up to extraction of a subsequence). Certainly $w_{\hat{L}_n}^{\mathcal{T}_n}$ is an eigenfunction of $\Lap$ that vanishes on $Q$ with energy growing up to $+\infty$ as $n \to \infty$, thus \eqref{Eq: conditions old EF vanishing on Q} is satisfied. Now, let $h_n \coloneqq \lambda_{\hat{L}_n}^{-1}$. For all $a \in D$, adding and subtracting $\ip{\mathcal{T}_n}{a}[\Dist]$,
    \[
    \abs*{ \ip*{w_{\hat{L}_n}^{\mathcal{T}_n}}{\Op[h_n]{a} w_{\hat{L}_n}^{\mathcal{T}_n} }[L^2(\S[d])] - \ip{\mathcal{T}}{a}[\Dist] } \leq \frac{1}{n} + \abs*{\ip{\mathcal{T}_n - \mathcal{T}}{a}[\Dist] } \ .
    \]
    Using that $\mathcal{T}_n \to \mathcal{T}$ in the \weakstar topology because $\gamma_{k,n} \to \gamma_k$ in $\mathcal{G}$ for all $\gamma_k \in F$, we find that
    \[
    \lim_{n \to \infty} \ip*{w_{\hat{L}_n}^{\mathcal{T}_n}}{\Op[h_n]{a} w_{\hat{L}_n}^{\mathcal{T}_n} }[L^2(\S[d])] = \int_{T^*S[d]} a \, \mathcal{T} \qquad \forall \, a \in D \ .
    \]
    Since $D$ is dense in $\Co(T^*\S[d])$, we conclude that $\mathcal{T}$ is the unique semiclassical defect measure of the sequence $(w_{\hat{L}_n}^{\mathcal{T}_n})_{n \in \N}$ (up to extraction of a subsequence).
\end{proof}

\begin{Lemma} \label{Lemma: properties of Zl matrix}
    Assume that $d \geq 2$. The sequence of matrices $(\mathcal{Z}_{\ell})_{\ell \in \N}$, $\mathcal{Z}_{\ell} = (z_{\ell}^{q}(p))_{p, q}$, enjoys the following properties:
    \begin{enumerate}
        \item $\mathcal{Z}_{\ell}$ is real-valued and symmetric for all $\ell \in \N$.
        \item $\mathcal{Z}_{\ell}$ is diagonally dominant for all large enough $\ell$.
        \item $\mathcal{Z}_{\ell}$ is invertible for all large enough $\ell$ and $\norm{\mathcal{Z}_{\ell}} = \BigO(\ell^{\frac{d-1}{2}})$ as $\ell \to \infty$.
    \end{enumerate}
\end{Lemma}
\begin{proof}
    The fact that $\mathcal{Z}_{\ell}$ is real-valued and symmetric is a direct consequence of Proposition \ref{Prop: properties zonal harm}.
    
    Proposition \ref{Prop: Gegen poly and zonal harm} gives the following analytic description of normalized zonal harmonics 
    \[
    z_{\ell}^{q} (x) = \frac{1}{\norm{Z_{\ell}^{q}}[L^2(\S[d])]} Z_{\ell}^{q} (x) = \frac{\sqrt{m_{\ell}}}{\sqrt{\vol(\S[d])}} \big[ C_{\ell}^{(\frac{d-1}{2})} (1) \big]^{-1} C_{\ell}^{(\frac{d-1}{2})} (\cos r) \qquad r = d(x, q) \ ,
    \]
    where $\{C_{\ell}^{(\alpha)}\}_{\ell \in \N}$ is the family Gegenbauer polynomials of order $\alpha$.
    
    On the one hand, diagonal entries of $\mathcal{Z}_{\ell}$ take the value
    \[
    z_{\ell}^{q} (q) = \norm*{Z_{\ell}^{q}}[L^2(\S[d])] = \frac{\sqrt{m_{\ell}}}{\sqrt{\vol(\S[d])}} \gtrsim \ell^{\frac{d-1}{2}} \ . 
    \]
    On the other hand, the off-diagonal entries are are asymptotically smaller because $d \geq 2$. Thanks to \eqref{Eq: value of Gegen poly at 1} and Proposition \ref{Prop: pointwise asymptotics of Gegen poly} for $\alpha = \tfrac{d-1}{2}$ we obtain
    \[
    \abs*{ z_{\ell}^{q}(p) } \lesssim \sqrt{m_{\ell}} \frac{\Gamma(\ell+1)}{\Gamma(\ell + d - 1)} \ell^{\frac{d-1}{2} - 1} \qquad \forall \, p \neq q \ .
    \]
    We can use the asymptotics for quotients of Gamma functions \eqref{Eq: asymp quotient Gamma function} to get the following asymptotic upper bound
    \[
    \abs*{ z_{\ell}^{q}(p) } \lesssim \ell^{\frac{d-1}{2}} \ell^{2-d} \ell^{\frac{d-1}{2} - 1} = 1 \qquad \forall \, p \neq q \ .
    \]
    Therefore, for $d \geq 2$, the matrix $\mathcal{Z}_{\ell}$ is diagonally dominant for all large enough $\ell \in \N$.
    
    Lastly, Gershgorin circle theorem (Theorem \ref{Thm: Gershgorin circle theorem} below) implies that, for large enough $\ell \in \N$, $\Spec(\mathcal{Z}_{\ell})$ is contained in a finite union of disks of bounded radius centered on the diagonal elements $z_{\ell}^{q}(q) \gtrsim \ell^{\frac{d-1}{2}}$. Therefore, for all large enough $\ell$, $\mathcal{Z}_{\ell}$ is invertible and $\norm{\mathcal{Z}_{\ell}^{-1}} \lesssim \ell^{-\frac{d-1}{2}}$.
\end{proof}
\begin{Theorem}[Gershgorin circle Theorem] \label{Thm: Gershgorin circle theorem}
    Let $A = (a_{ij})$ be a complex $N \times N$ matrix. For every $p \in \{1, \dots, N\}$, let $R_p = \sum_{j \neq p} \abs{a_{pj}}$ the sum of the absolute values of the non-diagonal entries of the $p$-th row, and $\mathbb{D}_p = D(a_{pp}, R_p)$ a disk in $\C$ of radius $R_p$ centered on $a_{pp}$ (Greshgorin disk). Then $\Spec(A) \subseteq \cup_{p = i}^{N} \mathbb{D}_p$.
\end{Theorem}

\subsection{Proof of \texorpdfstring{Theorem \ref{Thm: main theorem}}{main theorem}}
\label{Sec: proof of main theorem}

Assume that $d = 2,3$, and let $Q$ be a finite subset of $\S[d]$.

\textbf{Assume that $Q$ contains no pair of antipodal points.} Let $(\Lap_{L_n})_{n \in \N}$ be a sequence of point-perturbations of $\Lap$ on the set $Q$. Thanks to Theorem \ref{Thm: 2nd theorem old ef} we know that
\[
\mathcal{P}_{\mathrm{inv}}(S^*\S[d]) \subseteq \mathcal{M}_{\mathrm{sc}}(\Lap_{L_n}) \ .
\]
Let us show $\mathcal{M}_{\mathrm{sc}}(\Lap_{L_n}) \subseteq \mathcal{P}_{\mathrm{inv}}(S^*\S[d])$. Let $(h_n)_{n \in \N} \subseteq (0,1)$, $h_n \to 0^+$, and $(u_n)_{n \in \N} \subseteq D(\Lap_{L_n})$ such that
\[
(h_n^2 \Lap_{L_n}) u_n = u_n \ , \qquad \norm*{u_n}[L^2(\S[d])] = 1 \qquad \forall \, n \in \N \ .
\]
In addition, assume there exists a unique $\mu \in \mathcal{M}_{\mathrm{sc}}(\Lap_{L_n})$ such that
\[
\lim_{n \to \infty} \ip*{u_n}{\Op[h_n]{a} u_n}[L^2(\S[d])] = \int_{T^*\S[d]} a \, \mu \qquad \forall \, a \in \CinfK(T^*\S[d]) \ .
\]
We want to show that $\mu \in \mathcal{P}_{\mathrm{inv}}(S^*\S[d])$.

Thanks to equation \ref{Eq: eigenfunctions of LapL}, for every $n \in \N$ there exist
\begin{enumerate}
    \item $w_{h_n} \in \ker(h_n^2 \Lap {}- 1)$,
    \item $\beta_{n} \in \C[N]$, $N \coloneqq \card Q$, such that $\sum_{q \in Q} \beta_{q} \delta_{q} \equiv 0$ on $\ker(h_n^2 \Lap {}- 1)$.
\end{enumerate}
such that
\[
u_{n} = w_{h_n} + G_{h_n}^{Q, \beta_{n}} \in D(\Lap_{L_n}) \ ,
\]
and $G_{h}^{Q, \beta}$ was defined in Section \ref{Sec: Point-perturbations of Laplacian on the spheres}.

We observe that for $\beta \in \C[N]$ and $\ell \in \N$,
\[
\sum_{q \in Q} \beta_{q} \delta_{q} \equiv 0 \quad \text{on $\ker(\Lap {}- \lambda_{\ell}^2)$} \qquad \iff \qquad \sum_{q \in Q} \beta_{q} Z_{\ell}^{q} = 0 \quad \text{in $L^2(\S[d])$} \ .
\]
Therefore, Proposition \ref{Prop: linindep of zonal harm} implies that there exists some $\nu \in \N$ such that for every $n \geq \nu$, either $u_{n} = w_{h_n}$ and $h_n^{-2} \in \Spec(\Lap)$, or $u_{n} = g_{h_n}^{Q, \beta_n}$ and $h_n^{-2} \notin \Spec(\Lap)$. At least one of these possibilities must occur for an infinite collection of $n \in \N$.

If $u_{n} = w_{h_n}$ for an infinite collection of $n$, then along this subsequence we have
\[
\lim_{n \to \infty} \ip*{w_{h_n}}{\Op[h_n]{a} w_{h_n}}[L^2(\S[d])] = \int_{T^*\S[d]} a \, \mu \qquad \forall \, a \in \CinfK(T^*\S[d]) \ .
\]
Since $w_{h_n} \in \ker(h_n^2 \Lap {}- 1)$ for all $n \in \N$, then $\mu \in \mathcal{P}_{\mathrm{inv}}(S^*\S[d])$.

If $u_n = g_{h_n}^{Q, \beta_{n}}$ for an infinite collection of $n$, then along this subsequence we have
\[
\lim_{n \to \infty} \ip*{g_{h_n}^{Q, \beta_{n}}}{\Op[h_n]{a} g_{h_n}^{Q, \beta_{n}}}[L^2(\S[d])] = \int_{T^*\S[d]} a \, \mu \qquad \forall \, a \in \CinfK(T^*\S[d]) \ .
\]
Thanks to Theorem \ref{Thm: characterization of SDM of ghQbeta general case} we know that $\mu = \sum_{q \in Q} m_q \nu_{q}$, for some $m_q \in [0,1]$ for which $\sum_{q} m_{q} = 1$, and some probability measures $\nu_q \in \mathcal{P}_{\mathrm{inv}}(S^*\S[d])$ (see \eqref{Eq: def of nuq}). Consequently, $\mu \in \mathcal{P}_{\mathrm{inv}}(S^*\S[d])$ as well.

\textbf{Assume that $Q$ contains a pair of antipodal points.} Assume that
\[
Q = \{q_1, \dots, q_{K}, -q_{1}, \dots, -q_{K}, p_1, \dots, p_J\}
\]
for some $K \geq 1$. Let $m \in [0,2]^N$, with $m_{k}^{+} = \frac{1}{K}$ and $m_{k}^{-} = 0$ for all $1 \leq k \leq K$, and $m_j = 0$ for all $1 \leq j \leq J$. Thanks to Theorem \ref{Thm: 2nd theorem new ef} there exists a sequence of point-perturbations on $Q$, $(\Lap_{L_n})_{n \in \N}$, such that $\nu_{Q, m} \in \mathcal{M}_{\mathrm{sc}}(\Lap_{L_n})$. Due to the choice we have made of $m$, the measure $\nu_{Q, m} \notin \mathcal{P}_{\mathrm{inv}}(S^*\S[d])$. Meanwhile, Theorem \ref{Thm: 2nd theorem old ef} implies that $\mathcal{P}_{\mathrm{inv}}(S^*\S[d]) \subseteq \mathcal{M}_{\mathrm{sc}}(\Lap_{L_n})$.

\appendix

\section{Zonal harmonics}
\label{App: zonal harmonics}

We state some well-known facts about eigenspaces of $\Lap$ on $\S[d]$ and zonal harmonics; we refer the reader to \cite[Chapter IV]{SteinWeiss1971} for proofs.

For every $\ell \in \N$, the eigenspace $E_{\ell} \coloneqq \ker(\Lap {}- \lambda_{\ell}^2)$ is the space of all harmonic homogeneous polynomials of degree $\ell$ in $\R[d+1]$ restricted to $\S[d]$. The dimension of $E_{\ell}$ is expressed as a polynomial in $\ell$ of degree $d-1$:
\begin{equation} \label{Eq: dim(El)}
    \dim(E_{\ell}) = \binom{\ell + d}{\ell} - \binom{\ell - 2 + d}{\ell - 2} \ .
\end{equation}
For instance, $\dim(E_{\ell}) = 2l+1$ if $d = 2$, and $\dim(E_{\ell}) = (\ell+1)^2$ if $d = 3$.

For every $q \in \S[d]$ and every $\ell \in \N$, since $E_{\ell}$ is a finite-dimensional subspace of (smooth) functions, there exists a unique function $Z_{\ell}^{q} \in E_{\ell}$ such that
\[
\ip{Z_{\ell}^{q}}{u} = u(q) \qquad \forall \, u \in E_{\ell} \ .
\]
$Z_{\ell}^{q}$ is known as the $\ell$-th zonal harmonic with respect to $q$. We collect the properties of our interest in the following proposition

\begin{Prop} \label{Prop: properties zonal harm}
    Assume $d \geq 1$. Let $q \in \S[d]$ and $\ell \in \N$. The $\ell$-th zonal harmonic with respect to $q$, $Z_{\ell}^{q}$, enjoy the following properties:
    \begin{enumerate}
        \item $Z_{\ell}^{q}$ is real-valued and $Z_{\ell}^{q}(p) = Z_{\ell}^{p}(q)$ for all $p \in \S[d]$.
        \item $Z_{\ell}^{q}(q) = \sup \{ u(q) \colon u \in E_{\ell} \, , \ \norm{u}[L^2(\S[d])] \leq \norm{Z_{\ell}^{q}}[L^2(\S[d])]\} = \norm{Z_{\ell}^{q}}^2 = \frac{\dim(E_\ell)}{\vol(\S[d])}$.
        \item If $\varphi$ is an isometry of $\S[d]$ then $\varphi^* (Z_{\ell}^{\varphi(q)}) = Z_{\ell}^{q}$. In particular, if $\varphi(q) = q$, then $\varphi^*(Z_{\ell}^{q}) = Z_{\ell}^{q}$.
        \item $\hat{L} Z_{\ell}^{q} = 0$ for any Killing vector field $\hat{L}$ on $\S[d]$ that vanishes on $q$.
        \item If $-q$ is the antipodal point of $q$ in $\S[d]$, $Z_{\ell}^{-q} = (-1)^{\ell} Z_{\ell}^{q}$.
    \end{enumerate}
\end{Prop}
\begin{proof}
    We just prove items (1), (4), and (5), as they are not proved in \cite{SteinWeiss1971}

    (1) First, since $\Lap$ maps real functions into real functions, $\Im Z_{\ell}^{q}$ belongs to $E_{\ell}$ as well. Therefore, for the imaginary part $\Im Z_{\ell}^{q}$,
    \[
    [\Im Z_{\ell}^{q}](q) = \ip{Z_{\ell}^{q}}{\Im Z_{\ell}^{q}} = \ip{\Re Z_{\ell}^{q}}{\Im Z_{\ell}^{q}} + i \norm{\Im Z_{\ell}^{q} }^2 \ .
    \]
    However, since the left-hand side is real, we infer that $\Im Z_{\ell}^{q} = 0$, thus $Z_{\ell}^{q}$ is real-valued. As a consequence
    \[
    Z_{\ell}^{q}(p) = \ip*{Z_{\ell}^{p}}{Z_{\ell}^{q}}[L^2(\S[d])] = \ip*{Z_{\ell}^{q}}{Z_{\ell}^{p}}[L^2(\S[d])] = Z_{\ell}^{p}(q) \ .
    \]

    (4) By elliptic regularity of $\Lap$, we know that $Z_{\ell}^{q}$ is smooth, thus $[\hat{L} Z_{\ell}^{q}]$ is smooth too. Moreover, since $\hat{L}$ is Killing, the flow it generates, $\varphi_t$, is an isometry for every $t \in \R$. In addition, $\varphi_t(q) = q$ for all $t \in \R$, thus $(\varphi_t)^*(Z_{\ell}^{q}) = Z_{\ell}^{q}$. A a consequence, for every $y \in \S[d]$,
    \[
    [\hat{L} Z_{\ell}^{q}](y) = \lim_{t \to 0} \frac{(\varphi_t)^*Z_{\ell}^{q}(y) - Z_{\ell}^{q}(y)}{t} = 0 \ .
    \]

    (5) This a direct consequence of Proposition \ref{Prop: Gegen poly and zonal harm} and Proposition \ref{Prop: parity of Gegenbauer polynomial} using that
    \[
    \cos \big( d(x, -q) \big) = \cos \big( \pi - d(x, q) \big) = -\cos \big( d(x, q) \big) \qquad \forall \, x \in \S[d] \ . \qedhere
    \]
\end{proof}

\section{Gegenbauer polynomials}
\label{App: Gegenbauer polynomials}

\subsection{Properties of Gegenbauer polynomials}

We collect and prove here some facts about the family of Gegenbauer (ultraspherical) polynomials $(C_{\ell}^{(\alpha)})_{\ell \in \N}$, $\alpha > 0$. For every $\alpha > 0$, $(C_{\ell}^{(\alpha)})_{\ell \in \N}$ is a family of orthogonal polynomials on $[-1, 1]$ with respect to the weight $w(s) = (1 - s^2)^{\alpha - \frac{1}{2}}$; they are normalized by the following two properties: $C_{\ell}^{(\alpha)}$ is a polynomial of degree $\ell$, and 
\begin{equation} \label{Eq: value of Gegen poly at 1}
    C_{\ell}^{(\alpha)}(1) = \frac{\Gamma(2 \alpha + \ell)}{\Gamma(2\alpha) \Gamma(\ell+1)} \ , \qquad \forall \, \ell \in \N \ .
\end{equation}
They generalize the families of Legendre polynomials ($\alpha = \frac{1}{2}$) and Chebyshev polynomials of the second kind ($\alpha = 1$).

The family $(C_{\ell}^{(\alpha)})_{\ell \in \N}$ has other useful characterizations. For instance, for every $\ell \in \N$, $C_{\ell}^{(\alpha)}$ is the regular solution to the linear differential equation in $y(s)$
\begin{equation} \label{Eq: Gegenbauer diff eq}
    (1 - s^2) y'' - (2 \alpha + 1)s y' + \ell(\ell + 2\alpha) y = 0 \ , \qquad s \in (-1, 1) \ .
\end{equation}
Furthermore, for $\abs{s} \leq 1$, $(C_{\ell}^{(\alpha)}(s))_{\ell \in \N}$ are the coefficients of a certain analytic function in $\abs{x} < 1$:
\begin{equation} \label{Eq: generating function for Gegen. poly}
    \frac{1}{(1 - 2sx + x^2)^\alpha} = \sum_{\ell = 0}^{\infty} C_{\ell}^{(\alpha)}(s) x^{\ell} \ .
\end{equation}
We say that $x \mapsto (1 - 2sx + x^2)^{-\alpha}$ on $\abs{x} < 1$ is the \emph{analytic generating function} of the family $(C_{\ell}^{(\alpha)})_{\ell \in \N}$.

Gegenbauer polynomials $C_{\ell}^{(\alpha)}$ and zonal harmonics $Z_{\ell}^{q}$ are intimately related, as the following proposition shows.

\begin{Prop} \label{Prop: Gegen poly and zonal harm}
    Assume $d \geq 1$ and let $q \in \S[d]$. For every $\ell \in \N$, the following holds 
    \begin{equation} \label{Eq: identity zonal harm - Gegen poly}
        Z_{\ell}^{q} (x) = \frac{\dim(E_{\ell})}{\vol(\S[d])} \big[ C_{\ell}^{(\frac{d-1}{2})} (1) \big]^{-1} C_{\ell}^{(\frac{d-1}{2})} (\cos r) \ , \qquad r = d(x, q) \ .
    \end{equation}
\end{Prop}

\begin{proof}
    Proposition \ref{Prop: properties zonal harm} tells us that zonal harmonics depend only on the distance to $q$, therefore, for every $\ell \in \N$, we may write $Z_{\ell}^{q} (x) = v_{\ell}(r)$, $r = d(x, q)$, for some smooth function $v_{\ell}$ defined on $[0, \pi]$. Since $\Lap Z_{\ell}^{q} = \lambda_{\ell}^2 Z_{\ell}^{q}$, we infer that $v_{\ell}$ satisfies
    \begin{equation} \label{Eq: EV problem for vl}
        \begin{cases}
            - \frac{1}{\sin^{d-1}(r)} \partial_r\Big( \sin^{d-1}(r) \Big) v_{\ell} - \frac{(d - 1) \cos r}{\sin r} \partial_r v_{\ell} = \ell(\ell+ d - 1) v_{\ell} \qquad (0, \pi) \\
            v_{\ell}'(0) = v_{\ell}'(\pi) = 0
        \end{cases}
    \end{equation}
    Setting $v_{\ell}(r) = y_{\ell}(\cos r)$ for some soooth function $y_{\ell}(s)$, we get that $y_{\ell}$ must satisfy 
    \begin{equation} \label{Eq: Gegenbaur diff eq with BC}
        \begin{cases}
            (1 - s^2)\, y_{\ell}'' - d s \, y_{\ell}' + \ell(\ell + d - 1) y_{\ell} = 0 \qquad (-1, 1) \\
            \lim_{r \to 0^+} y_{\ell}'(\cos r) \sin r = 0 \\
            \lim_{r \to \pi^+} y_{\ell}'(\cos r) \sin r = 0
        \end{cases}
    \end{equation}
    This is the Gegenbauer differential equation for $\alpha = \frac{d-1}{2}$, and the boundary conditions in $s = -1$ and $s = 1$ force the solutions to be smooth in $[-1, 1]$, thus $y_{\ell}$ must be a scalar multiple of the Gegenbauer polynomial $C_{\ell}^{(\frac{d-1}{2})}$. As a consequence, for every $\ell \in \N$, there exists $c_{l, d} \in \R$ such that
    \[
    Z_{\ell}^{q}(x) = c_{l, \alpha} \, C_{\ell}^{(\frac{d-1}{2})} (\cos r) \ , \qquad r = d(x, q) \ .
    \]
    Imposing that $Z_{\ell}^{q}(q) = \frac{\dim(E_{\ell})}{\vol(\S[d])}$ (see Proposition \ref{Prop: properties zonal harm}), we get what was claimed.
\end{proof}

Most, if not all, of the properties of Gegenbauer polynomials may be inferred from its generating function.
\begin{Prop} \label{Prop: parity of Gegenbauer polynomial}
    For every $\alpha > 0$, and every $\ell \in \N$,
    \[
    C_{\ell}^{(\alpha)}(-s) = (-1)^{\ell} C_{\ell}^{(\alpha)}(s) \qquad \forall \ s \in (-1, 1) \ .
    \]
\end{Prop}
\begin{proof}
    Using the generating function of $C_{\ell}^{(\alpha)}$,
    \[
    \sum_{\ell = 0}^{\infty} (-1)^{\ell}C_{\ell}^{(\alpha)}(s) x^{\ell} = \frac{1}{(1 - 2s(-x) + (-x)^2)^\alpha} = \frac{1}{(1 - 2(-s)x + x^2)^\alpha} = \sum_{\ell = 0}^{\infty} C_{\ell}^{(\alpha)}(-s) x^{\ell} \ .
    \]
    Identifying coefficients of the same order leads to the desired identities.
\end{proof}

In Section \ref{Sec: SDM Green functions} we were interested in the existence of ladder operators for the family $(C_{\ell}^{(\alpha)})_{\ell \in \N}$; here is the proof of that fact.
\begin{Prop} \label{Prop: ladder operators for Gegenbauer poly}
    Let $\alpha > 0$. For every integer $\ell \geq 1$, there are first order differential operators on $(-1, 1)$,
    \begin{equation}
        A_{\ell}^{\alpha, +} \coloneqq \frac{\ell + 2\alpha}{\ell + 1} s + \frac{s^2 - 1}{\ell + 1}\partial_s \ , \qquad A_{\ell}^{\alpha, -} \coloneqq \frac{\ell}{\ell + 2\alpha - 1}s - \frac{s^2-1}{\ell + 2\alpha - 1}\partial_s \ ,
    \end{equation}
    such that
    \[
    A_{\ell}^{\alpha, +} C_{\ell}^{(\alpha)} = C_{\ell+1}^{(\alpha)} \ , \qquad A_{\ell}^{\alpha, - } C_{\ell}^{(\alpha)} = C_{\ell-1}^{(\alpha)} \ .
    \]
\end{Prop}

\begin{proof}
    We distinguish two cases: $\alpha \neq 1$ and $\alpha = 1$. We will deal with the case $\alpha = 1$ later.
    
    First we fix $s \in (-1,1)$. The generating function for Gegenbauer polynomials is analytic on $\abs{x} < 1$, thus we can differentiate it with respect to $x$ under the series sign, obtaining
    \[
    \frac{2\alpha(s - x)}{(1 - 2sx + x^2)^{\alpha}} = (1 - 2sx + x^2) \sum_{\ell = 0}^{\infty} (\ell+1) C_{\ell+1}^{(\alpha)}(s) x^\ell \qquad \forall \ x \in (-1, 1) \ ,
    \]
    which is equivalent in terms of generating functions
    \begin{multline*}
        (2\alpha s)C_0^{(\alpha)}(s) + (2\alpha) [sC_1^{(\alpha)}(s) - C_0^{(\alpha)}(s)] x + \sum_{\ell = 2}^{\infty} (2\alpha)[s \cdot C_{\ell}^{(\alpha)}(s) - C_{\ell-1}^{(\alpha)}(s) ] x^{\ell} = \\ C_1^{(\alpha)}(s) + [2C_2^{(\alpha)}(s) -2s C_1^{(\alpha)}(s)]x + \sum_{\ell = 2}^{\infty} [ (\ell+1) C_{\ell+1}^{(\alpha)}(s) -2 \ell s C_{\ell}^{(\alpha)}(s) + (\ell-1)C_{\ell-1}^{(\alpha)}(s)] x^{\ell} \ .
    \end{multline*}
    Identifying same degree coefficients of the series for $\ell \geq 1$ we get
    \begin{equation} \label{Eq: Bonnet rec formula}
        (\ell+1) C_{\ell+1}^{(\alpha)}(s) - 2(\ell+\alpha) s C_\ell^{(\alpha)}(s) + (\ell + 2\alpha - 1) C_{\ell-1}^{(\alpha)}(s) = 0 \ , \qquad \forall \ s \in (-1, 1) \ ,
    \end{equation}
    known as Bonnet’s recursion formula.
    
    Now we work the other way around and start fixing $x \in (-1, 1)$. Since the generating function is analytic in $\abs{s} < 1$, we can differentiate with respect to $s$ under the series sign, obtaining
    \[
    \frac{2\alpha \, x}{(1 - 2sx + x^2)^{\alpha}} = (1 - 2sx + x^2) \sum_{\ell = 0}^{\infty} \partial_s C_{\ell}^{(\alpha)}(s) x^{\ell} \qquad \forall \ s \in (-1, 1) \ .
    \]
    Since this holds for every $x \in (-1, 1)$, this can be translated in terms of generating functions
    \begin{multline*}
        2\alpha \, C_0^{(\alpha)}(s) x + \sum_{\ell = 2}^{\infty} 2\alpha \, C_{\ell-1}^{(\alpha)}(s) x^{\ell} = \\
        \partial_s C_0^{(\alpha)}(s) + [\partial_s C_1^{(\alpha)}(s) -2s \partial_s C_0^{(\alpha)}(s)]x + \sum_{\ell = 2}^{\infty} [\partial_s C_{\ell}^{(\alpha)}(s) - 2s \partial_s C_{\ell-1}^{(\alpha)}(s) + \partial_s C_{\ell-2}^{(\alpha)}(s)] x^{\ell} \ ,
    \end{multline*}
    which leads to the identity
    \begin{equation*} 
        \partial_s C_{\ell+1}^{(\alpha)} - 2s\ \partial_s C_{\ell}^{(\alpha)} + \partial_s C_{\ell-1}^{(\alpha)} = 2\alpha \ C_{\ell-1}^{(\alpha)} \qquad \forall \ \ell \geq 1 \ .
    \end{equation*}
    This identity has a similar flavor to Bonnet's recursion formula \eqref{Eq: Bonnet rec formula}, thus we differentiate \eqref{Eq: Bonnet rec formula}:
    \[
    (\ell+1) \ \partial_s C_{\ell+1}^{(\alpha)} - 2(\ell+\alpha) s \ \partial_s C_{\ell}^{(\alpha)} + (\ell + 2\alpha - 1) \ \partial_{s}C_{\ell-1}^{(\alpha)} = 2(\ell + \alpha) \ C_{\ell}^{(\alpha)} \qquad \forall \ \ell \geq 1 \ .
    \]
    We can play with these identities to get another one with either $\partial_s C_{\ell+1}^{(\alpha)}$ or $\partial_{s}C_{\ell-1}^{(\alpha)}$ removed. From that, we get the system of differential identities
    \[
    \begin{cases}
        2(\alpha - 1)s \ \partial_s C_{\ell}^{(\alpha)} - 2(\alpha - 1) \ \partial_s C_{\ell-1}^{(\alpha)} = 2(\alpha - 1)\ell \ C_{\ell}^{(\alpha)} \\
        2(\alpha - 1) \ \partial_{s}C_{\ell+1}^{(\alpha)} - 2(\alpha - 1) s \ \partial_{s} C_{\ell}^{(\alpha)} = 2(\alpha - 1) (\ell+2\alpha) \ C_{\ell}^{(\alpha)}
    \end{cases} \qquad \forall \ \ell \geq 1 \ .
    \]
    Since we are assuming that $\alpha \neq 1$, this system simplifies into
    \[
    \begin{cases}
        s \ \partial_s C_{\ell}^{(\alpha)} - \partial_s C_{\ell-1}^{(\alpha)} = \ell \ C_{\ell}^{(\alpha)} \\
        \partial_{s}C_{\ell+1}^{(\alpha)} - s \ \partial_{s} C_{\ell}^{(\alpha)} = (\ell+2\alpha) \ C_{\ell}^{(\alpha)}
    \end{cases} \qquad \forall \ \ell \geq 1 \ .
    \]
    With the help of index-shifting, we can solve this system of equations for $C_{\ell-1}^{(\alpha)}$ and $C_{\ell+1}^{(\alpha)}$, obtaining the identities
    \begin{align*}
        C_{\ell+1}^{(\alpha)} & = \Big[\frac{\ell + 2\alpha}{\ell+1} s + \frac{s^2 - 1}{\ell+1} \partial_s \Big]C_{\ell}^{(\alpha)} \\
        C_{\ell-1}^{(\alpha)} & = \Big[ \frac{\ell}{\ell+2\alpha - 1} s - \frac{s^2 - 1}{\ell+2\alpha - 1} \partial_s \Big] C_{\ell}^{(\alpha)}
    \end{align*}
    These prove what we wanted for $\alpha \neq 1$.

    We work the details of the case $\alpha = 1$ now. The Gegenbauer polynomials for $\alpha = 1$ are the Chebyshev polynomials of the second kind, $U_{\ell}$. We proceed differently, exploiting the definition of $U_{\ell}(\cos r)$ instead of $U_{\ell}(s)$:
    \[
    U_{\ell}(\cos r) \coloneqq \frac{\sin((\ell+1) r)}{\sin r} \ .
    \]
    If we denote $T_{\ell}(\cos r) \coloneqq \cos(\ell \, r)$, Chebyshev polynomial of the first kind, one easily checks that
    \[
    \sin r \, \partial_r [U_{\ell}(\cos r)] = (\ell+1) T_{\ell+1}(\cos r) - \cos r \, U_{\ell}(\cos r) \ .
    \]
    On the other hand, using that $\sin(\alpha + \beta) = \sin \alpha \ \cos \beta + \cos \alpha \ \sin \beta$, we have that 
    \[
    \begin{cases}
        U_{\ell+1}(\cos r) = \cos r \, U_{\ell}(\cos r) + T_{\ell+1}(\cos r) \\
        U_{\ell-1}(\cos r) = \cos r \, U_{\ell}(\cos r) - T_{\ell+1}(\cos r)
    \end{cases}
    \]
    Substituting the former identity in the latter ones, we obtain
    \begin{align*}
        U_{\ell+1}(\cos r) & = \frac{\ell+2}{\ell+1} \cos r \, U_{\ell}(\cos r) + \frac{\sin r}{\ell+1} \partial_r [U_{\ell}(\cos r)] \\
        U_{\ell-1}(\cos r) & = \frac{\ell}{\ell+1} \cos r \, U_{\ell}(\cos r) - \frac{\sin r}{\ell+1} \partial_r [U_{\ell}(\cos r)]
    \end{align*}
    Finally, performing the change of variables $s = \cos r$.
    \begin{align*}
        U_{\ell+1} & = \frac{\ell+2}{\ell+1} s \, U_{\ell} + \frac{s^2 - 1}{\ell+1} \partial_s U_{\ell} \\
        U_{\ell-1} & = \frac{\ell}{\ell+1} s \, U_{\ell} - \frac{s^2 - 1}{\ell+1} \partial_s U_{\ell}
    \end{align*}
    This proves the $\alpha = 1$ case.
\end{proof}

In Lemma \ref{Lemma: properties of Zl matrix} we claimed that the matrix of coefficients of equation \ref{Eq: linear system for coefficients of perturbation} was diagonally dominant; in the proof of Lemma \ref{Lemma: properties of Zl matrix} we invoked the following proposition.
\begin{Prop} \label{Prop: pointwise asymptotics of Gegen poly}
    Let $\alpha > 0$. For every $\delta > 0$, the following asymptotic is uniform on $\theta \in [\delta, \pi - \delta]$
    \[
    C_{\ell}^{(\alpha)}(\cos \theta) = \ell^{\alpha - 1} \Big( B(\alpha, \tfrac{1}{2}) \frac{2^{\alpha} \Gamma(\alpha)}{\Gamma(2\alpha)} (\sin \theta)^{-\alpha} \cos [(\ell+\alpha)(\theta - \pi)] + \BigO(\ell^{-1}) \Big) \qquad \text{as $\ell \to \infty$} \ .
    \]
\end{Prop}

\begin{proof}
    We combine several identities from \cite{NISTHandbook2010}. From \cite[(18.7.1)]{NISTHandbook2010} we have that
    \begin{equation} \label{Eq: aux 10}
        C_{\ell}^{(\alpha)} = \frac{ \Gamma(\alpha + \frac{1}{2}) }{\Gamma(2\alpha)} \frac{\Gamma(\ell + 2\alpha) }{ \Gamma(\ell + \alpha + \frac{1}{2})} P_\ell^{(\alpha - \frac{1}{2}, \alpha - \frac{1}{2})} \ ,
    \end{equation}
    where $P_{\ell}^{(p, q)}$ is the $\ell$-th Jacobi polynomial with indices $p$ and $q$.

    On the other hand, Jacobi polynomials $P_{\ell}^{(\alpha - \frac{1}{2}, \alpha - \frac{1}{2})}(\cos \theta)$ have the following asymptotic for $\theta \in [\delta, \pi - \delta]$ \cite[(18.15.1)]{NISTHandbook2010} as $\ell \to \infty$,
    \begin{equation} \label{Eq: aux 11}
        \Big( \sin \tfrac{\theta}{2} \cos \tfrac{\theta}{2} \Big)^{\alpha} P_\ell^{(\alpha - \frac{1}{2}, \alpha - \frac{1}{2})} (\cos \theta) = \frac{2^{2(\ell + \alpha)}}{\pi} B(\ell + \alpha + \tfrac{1}{2}, \ell + \alpha + \tfrac{1}{2}) \Big( \cos \big[ (\ell + \alpha) (\theta - \pi) \big] + \BigO(\ell^{-1}) \Big) \ ,
    \end{equation}
    where $B(x, y)$ is the beta function
    \[
    B(x, y) \coloneqq \frac{\Gamma(x) \Gamma(y)}{\Gamma(x + y)} \ .
    \]
    Putting \eqref{Eq: aux 10} and \eqref{Eq: aux 11} together, we have the following asymptotic for $C_{\ell}^{(\alpha)}(\cos \theta)$ as $\ell \to \infty$:
    \begin{equation} \label{Eq: aux 12}
        \Big( \frac{\sin \theta}{2} \Big)^{\alpha} C_{\ell}^{(\alpha)}(\cos \theta) = \frac{2^{2\ell+2\alpha}}{\pi} \frac{\Gamma(\alpha + \frac{1}{2})}{\Gamma(2\alpha)} \frac{ \Gamma(\ell + 2\alpha) \Gamma(\ell + \alpha + \frac{1}{2})}{ \Gamma(2(\ell+\alpha) + 1) } \Big( \cos\big( (\ell+\alpha) (\theta - \pi) \big) + \BigO(\ell^{-1}) \Big)
    \end{equation}
    We may invoke now \cite[(5.11.12)]{NISTHandbook2010},
    \begin{equation} \label{Eq: asymp quotient Gamma function}
        \frac{\Gamma(x + a)}{\Gamma(x + b)} = x^{a-b} \big( 1 + \BigO(x^{-1}) \big) \qquad \text{as $x \to \infty$} \ ,
    \end{equation}
    to get
    \begin{equation} \label{Eq: aux 13}
        \begin{aligned}
            \frac{ \Gamma(\ell + 2\alpha) \Gamma(\ell + \alpha + \frac{1}{2})}{ \Gamma(2(\ell+\alpha) + 1) } & = \Big[ \Gamma(\ell) \ell^{2\alpha} \Gamma(\ell) \ell^{\alpha + \frac{1}{2}} \Big] \frac{(2\ell)^{-2\alpha - 1} }{ \Gamma(2\ell) } \big(1 + \BigO(\ell^{-1}) \big) \\
            & = \frac{\ell^{\alpha - \frac{1}{2}}}{2^{2\alpha + 1}} \frac{\Gamma(\ell)^2}{\Gamma(2\ell)} \big(1 + \BigO(\ell^{-1}) \big) \qquad \text{as $\ell \to \infty$} \ .
        \end{aligned}
    \end{equation}
    Applying Stirling approximation \cite[(5.11.3)]{NISTHandbook2010},
    \begin{equation} \label{Eq: Stirling approx}
        \Gamma(x) = e^{-x} x^x \Big( \frac{2\pi}{x} \Big)^{1/2} \big(1 + \BigO(x^{-1}) \big) \qquad \text{as $x \to \infty$} \ ,
    \end{equation}
    on \eqref{Eq: aux 13}, we have
    \begin{equation} \label{Eq: aux 14}
        \begin{aligned}
            \frac{ \Gamma(\ell + 2\alpha) \Gamma(\ell + \alpha + \frac{1}{2})}{ \Gamma(2(\ell+\alpha) + 1) } & = \frac{\ell^{\alpha - \frac{1}{2}}}{2^{2\alpha + 1}} \frac{(e^{-\ell} \ell^{\ell})^2 \, \frac{2\pi}{\ell} }{ e^{-2\ell} (2\ell)^{2\ell} (\frac{2\pi}{2\ell})^{1/2} } \big(1 + \BigO(\ell^{-1}) \big) \\
            & = \sqrt{\pi} \frac{\ell^{\alpha - 1}}{2^{2\alpha + 2\ell}} \big(1 + \BigO(\ell^{-1}) \big) \qquad \text{as $\ell \to \infty$} \ .
        \end{aligned}
    \end{equation}
    Combining \eqref{Eq: aux 12} and \eqref{Eq: aux 14}, we get
    \begin{equation*}
        \Big( \frac{\sin \theta}{2} \Big)^{\alpha} C_{\ell}^{(\alpha)}(\cos \theta) = \frac{\Gamma(\alpha + \frac{1}{2})}{ \sqrt{\pi} \Gamma(2\alpha) } \ell^{\alpha - 1} \Big( \cos[(\ell + \alpha) (\theta - \pi) ] + \BigO(\ell^{-1}) \Big) \qquad \text{as $\ell \to \infty$} \ .
    \end{equation*}
    Lastly, using that $\Gamma(\frac{1}{2}) = \sqrt{\pi}$ and the definition of the beta function, we finally arrive to
    \[
    C_{\ell}^{(\alpha)}(\cos \theta) = \ell^{\alpha - 1} \Big( B(\alpha, \tfrac{1}{2}) \frac{2^{\alpha} \Gamma(\alpha)}{\Gamma(2\alpha)} (\sin \theta)^{-\alpha} \cos [(\ell+\alpha)(\theta - \pi)] + \BigO(\ell^{-1}) \Big) \qquad \text{as $\ell \to \infty$} \ . \qedhere
    \]
\end{proof}

\section*{Declarations}

The author has no relevant financial or non-financial interests to disclose.

Data sharing is not applicable to this article, as no datasets were generated or analyzed during the current study.

\printbibliography

\end{document}

%% file: Repository/packages.tex
\usepackage[utf8]{inputenc}
\usepackage[T1]{fontenc}
\usepackage{geometry}
\usepackage[sorting = nyt, sortcites = true, giveninits = true, isbn = false]{biblatex}                          
\AtEveryBibitem{\clearfield{number}\clearfield{pages}}
\renewbibmacro{in:}{}

\usepackage{amssymb}
\usepackage{mathrsfs} 
\usepackage{mathtools} 
\usepackage{amsthm} 
\usepackage{thmtools} 
\usepackage{dsfont}

\usepackage[dvipsnames]{xcolor}
\usepackage[shortlabels]{enumitem} 
\usepackage[colorlinks = true]{hyperref}

\theoremstyle{plain}
\declaretheorem[parent = section]{Theorem}
\declaretheorem[sibling = Theorem]{Lemma}
\declaretheorem[sibling = Theorem]{Corollary}
\declaretheorem[name = Proposition, sibling = Theorem]{Prop}

\theoremstyle{definition}

\declaretheorem[sibling = Theorem]{Remark}
\declaretheorem[numbered = no]{Claim}

\numberwithin{equation}{section} 



%% file: Repository/commands.tex
\def \eps{ \varepsilon }
\DeclareDocumentCommand{\R}{O{} O{}}{ \mathbb{R}_{#2}^{#1} }
\DeclareDocumentCommand{\C}{O{} O{}}{ \mathbb{C}_{#2}^{#1} }
\DeclareDocumentCommand{\S}{O{} O{}}{ \mathbb{S}_{#2}^{#1} }
\DeclareDocumentCommand{\T}{O{} O{}}{ \mathbb{T}_{#2}^{#1} }
\DeclareDocumentCommand{\Z}{O{} O{}}{ \mathbb{Z}_{#2}^{#1} }
\def \N{\mathbb{N}}

\def \Cont{\mathcal{C}}
\def \Co{\mathcal{C}_{0} }

\def \Cinf{\mathcal{C}^{\infty}}
\def \CinfK{\mathcal{C}_{c}^{\infty}}
\def \Dist{\mathcal{D}'}

\def \weakstar{weak-$\star$ }
\def \charf{\mathds{1}}

\newcommand \littleo{
  \mathchoice
    {\scalebox{0.9}{$\scriptstyle\mathcal{O}$}}
    {\scalebox{0.9}{$\scriptstyle\mathcal{O}$}}
    {\scalebox{0.9}{$\scriptscriptstyle\mathcal{O}$}}
    {\scalebox{0.7}{$\scriptscriptstyle\mathcal{O}$}}
  }
\def \BigO{\mathcal{O}}
\newcommand \Haarmeasure{
  \mathchoice
    {\scalebox{1.2}{$\scriptstyle\mathbb{H}$}}
    {\scalebox{1.2}{$\scriptstyle\mathbb{H}$}}
    {\scalebox{1.2}{$\scriptscriptstyle\mathbb{H}$}}
    {\scalebox{0.9}{$\scriptscriptstyle\mathbb{H}$}}
  }

\DeclarePairedDelimiter{\abs}{\lvert}{\rvert}

\NewDocumentCommand{\norm}{s O{} m O{}}
    {
        \IfBooleanTF{#1}
            {\ensuremath{{\left\lVert #3 \right\rVert}_{#4}}}
            {{#2\lVert {#3} #2\rVert}_{#4}}
    } 

\NewDocumentCommand{\ip}{s O{} m m O{}}
    {
        \IfBooleanTF{#1}
            {\ensuremath{{\left\langle #3 \, \middle| \, #4 \right\rangle}_{#5}}}
            {{#2\langle #3 \, #2| \, #4 #2\rangle}_{#5}}
    }

\newcommand*{\conj}[1]{ \mkern 2mu \overline{\mkern-1mu #1 \mkern-1mu} \mkern 2mu}
\newcommand*{\D}[1]{\mathop{\mathrm{d}#1}}

\DeclareMathOperator{\spn}{Span}

\NewDocumentCommand{\Span}{O{} O{} m}{\spn_{#1} #2\{ #3 #2\}}

\let\Re\relax
\let\Im\relax
\DeclareMathOperator{\Re}{\mathfrak{Re}}
\DeclareMathOperator{\Im}{\mathfrak{Im}}
\DeclareMathOperator{\dom}{dom}

\DeclareMathOperator{\Lap}{\Delta}

\DeclareMathOperator{\supp}{supp} 
\DeclareMathOperator{\Opp}{Op}
\DeclareMathOperator{\vol}{vol}
\DeclareMathOperator{\Spec}{Spec}
\DeclareMathOperator{\card}{\#}

\NewDocumentCommand{\Op}{O{} O{} m}{\ensuremath{\Opp_{#1}^{#2} \left( #3 \right)}}